\documentclass[10pt,a4paper, reqno]{amsart}
\usepackage{amsmath}
\usepackage{amssymb}
\usepackage{amscd}
\usepackage{mathrsfs}
\usepackage{txfonts}
\usepackage{hyperref}

\newtheorem*{theorem*}{Theorem}
\newtheorem{lemma}{Lemma}[section]
\numberwithin{equation}{section}
\newtheorem{proposition}[lemma]{Proposition}
\newtheorem{remark}[lemma]{Remark}

\newtheorem{theorem}[lemma]{Theorem}

\newtheorem{corollary}[lemma]{Corollary}

\baselineskip 18pt \textwidth 16cm  \theoremstyle{plain}

\newcommand{\End}{\operatorname{End}}
\newcommand{\Hom}{\operatorname{Hom}}

\newcommand{\Bc}{\operatorname{Bc}}

\newcommand{\Irr}{\operatorname{Irr}}

\newcommand{\Ind}{\operatorname{Ind}}
\newcommand{\Rep}{\operatorname{Rep}}
\newcommand{\Res}{\operatorname{Res}}
\newcommand{\res}{\operatorname{res}}
\newcommand{\St}{\operatorname{St}}

\newcommand{\C}{\mathbb C}
\newcommand{\F}{\mathbb F}

\newcommand{\Gal}{\operatorname{Gal}}
\newcommand{\GL}{\operatorname{GL}}
\newcommand{\GSp}{\operatorname{GSp}}
\newcommand{\Mp}{\operatorname{Mp}}
\newcommand{\SL}{\operatorname{SL}}

\newcommand{\Sp}{\operatorname{Sp}}
\newcommand{\SU}{\operatorname{SU}}
\newcommand{\U}{\operatorname{U}}

\newcommand{\GO}{\operatorname{GO}}
\newcommand{\Ha}{\operatorname{H}}
\newcommand{\Orbit}{\operatorname{Orbit}}

\newcommand{\diag}{\operatorname{diag}}

\newcommand{\id}{\operatorname{Id}}
\newcommand{\nnn}{\operatorname{N}}

\newcommand{\stab}{\operatorname{Stab}}
\newcommand{\Stab}{\operatorname{Stab}}
\newcommand{\supp}{\operatorname{supp}}

\newcommand{\tr}{\operatorname{Tr}}
\newcommand{\Tr}{\operatorname{Tr}}
\begin{document}

\title{On a question of Drinfeld on the Weil representation: the finite field case}
\date{September  2013}

\author{Chun-Hui Wang }
\address{NCMIS, Academy of Mathematics and Systems Science\\
Chinese Academy of Sciences\\
Beijing, 100190, P.R. China}
\email{cwang@amss.ac.cn}


\maketitle
\setcounter{tocdepth}{5}
\begin{abstract}
Let $F$ be a finite field of odd cardinality, and let $G= \GL_2(F)$. The group $G \times G \times G$ acts on $F^2 \otimes F^2 \otimes F^2$ via symplectic similitudes, and has a natural Weil representation. Answering a question raised by V. Drinfeld, we decompose that representation into irreducibles. We also decompose the analogous representation of $\GL_2(A)$, where $A$ is a cubic algebra over $F$.
\end{abstract}
\section*{Introduction}
Let $F$ be a finite field of  \emph{odd} cardinality $q$, and let $W$, $\langle, \rangle$  be a symplectic vector space over $F$ of dimension $2n$. The Heisenberg group $\Ha(W)$, attached to $W$ and $F$, is a set $W\oplus F$ with the group law: $(w,t) (w',t')=(w+w', t+t' + \frac{\langle w, w'\rangle}{2}).$
Let $\Sp(W)$ be the isometry group of $(W, \langle, \rangle)$. Define a semi-direct product group $\Ha(W) \rtimes \Sp(W)$ by
$[(w,t), g ]\cdot [(w',t'), g' ]= [(w ,t)+ (g \cdot w',t'), gg'].$
Fix  a non-trivial character $\psi$ of $F$. According to the Stone-Von Neumann theorem, there is only one equivalence class of irreducible complex representation $\omega_{\psi}$ of $\Ha(W)$ with central character $\psi$. By Weil's celebrated paper \cite{Weil}, in fact $\omega_{\psi}$ is a representation of $\Ha(W)\rtimes\Sp(W)$ in the finite field case. The restriction of $\omega_{\psi}$ to $\Sp(W)$, now is well-known as the \emph{Weil representation}; in \cite{Gero}, G\'erardin investigated fully this representation. Following Shinoda \cite{Shin},   we extend it to the  symplectic similitude group $\GSp(W)$ by setting $\rho=\Ind_{\Sp(W)}^{\GSp(W)} \omega_{\psi}$, which  does not depend on the choice of $\psi$ (\cite[p. 270, Theorem]{Shin}). \\

The initial question raised by V.Drinfeld, in the finite field case, is understood roughly  in the following way.     Let $F^2$, $\langle, \rangle$  be  a   symplectic space over $F$ of dimension $2$. Consider now  $W=F^2 \otimes F^2 \otimes F^2$,  a symplectic vector space over $F$ of dimension $8$  endowed  with the symplectic form $\langle, \rangle_{F^2}\otimes \langle, \rangle_{F^2} \otimes \langle, \rangle_{F^2}$. So there is  a map from $\GL_2(F) \times \GL_2(F) \times \GL_2(F)$ to $\GSp(W)$. In this way we can define a Weil representation $\pi$ for $\GL_2(F)\times \GL_2(F) \times \GL_2(F)$ via the restriction of $\rho$, where   $\rho$ is  the Weil representation of $\GSp(W)$.  The question is asked about the set of  the quotients of $\pi$. Does it contain the representations of the form $\sigma\otimes \sigma \otimes \sigma$ for any irreducible representation $\sigma$ of $\GL_2(F)$? In this paper, we answer this question and also consider its variant version. To be  precise, suppose now that $E/F$ (resp. $K/F$) is  a  field extension of degree 2 (resp. 3). Take $A$ to be an \'etale algebra over $F$ of degree $3$, so $A$ is isomorphic to one of the algebras $F \times F \times F, F \times E, K$.   We shall construct a homomorphism from $\GL_2(A)$ to $\GSp_8(F)$.
 If $A=F \times F \times F$, then  $\GL_2(A)\simeq \GL_2(F) \times \GL_2(F) \times \GL_2(F)$. This goes back to Drinfeld's question.  If $A=F \times E$, then $\GL_2(A)\simeq \GL_2(F) \times \GL_2(E) $. By Weil's Galois descent, we construct a quadratic vector space $M=\{\begin{pmatrix}
x  &  \alpha \\
 \overline{\alpha} & y
\end{pmatrix} | x,y \in F, \alpha \in E\}$ over $F$ of dimension $4$, with the quadratic form $Q$ defined by the determinant of the  matrix. Clearly there is a map from $\GL_2(E)$ to $\GO(Q)$, which is defined by $h\cdot m=hm\overline{h}^t$, where   $h\in \GL_2(E), m\in M$ and $\overline{h}^t$ is the conjugate transpose  of $h$. So $F^2 \otimes M$ is a symplectic vector space over $F$ of dimension $8$, and there is a map from $\GL_2(F) \times \GL_2(E)$ to $\GSp_8(F)$. If $A=K$, in this situation, we also need to use  Weil's Galois descent to construct a map from $\GL_2(K)$ to  $\GSp_8(F)$.
 The map from $\GL_2(A)$ to $\GSp_8(F)$ leads us to define a representation $\pi_A$ of $\GL_2(A)$ via the restriction of $\rho$. The main purpose of this paper  is  to obtain the complete decomposition of  $\pi_A$ in each case.  \\

 For the group $G=\GL_2(F)$, we write   $1_G$ for the trivial representation of $G$, and $\St_G$ for the Steinberg representation of $G$. Let  $T=\{\begin{pmatrix}
  a& 0\\
  0&d
\end{pmatrix} \in G \}$,  $B=\{\begin{pmatrix}
  a& b\\
  0&d
\end{pmatrix} \in G \}$.   Let $\chi_1\otimes \chi_2$ be the character of  $T$ defined by $\begin{pmatrix}
  a& 0\\
  0&d
\end{pmatrix}  \longmapsto \chi_1(a) \chi_2(d)$ for  two characters $\chi_1, \chi_2$ of $F^{\times}$. We will denote  the principal series representation  $\Ind_{B}^G (\chi_1 \otimes \chi_2)$ of $G$ by $\pi_{\chi_1, \chi_2}$.
 If $(\sigma, V)$ is a representation of $G$ and $\psi$  a character of $F^{\times}$, we write  the $\psi\cdot \sigma$ for the representation
 $\psi\cdot \sigma(g)=\psi(\det\,g) \sigma(g).$
  Let $\Irr(G)$ denote  the class of all irreducible complex representations of the group $G$.  Let $L$ be a field extension of $F$. By  Shintani's work \cite{Shint}, one knows that  there exists the base-change map $\Bc_{L/F}: \Irr(\GL_2(F)) \longrightarrow  \Irr(\GL_2(L))$,  which is determined by  character equalities.  Our main results may be formulated as follows:

\begin{theorem*}[1]
If $A=F \times F \times F$, $\GL_2(A)\simeq \GL_2(F) \times \GL_2(F) \times \GL_2(F)$, then
$$ \pi_A\simeq \bigoplus_{\sigma \in \Irr(\GL_2(F))} \sigma \otimes \sigma\otimes \sigma
\oplus \bigoplus_{\psi\in \Irr(F^{\times})}\Bigg(
(\psi\!\cdot\!\St_{\GL_2(F)}\!\otimes\!\psi\!\cdot\!1_{\GL_2(F)}\!\otimes\!\psi\!\cdot\!1_{\GL_2(F)})
\,\oplus\, (\psi\!\cdot\!1_{\GL_2(F)}\!\otimes\!\psi\!\cdot\!\St_{\GL_2(F)}\!\otimes\!\psi\!\cdot\!1_{\GL_2(F)})
\,\oplus\, (\psi\!\cdot\!1_{\GL_2(F)} \!\otimes\!\psi\!\cdot\!1_{\GL_2(F)}\!\otimes\!\psi\!\cdot\!\St_{\GL_2(F)})\Bigg).$$
\end{theorem*}

\begin{theorem*}[2]
If $A=F \times E$, $\GL_2(A) \simeq \GL_2(F) \times \GL_2(E)$, then
$$ \pi_A\simeq\bigoplus_{\sigma \in \Irr(\GL_2(F))} \sigma \otimes \Bc_{E/F}(\sigma)\oplus \bigoplus_{\psi\in \Irr(F^{\times}), \Psi\in \Irr(E^{\times}),\Psi=\psi\circ\nnn_{E/F}} \Big( \psi\!\St_{\GL_2(F)}\!\otimes\Psi\!\cdot\!1_{\GL_2(E)}\Big).$$
\end{theorem*}
\begin{theorem*}[3]
If $A=K$, $\GL_2(A)\simeq\GL_2(K)$, then
\begin{displaymath}
\pi_A\simeq\bigoplus_{\sigma \in \Irr(\GL_2(F))} \Bc_{K/F}(\sigma).
\end{displaymath}
\end{theorem*}

Let us briefly review the whole story. Theorems (1) is obtained mainly by using the method in \cite{Andr} to decompose reducible representations. In \cite{Andr}, Andrade  considered   higher rank groups.   We  first  formulate   the representation $\pi_A$ of $\GL_2(F)\times \GL_2(F) \times \GL_2(F)$ concerned in this case. This can be done by following  works of  G\'erardin  and of  Shinoda on the Weil representations(cf. \cite{Gero}, \cite{Shin}).  Then we take two irreducible representations $\pi_1$, $\pi_2$ of $\GL_2(F)$, and determine the dimension of $\Hom_{1\times \GL_2(F)\times \GL_2(F)} \Big( \pi_A, \pi_1\otimes \pi_2\Big)$. One key ingredient is that   $\pi_A$ is  in fact  a representation of the group $\Bigg(\GL_2(F)\times \GL_2(F) \times \GL_2(F) \Bigg) \rtimes S_3$, where $S_3$ is the permutation group of $3$ variables. So if we put  $\mathcal{R}\Big( \pi_A\Big)=\{ \pi_1 \otimes \pi_2 \otimes \pi_3| \Hom_{\GL_2(F)\times \GL_2(F)\times \GL_2(F)}\Big( \pi_A, \pi_1 \otimes \pi_2 \otimes \pi_3\Big) \neq 0\}$,  by Clifford theory, $\mathcal{R}\Big( \pi_A\Big)$ is $S_3$-invariant. This together with the above calculations of dimension    derives  Theorem (1). For  Theorem (2),   following the  method in \cite{Andr},  we first write down the  Weil representation $\pi_A$ of $\GL_2(F)\times \GL_2(E)$ in this case,   and then   decompose  the canonical representation    \Big($\GL_2(F)$, $\Hom_{\GL_2(E)}\Big( \pi_A, \Pi\Big)$\Big) into irreducibles for  each $\Pi\in \Irr(\GL_2(E))$.  We did this by  checking the irreducible representations of $\GL_2(E)$ one by one. The main  difficulty is when $\Pi$ is cuspidal. For that case, we use the explicit models given by \cite{Andr}. The \'etale algebra $A=K$ is a new case.  We use Weil's Galois  descent to construct a map $i$ from $\GL_2(K)$ to $\GSp_8(F)$. Through this map, we shall define a new Weil representation $\pi_A$ for the group $\GL_2(K)$.  However the explicit realisation of this representation   is somehow complex, this causes the difficulty to study its irreducible components. One point is that the map $i$ sends the standard Borel subgroup of $\GL_2(K)$ to that of $\GSp_8(F)$. By  virtue of Frobenius reciprocity,  we obtain the results for the principal series representations.  For the cuspidal representations,  we use  a technique so-called  ``base change'' to reduce  to deal with  some  principal series representations. We should mention that this technique has  been used in  Gan's paper  \cite{Gan} to obtain Howe correspondences for exceptional groups.

Another approach to the  results of this paper maybe use  character theory in representations and  it  sometimes  involves to solve certain  equations. In practice,  giving  such equations in some sense is also complex.

The structure of this paper is as follows. The first section is devoted to giving  some notations and  recalling some known results. In the second section, we consider the \'etale algebra $A=F \times F \times F$. In the third section, we deal with the case   $A=F\times E$. In the fourth  section,   we consider the case $A=K$; there we  put some calculations in two appendices.

\emph{Acknowledgements.}
The paper is one part of the author's Ph.D. thesis and completes in the Academy  of Mathematics  and Systems Science. He would like to thank his advisor  Guy Henniart for   useful comments.
\section{Notation and Preliminaries}\label{Prel}
\subsection{} The following notations will be standard through the whole paragraph, and  used repeatedly without recalling their meanings:

\begin{itemize}
\item $F$ = a finite field with odd cardinality $q$;
\item $E$ = a fixed field extension of $F$ of degree 2;
\item $K$ = a fixed field extension of $F$ of degree 3;
\item $\phi$ = a fixed non-trivial character of the additive group  $F$;
\item $\phi^a$ = the character of $F$, defined by $\phi^a(b):=\phi(ab)$ for $b\in F$, $a\in F^{\times}$;
\item $X_A$ = the set of all non-trivial irreducible complex representations of an abelian group $A$;
\item $\Rep(G)$ = the category of complex representations of a finite group $G$;
\item $\Irr(G)$ = the class of all irreducible complex representations of a  finite group $G$, up to isomorphism;
\item $\check{\sigma}$ = the contragredient representation of $\sigma$, for $\sigma \in \Rep(G)$;
\item  If $(\sigma, V)$ is a representation of $G$, then  we will denote its $G$-invariant set by $ V^G$.
\end{itemize}

\subsection{} For later use,  we regroup  some results of the  Weil representation of $\GSp_{2n}(F)$ (cf. \cite{Gero},  \cite{MVW}, \cite{Shin}).

 Let $V$ be a 2n-dimensional $F$-vector space, endowed with a non-degenerate symplectic form $\langle \  \rangle $.
 To each non-trivial character $ \psi $ of the additive group $F$, one can associate the Weil representation $ (\omega_{\psi}, W_{\psi})$ of the metaplectic group
 $\Mp_{2n}(F)$ (cf. \cite[Chapter 2]{MVW}). The exact sequence $$  1 \longrightarrow \C^{\times }
 \longrightarrow \Mp_{2n}(F) \stackrel{p}{\longrightarrow} \Sp_{2n}(F) \longrightarrow 1$$ is splitting. Except $n=1$, $F=\F_3$, the group $\Sp_{2n}(F)$ is perfect, so there exists a unique section of morphism $i$ from $\Sp_{2n}(F)$ to $\Mp_{2n}(F)$, such that $p\circ i=\id_{\Sp_{2n}(F)}$.  In the case $n=1$, $F=\F_3$, we  choose a certain section $i$ in the sense of G\'erardin (cf. \cite[p. 63, Theorem 2.4 (a'')]{Gero}). Via the map $i$,  one  obtains a  representation $(\omega_{\psi},  W_{\psi})$   of $\Sp_{2n}(F)$ with respect to $\psi$, called the \emph{Weil representation}.
 One can extend it as a representation of $\GSp_{2n}(F)$ by setting $\rho_{\psi}=\Ind_{\Sp_{2n}(F)}^{\GSp_{2n}(F)} \omega_{\psi}$.
 It is observed that $\rho_{\psi}$ is independent of $\psi$ (see \cite[p. 270, Theorem]{Shin}). Hence we could omit $\psi$, and  only  write $\rho$  briefly.

 The study of the Weil representation often involves an explicit realized model.
  We recall one so-called `` the Schr\"odinger model'':  Let $V=V_+ \oplus V_-$ be a complete polarization. Let $\{ v_1, \ldots, v_n\}$ be a F-basis of $V_+$, and $\{ v'_1, \ldots , v'_n\}$  its dual basis with respect to $\langle, \rangle$.
 Every element $g \in \GSp(V)$  can be written in the following form:
 $ g= \begin{pmatrix}
 \alpha  & \beta \\
  \gamma & \delta
\end{pmatrix}$ where $\alpha\in \End_F(V_+),$ $ \beta \in \Hom_F(V_-, V_+), $ $\gamma \in \Hom_F(V_+, V_-),$ $\delta \in  \End_F(V_-).$
The group $\GSp(V)$ is generated by the set  $\{h(a), u(b), h'(t), \omega \}$ (see \cite{Andr}, p. 163),  where
$h(a)= \begin{pmatrix}
a  & 0 \\
 0 & a^{\vee }
\end{pmatrix},
$ \ $a^{\vee }$ is the contragredient of $a$;
$u(b)= \begin{pmatrix}
 1 & b  \\
 0 & 1
\end{pmatrix}
$ $\quad$ for a  symmetric morphism $b\in \Hom_F(V_-, V_+)$;
$h'(t)= \begin{pmatrix}
  1& 0 \\
  0&t
\end{pmatrix}, t\in F^{\times}$;
$\omega= \begin{pmatrix}
  0& I \\
 -I & 0
\end{pmatrix}$ with $\omega(v_i)=-v'_i,\, \omega(v_i')=v_i$.\footnote{In \cite{Shin}, the $\omega$ is defined as $\omega(v_i)=v_i'$, $\omega(v_i')=-v_i$, but it does not affect the following equation (\ref{weilGSp3}). To obtain (\ref{weilGSp3}), we use the equality: $\gamma\big( \psi^{\frac{1}{2}}\big)^{-n} \chi_q^+(-1)^n=\gamma(\psi^{-\frac{1}{2}})^{-n}$.}  The Weil representation $\rho$ of $\GSp(V)$ can be realized in the space $W_-=\C[V_-\times X_F]$ of complex functions on $V_- \times X_F$.
More precisely the action of $\GSp(V)$ on $W_-$ is determined by the following formulas (cf. \cite[ p. 270]{Shin}):\\
\begin{equation}\label{weilGSp1}
 \big(\rho(h(a))F\big)(y, \psi)= \chi_q^+(det_{V_+}a)F(a^{\vee -1}y, \psi ),
\end{equation}
\begin{equation}\label{weilGSp2}
 \big(\rho(u(b))F\big)(y, \psi)= \psi(\frac{1}{2}\langle by, y \rangle )F(y,\psi),
\end{equation}
\begin{equation}\label{weilGSp3}
 \big(\rho(\omega)F\big)(y,\psi)= \gamma(\psi^{-\frac{1}{2}})^{-n} \sum_{z\in V_-} F(z, \psi)\psi(\langle z, \omega^{-1}(y) \rangle),
\end{equation}
\begin{equation}\label{weilGSp4}
 \big(\rho(h'(t))F\big)(y,\psi)= F(y, \psi^{t^{-1}}),
\end{equation}
where $y\in V_-, \psi \in X_F, \gamma(\psi)= \sum_{x\in F}\psi(x^2), x_q^{+}= \textrm{ Legendre symbol } ( \frac{ }{\mathbb{F}_q})$.\\

\subsection{} We summarize some facts about the irreducible representations of $\GL_2(F)$ and  its  Borel subgroup (cf. \cite[Chapter 2]{BushH}  and \cite{Shap}).

We write $G=\GL_2(F), B=\{\begin{pmatrix}
  a& b\\
  0&d
\end{pmatrix} \in G \}
, \ N= \{
 \begin{pmatrix}
  1& b \\
  0&1
\end{pmatrix} \in G \}, \ T=\{\begin{pmatrix}
  a& 0 \\
  0&d
\end{pmatrix} \in G \}, \ M=\{\begin{pmatrix}
  a& b \\
  0&1  \end{pmatrix} \in G  \}, \ Z=\{\begin{pmatrix}
  a& 0 \\
  0&a  \end{pmatrix} \in G  \}$. Recall that $1_G$ is the trivial representation of $G$, and $\St_G$ is the Steinberg representation of $G$. Let $\theta$  be a regular character of $E^{\times}$;  the irreducible cuspidal representation of $G$  corresponding  to $\theta$ will be denoted by $\pi_{\theta}$.  If $(\sigma, V)$ is a representation of $G$ and $\psi$  a character of $F^{\times}$, we define the representation $(\psi\cdot\sigma, V)$ of $G$ by  $\psi\cdot \sigma(g)=\psi(\det\,g) \sigma(g)$.

\begin{theorem}[{\cite[Chapter 2]{BushH}}]
The following is a complete list of the isomorphism classes of the  irreducible representations of $G$:
\begin{itemize}
\item[(1)] $\pi_{\chi_1, \chi_2},\quad$ where $\chi_1\neq  \chi_2$  are characters of $F^{\times}$;
\item[(2)] $\psi\!\cdot\!1_G,\quad$ where $\psi$ ranges over the characters of $F^{\times}$;
\item[(3)] $\psi\!\cdot\!\St_G, \quad$ where $\psi$ ranges over the characters of $F^{\times}$;
\item[(4)] $\pi_{\theta}, \qquad$ where $\theta$ ranges over the regular characters of $E^{\times}$.
\end{itemize}
The classes in the list all are distinct except that in (1) $ \pi_{\chi_1,\chi_2} \simeq \pi_{\chi_2,\chi_1}$, and in (4) $\pi_{\theta} \simeq \pi_{\theta^q}$.
\end{theorem}
\begin{lemma}[{\cite[Chapter 2]{BushH}}]
 Notations being in above Theorem, we then have $(\pi_{\chi_1, \chi_2})^{\vee} \simeq \pi_{\chi_1^{-1}, \chi_2^{-1}}$, $(\psi \cdot 1_G)^{\vee}\simeq \psi^{-1} \cdot 1_G$, $(\psi \cdot \St_G)^{\vee} \simeq \psi^{-1} \cdot \St_G$ and $(\pi_{\theta})^{\vee} \simeq \pi_{\theta^{-1}}$.
 \end{lemma}
Now we investigate  the representations of the group B.  Let $\sigma_{\chi_1, \chi_2}$ be the character of $B$, defined by $\sigma_{\chi_1, \chi_2} \Big(\begin{pmatrix}
  a& b \\
  0&d  \end{pmatrix} \Big)= \chi_1(a) \chi_2(d)$. Let $\sigma$ be the unique  irreducible representation of $M$ of the highest dimension and  $\psi$  a character of $F^{\times}$. Attached to $\sigma$ and $\psi$, there is an irreducible representation $\psi\!\otimes\!\sigma$ of $B$, defined by $(\psi\otimes \sigma)(zm):=\psi(z)\sigma(m)$ for $z\in Z, m\in M$.
\begin{theorem}[{\cite[Theorem 7.1]{Shap}}]\label{representationsofB}
The following is a complete list of the isomorphism classes of the irreducible representations of $B$:
\begin{itemize}
\item[(1)] $\sigma_{\chi_1, \chi_2} \quad$ for any pair $(\chi_1, \chi_2)$ of characters of $F^{\times}$;
\item[(2)] $\psi\!\otimes\!\sigma \quad$ for any character $\psi$ of $Z$.
\end{itemize}
\end{theorem}

For convenience use, we describe the decomposition of the restriction to $B$ of any irreducible   representation of $G$.
\begin{proposition}\label{Rest1}
\begin{itemize}
\item[(1)] $\Res_B^G(\psi\!\cdot\!1_G) = \sigma_{\psi, \psi}$.
\item[(2)] $\Res_B^G(\psi \cdot \St_G) = \big(\sigma_{\psi, \psi}\big)\, \oplus \, \big( \psi^2\otimes \sigma\big)$.
\item[(3)] $\Res_B^G\pi_{\chi_1, \chi_2} = \big(\sigma_{\chi_1, \chi_2}\big)\,\oplus \,\big(\sigma_{\chi_2, \chi_1}\big) \,\oplus \,\big(\chi_1\chi_2 \otimes\sigma\big) $.
\item[(4)] $\Res_B^G\pi_{\theta}= (\theta|_{F^{\times}}) \otimes \sigma$.
\end{itemize}
\end{proposition}
\begin{proof}[Proof]
See the table in \cite[p. 87]{Andr}.
\end{proof}

\subsection{} Let $L$ be the Galois field extension of $F$ of degree $n$. One knows that there exists the base-change map $\Bc_{L/F}: \Irr(\GL_2(F)) \longrightarrow  \Irr(\GL_2(L))$ (cf. \cite{Shint}). Now we describe the explicit behaviour of this map in terms of the classification of the irreducible representations of the group $\GL_2$ in the cases $n=2, 3$.
\begin{theorem}\label{basechangeGL2}
(1) If $[L:F]=2$,  then
\begin{itemize}
\item[(i)] $\Bc_{L/F} (\pi_{\xi_1, \xi_2})=\Pi_{\Xi_1, \Xi_2}  \, \textrm{ where } \, \Xi_i=\xi_i\circ \nnn_{L/F}$ as characters of $L^{\times}$, for $i=1,2$;
\item[(ii)] $ \Bc_{L/F} (\psi\cdot 1_{\GL_2(F)}) =\Psi\cdot 1_{\GL_2(L)}  \, \textrm{ where } \, \Psi=\psi\circ \nnn_{L/F}$ as  characters of $L^{\times}$;
\item[(iii)] $ \Bc_{L/F} (\psi\cdot \St_{\GL_2(F)}) =\Psi\cdot \St_{\GL_2(L)}  \, \textrm{ where } \, \Psi=\psi\circ \nnn_{L/F}$ as  characters of $L^{\times}$;
\item [(iv)] $ \Bc_{L/F} (\pi_{\theta}) =\Pi_{\theta, \theta^q} $.
\end{itemize}
(2)If $[L:F]=3$, then
\begin{itemize}
\item[(i)] $ \Bc_{L/F} (\pi_{\xi_1, \xi_2})=\Pi_{\Xi_1, \Xi_2}  \, \textrm{ where } \, \Xi_i=\xi_i\circ \nnn_{L/F}$ as  characters of $L^{\times}$,for $i=1,2$;
\item[ (ii)] $ \Bc_{L/F} (\psi\cdot 1_{\GL_2(F)}) =\Psi\cdot 1_{\GL_2(L)}  \, \textrm{ where } \, \Psi=\psi\circ \nnn_{L/F}$ as  characters of $L^{\times}$;
\item [ (iii)] $ \Bc_{L/F} (\psi\cdot \St_{\GL_2(F)}) =\Psi\cdot \St_{\GL_2(L)}  \, \textrm{ where } \, \Psi=\psi\circ \nnn_{L/F}$ as  characters of $L^{\times}$;
\item [(iv)] $ \Bc_{L/F} (\pi_{\theta}) =\Pi_{\Theta} $  where $[F_1:F]=2, [L_1:L]=2, L_1\supseteq F_1, \theta \in \Irr(F_1^{\times})-\Irr(F^{\times}), \Theta\in\Irr(L_1^{\times})-\Irr(L^{\times}),\,\textrm{ and } \Theta=\theta\circ \nnn_{L_1/F_1}$ as characters of $L_1^{\times}$.
\end{itemize}
\end{theorem}
\begin{proof}
See \cite[Section 4, p. 410---414]{Shint}.
\end{proof}
\subsection{}\label{shintanilift} As is known that one can generalize the above base-change map  defined for  other groups, called \emph{Shintani lifting} or \emph{Shintani descent} (e.g. \cite{DignM2}).  In the article \cite{Gyoja}, Gyoja studied   systematically   Shintani lifting for connected linear algebraic groups.   We will recall his certain results below. In addition, we also present one main result in \cite{HW} about the behaviour of   the Weil representations with respect to Shintani lift.

Let $\overline{F}$ be a fixed algebraic closure of $F$ with Frobenius map $\sigma$. Let $\textbf{G}$ be a connected linear algebraic group over $F$. Denote by $F_i$ the $\sigma^i$-fixed points of $\overline{F}$. Let $Y$ be a set on which there exists a $\sigma$-action; we denote  the set of $\sigma$-fixed points by $Y_{\sigma}$. Denote by
 $\textbf{G}(F_i)$ the $F_i$-geometric points of $\textbf{G}$ and $\mathcal {C}(\textbf{G}(F_i))$
 the set of complex valued class functions of $\textbf{G}(F_i)$. Fix a positive integer $m$.  Via the map $\Gal(\overline{F}/F) \twoheadrightarrow  \Gal(F_m/F)$, we view the Frobenius element $\sigma$ as one generator for the group $\Gal(F_m/F)$.  For $0 \leq i \leq m-1$, let us denote by $\textbf{G}(F_m)\rtimes \sigma^i $,  the subset of the semi-direct product $\textbf{G}(F_m)\rtimes \Gal(F_m/F)$  consisting of  $(g, \sigma^i)$'s for $g\in \textbf{G}(F_m)$. In the article \cite{Gyoja}, following Kawanaka \cite{Kawa}, Gyoja defined the norm maps $\nnn_i$ as follows:
$$N_i: \textbf{G}(F_m)\rtimes \sigma^i \longrightarrow \textbf{G}(\overline{F});$$
$$[ x, \sigma^i] \longmapsto  \alpha(x)\Big(x  \sigma^i(x)\cdots \sigma^{i\big( \frac{m}{d}-1\big)}(x)\Big) \alpha(x)^{-1}, $$
where $\alpha(x)$ is an element in $\textbf{G}(\overline{F})$ such that $ \alpha(x)^{-1}\sigma^d \Big(\alpha(x)\Big)=x\sigma^i(x)\cdots \sigma^{i(t-1)}(x)$ and $d, \, t$ are the integers given by $d=(m,i)$  and $ti\equiv d (\!\mod \, m)$. Here $(m,i)$  denotes the greatest common divisor of $m$ and $i$.

Each  norm map  $\nnn_i$ induces a bijection from the set of $\textbf{G}(F_m)$-conjugacy
classes of $\textbf{G}(F_m)\rtimes\sigma^i$ onto the set of conjugacy classes of $\textbf{G}(F_m)_{\sigma^i}=\textbf{G}(F_{(m,i)})$.
Through $\nnn_i$, one  defines the $i$-restriction map from $\mathcal {C}(\textbf{G}(F_m)\rtimes\Gal(F_m/F))$ to
$\mathcal{C}(\textbf{G}(F_{(m,i)}))_{\sigma}$ such that $\big(i\!-\!res (f)\big)\circ \nnn_i= f|_{ \textbf{G}(F_m)\rtimes\sigma^i}$ for any
$f\in \mathcal {C}(\textbf{G}(F_m)\rtimes\Gal(F_m/F))$.

\begin{lemma}\label{basenorm}
(i) For any $f, g\in \mathcal {C}\big(\textbf{G}(F_m)_{\sigma^i}\big)$, we have $\langle f, g\rangle_{\textbf{G}(F_m)_{\sigma^i}}= \langle f\circ N_i, g\circ N_i\rangle_{\textbf{G}(F_m)\rtimes\sigma^i}$, where $\langle f, g\rangle_{\textbf{G}(F_{(m,i)})}:= \frac{1}{|\textbf{G}(F_{(m,i)})|} \sum_{x\in \textbf{G}(F_{(m,i)})} f(x) \overline{g(x)}$ and $\langle f\circ N_i, g\circ N_i\rangle_{\textbf{G}(F_m)\rtimes\sigma^i}:= \frac{1}{|\textbf{G}(F_m)\rtimes\sigma^i |} \sum_{t\in \textbf{G}(F_{m})} f\big( N_i(\sigma^i, t)\big) \overline{g\big(N_i(\sigma^i, t)\big)}$. \\
(ii) The $i$-restrictions define an isomorphism: $\mathcal {C}\big(\textbf{G}(F_m)\rtimes\Gal(F_m/F)\big) \simeq \oplus_{i=0}^{m-1} \mathcal{C}\big(\textbf{G}(F_m)_{\sigma^i}\big)_{\sigma}$.
\end{lemma}
\begin{proof}
See \cite[ p.11, Corollary 3.3 and p.1, Introduction]{Gyoja}.
\end{proof}
\begin{lemma}\label{Indresnorm}
Let $\textbf{H}$ be a connected closed subgroup of $\textbf{G}$ defined over $F$. Then the following diagram is commutative
\[
\begin{CD}
\mathcal{C}\big(\textbf{G}(F_m)\rtimes\Gal(F_m/F)\big) @>\Res>> \mathcal{C}\big(\textbf{H}(F_m)\rtimes\Gal(F_m/F)\big)\\
@VVi-resV   @VVi-resV\\
\mathcal{C}\big( \textbf{G}(F_{(m,i)})\big)_{\sigma} @>\Res>> \mathcal{C}\big( \textbf{H}(F_{(m,i)})\big)_{\sigma}
\end{CD}
\]
\end{lemma}
\begin{proof}
See \cite[p. 12, Lemma 3.6]{Gyoja} .
\end{proof}
Now let $V, \langle, \rangle$ be a symplectic space over $F$ and let $\textbf{G}=\textbf{GSp}_V$  be the algebraic group of symplectic  similitudes of  $(V, \langle, \rangle)$.  For $0 \leq i \leq m-1$,  write  $\Xi_{F_{(m,i)}}$ for the Weil representation of $\textbf{G}(F_{(m,i)})$
\begin{proposition}\label{liftingGSp}
 There exists a unique representation $\widetilde{\Xi_{F_m}}$ of  $ \textbf{G}(F_m)\rtimes\Gal(F_m/F)$ such that $i$-res($\widetilde{\Xi_{F_m}}$)$=\Xi_{F_{(m,i)}}$ for $0 \leq i \leq m-1$.
\end{proposition}
\begin{proof}
See  \cite[Theorem 4.2]{HW}.
\end{proof}
\section{The decomposition of the Weil representation of $\GL_2(F)\!\times\!\GL_2(F)\!\times\!\GL_2(F)$ }\label{demp1}

\subsection{} We give some notations and formulate  the mainly studied  representation in this section.

 In this section, we use the following notations: $G=\GL_2(F)$, $H=G \times G$, $ B=\{\begin{pmatrix}
  a& b\\
  0&d
\end{pmatrix} \in G \}
, \ N= \{
 \begin{pmatrix}
  1& b \\
  0&1
\end{pmatrix} \in G \}, \ T=\{\begin{pmatrix}
  a& 0 \\
  0&d
\end{pmatrix} \in G\},  \  Z=\{\begin{pmatrix}
  a& 0 \\
  0&a  \end{pmatrix} \in G \}$; $h(r)=\begin{pmatrix}
  r& 0 \\
  0&r^{-1}
\end{pmatrix}$, $u(b)=\begin{pmatrix}
  1& b \\
  0&1
\end{pmatrix}$, $h'(t)=\begin{pmatrix}
  1& 0 \\
  0&t
\end{pmatrix}$, $\omega'=\begin{pmatrix}
  0& 1 \\
  1&0
\end{pmatrix}$, $\omega=\begin{pmatrix}
  0& 1 \\
  -1&0
\end{pmatrix}$  in $G$, $S_3=$ the permutation group of $3$ variables.\\

  Let $V$ be a vector space over $F$ of dimension $2$, endowed with a non-degenerate symplectic form $\langle,\, \rangle$. Let $\{e_1,e_2\}$ be a symplectic basis of $V$\, i.e. $\langle e_1, e_2 \rangle=1, \, \langle e_2, e_1\rangle =-1 $. We attach  the vector space $V^{\otimes 3}= V\otimes_F V\otimes_F V$ with the natural  symplectic
form $\langle, \rangle \otimes \langle, \rangle \otimes \langle, \rangle $, so there exists a homomorphism  $p$ from $\Big(\GSp(V) \times \GSp(V) \times \GSp(V)\Big)\rtimes S_3$  to $\GSp(V^{\otimes 3})$. The above group $S_3$ acts on $V\otimes V\otimes V$ by permutations.  By  the fixed basis $\{ e_1, e_2\}$ of V and $\{e_i\otimes e_j \otimes e_k| 1\leq i,j,k \leq 2\}$ of $V^{\otimes 3}$, we could identify the group $G$ with $\GL(V)$, and  the group $\GSp_8(F)$ with $\GSp(V^{\otimes 3})$.

   Let $\rho$ be the Weil representation of $\GSp(V^{\otimes 3})$.  Through the above morphism $p$, we get a representation $\pi'$ of the group $\bigg(\GSp(V)\times \GSp(V) \times \GSp(V)\bigg)\rtimes S_3$. Let $\pi$ denote  the restriction  of $\pi'$ to $\GSp(V)\times \GSp(V) \times \GSp(V)$. Write $_+V^{\otimes 3}=\{x\in V^{\otimes 3} | x\in Fe_1\otimes V \otimes V\}$ and  $_-V^{\otimes 3}=\{y\in V^{\otimes 3} | y\in Fe_2\otimes V \otimes V\}$. Every element $y\in {}_-V^{\otimes 3}$ has the form $y=\sum_{j,k=1}^2 a_{j,k}e_2\otimes e_j \otimes e_k$, which corresponds to a matrix
$m=\begin{pmatrix}
  a_{11}& a_{12}\\
  a_{21}&a_{22} \end{pmatrix} $. So we could identify $ _-V^{\otimes 3}$ with the matrix ring $M_2(F)$ as  vector space over $F$. The representation $\pi$ of $G \times G \times G$ can be realized in the  vector space $W=\C[M_2(F) \times X_F]$ of complex functions on $M_2(F)  \times X_F$.

\begin{proposition}\label{Mode2}
The representation $(\pi, G\times G\times G, W)$ is given by the following formulas:\\
\begin{equation}\label{weilGGG1}
 \big(\pi[h(a) , 1 , 1]f\big)(m,\psi)=f(am, \psi),
\end{equation}
\begin{equation}\label{weilGGG2}
  \big(\pi[u(b) , 1 , 1]f\big)(m,\psi)= \psi(b\, det (m)) f(m,\psi),
\end{equation}
\begin{equation}\label{weilGGG3}
\big(\pi[h'(t) , 1 , 1]f\big)(m,\psi)=f(m, \psi^{t^{-1}}),
\end{equation}
\begin{equation}\label{weilGGG4}
\big(\pi[\omega , 1 , 1]f\big)(m,\psi)= q^{-2} \sum_{n\in M_2(F)} \psi(B(m,n)) \,f(n,\psi),
\end{equation}
\begin{equation}\label{weilGGG5}
\big(\pi[1 , g_2 , g_3]f\big)(m,\psi)= f(\det (g_2g_3)g_2^{-1}m(g_3^{-1})^t, \psi^{det(g_2g_3)^{-1}}),
\end{equation}
where $g_2,g_3 \in G, m\in M_2(F)$, $g_3^t$= the transpose  of $g_3$,
  $B(m,n)=m_{11}n_{22}+m_{22}n_{11}-m_{12}n_{21}-m_{21}n_{12}$ for $m=\begin{pmatrix}
  m_{11}& m_{12}\\
  m_{21}&m_{22} \end{pmatrix} $, $n =\begin{pmatrix}
  n_{11}& n_{12}\\
  n_{21}& n_{22} \end{pmatrix} \in M_2(F)$.
\end{proposition}
\begin{proof}[Proof]
 (\ref{weilGGG1})---(\ref{weilGGG4}) come directly from the formulas (\ref{weilGSp1})---(\ref{weilGSp4}) in Section \ref{Prel}. Consider now the formula (\ref{weilGGG5}). Recall, for $g\in G$, $g\cdot e_1:= (e_1,e_2) g \begin{pmatrix}
  1\\
 0 \end{pmatrix}$ and $g\cdot e_2:= (e_1,e_2) g \begin{pmatrix}
  0\\
 1 \end{pmatrix}$. By the fixed basis $\{ e_2 \otimes e_j\otimes e_k| 1 \leq j, k \leq  2\}$, we obtain $g_2 \otimes g_3 \cdot m:= g_2 m g_3^t$ for $g_2, g_3 \in G, m \in M_2(F)$. Then, by (\ref{weilGSp1}), (\ref{weilGSp4}) in Section \ref{Prel}, we have
 $$\Big( \pi[ 1, g_2, g_3]f\Big) (m, \psi)=\rho \bigg( \begin{pmatrix}
  1& 0\\
  0& \det(g_2g_3) \end{pmatrix} \begin{pmatrix}
  g_2 \otimes g_3& 0\\
  0& \det(g_2g_3)^{-1} g_2 \otimes g_3 \end{pmatrix}\bigg) f(m, \psi)$$
 $$= \rho \bigg( \begin{pmatrix}
  g_2 \otimes g_3& 0\\
  0& \det(g_2g_3)^{-1} g_2 \otimes g_3 \end{pmatrix}\bigg) f (m, \psi^{\det(g_2g_3)^{-1}})$$
$$= \chi_q^+\Big( \det(g_2 \otimes g_3)^2\Big) f(g_2^{-1} \otimes g_3^{-1} \det(g_2g_3) \cdot m, \psi^{\det(g_2g_3)^{-1}})
= f( \det(g_2g_3)g_2^{-1} m (g_3^{-1})^t, \psi ^{\det(g_2g_3)^{-1}}).$$
\end{proof}

\subsection{} To decompose the representation $\pi$, it involves to describe the $1\times G \times G$-invariant  part of the vector space $W$.

 We consider the set $\mathcal{S}=\{ (\pi_1, \pi_2)| \textrm{ for } i=1,2, (\pi_i,V_i) \in \Rep(G)$ such that $(\pi_1 \otimes \pi_2)|_Z=\id_{V_1 \times V_2}\}$. For each pair $(\pi_1, \pi_2) \in \mathcal{S}$, it determines a representation $(\pi_1 \otimes \pi_2, V_1 \otimes V_2)$ of the group $H$. We  write $\Irr_0(H)$ for  the set of the isomorphism classes of all irreducible representations $(\pi_1\otimes \pi_2, V_1 \otimes V_2)$ of $H$  for which $(\pi_1, \pi_2) \in \mathcal{S}$.

Now we concentrate on the decomposition of the representation $(\pi, G\times G \times G, W)$. Following the method  in \cite{Andr},  we first associate a representation $(\pi_0', G, W[\pi_1 \otimes \pi_2])$ to any representation $\pi_1 \otimes \pi_2 \in \Irr_0(H)$, where the vector space $W[\pi_1 \otimes \pi_2]$  consists of  all functions $f: M_2(F) \times X_F \longrightarrow V_1 \otimes V_2$ such that
\begin{equation}
 f(\det(g_1^{-1}g_2^{-1})g_1mg_2^t, \psi^{\det(g_1g_2)})= \big(\pi_1(g_1)\otimes \pi_2(g_2)\big)f(m,\psi),
\end{equation}
for $(g_1, g_2) \in H, m\in M_2(F), \psi \in X_F$, and the action of $G$ on $W[\pi_1 \otimes \pi_2]$  is given by the formulas (\ref{weilGGG1})--- (\ref{weilGGG4}) in Proposition \ref{Mode2} .

\begin{proposition}
For the representation $(\pi,  G\times G \times G, W)$, we have
$$ W=\bigoplus_{\pi_1\otimes \pi_2 \in Irr_0(H)}W[\pi_1\otimes \pi_2] \otimes \check{V}_1 \otimes \check{V}_2.$$
\end{proposition}
\begin{proof}[Proof]
Since the representation $(\pi, W)$ of $G \times G \times G$ is semi-simple and arises from the restriction of $\rho$, we have
$$W=\oplus_{\check{\pi}_1 \otimes \check{\pi}_2\in Irr_0(H)} \, W_{\check{\pi}_1\otimes \check{\pi}_2} \otimes  \check{V}_1 \otimes \check{V}_2.$$
Here $W_{\check{\pi}_1 \otimes \check{\pi}_2}$ is the greatest $\check{\pi}_1 \otimes\check{\pi}_2$-isotypic quotient of $W$ (cf. \cite[p. 46, III.4]{MVW}). Note that
\begin{equation}\label{decompositionnnn}
 W_{\check{\pi}_1 \otimes \check{\pi}_2} \simeq [ W\otimes (V_1 \otimes V_2)]^{ G\times G}
 \simeq W[V_1 \otimes V_2]
\end{equation}
as $G$-module.  In (\ref{decompositionnnn}), we treat $W$ as $G\times G$-module via the embedding $G \times G \simeq 1 \times G \times G \hookrightarrow G \times G \times G$.
\end{proof}

 Recall the Cartan involution:
 $\theta: G\longrightarrow G; g\longmapsto (g^t)^{-1}.$
 It is well-known that
 $\sigma  \simeq {(\sigma\circ \theta)}^{\vee}$ for  $\sigma \in \Irr(G)$.
 Let $\Big([\pi_1, \pi_2\circ t], \Hom_{\C}(V_2, V_1)\Big)$ be a representation of $G \times G$, defined by
 $$[\pi_1, \pi_2\circ t](g_1, g_2)(\varphi )= \pi_1(g_1) \circ \varphi \circ \pi_2(g_2^t),  \quad g_1, g_2 \in G, \varphi  \in \Hom_{\C}(V_2, V_1).$$
 Define an isomorphism of vector spaces:
  $$\lambda: V_1 \otimes V_2^{\star} \longrightarrow \Hom_{\C} (V_2, V_1);$$
$$ v_1\otimes v_2^{\star} \longrightarrow ( \varphi_{v_1\otimes v_2^{\star}}: v_2\longmapsto \langle v_2^{\star}, v_2 \rangle v_1).$$
It can be checked that $\lambda$ defines an intertwining operator between $(\pi_1 \otimes (\pi_2 \circ \theta)^{\vee}, V_1 \otimes V_2^{\star})$ and $\big([\pi_1, \pi_2 \circ t], \Hom_{\C}(V_2, V_1)\big)$.   To simplify calculation, we replace $(\pi_1 \otimes \pi_2, V_1 \otimes V_2)$ with $( [\pi_1, \pi_2 \circ t], \Hom_{\C}(V_2, V_1))$ in (\ref{decompositionnnn}), and get an isomorphic representation of $(\pi_0', G, W[\pi_1 \otimes \pi_2])$, say $(\pi_0, G, W[\pi_1, \pi_2])$, where $W[\pi_1, \pi_2]$ is a vector space over $\C$ consisting of all functions $f: M_2(F) \times X_F \longrightarrow \Hom_{\C}(V_2, V_1)$ such that
 \begin{equation}\label{star}
 f(\det(g_1^{-1}g_2^{-1})g_1 m g_2^t, \psi^{det(g_1g_2)})= \pi_1(g_1) \circ f(m,\psi) \circ \pi_2(g_2^t), \qquad g_1, g_2\in G,
\end{equation}
and the action of $G$ on $W[\pi_1,\pi_2]$ arises naturally from the above formulas (\ref{weilGGG1})---(\ref{weilGGG4}) in Proposition \ref{Mode2}.

\subsection{}We continue the above discussion, and determine the dimension of  the vector space $W[\pi_1, \pi_2]$.

For $\pi_1 \otimes \pi_2 \in \Irr_0(H),$ we write $W[\pi_1, \pi_2](\xi )= \{f(\xi)| f \in W[\pi_1, \pi_2], \xi\in M_2(F)\times X_F\}$. Now we define an $H$-action on the set $M_2(F) \times X_F$ as follows:
\begin{equation}
(g_1, g_2) (m, \psi):= \big(\det(g_1^{-1}g_2^{-1})g_1mg_2^t, \psi^{\det(g_1g_2)}\big),
\end{equation}
where $ (g_1, g_2) \in H, \psi\in X_F, m\in M_2(F)$. It is observed that $W[\pi_1, \pi_2](\xi)
=\text{Fix}_{\Hom_{\C}(V_2, V_1)}(\stab_H(\xi))$  for $\xi \in M_2(F)\times X_F$,
more precisely
\begin{equation}\label{equationsimple}
 W[\pi_1, \pi_2](\xi)= \{\varphi: V_2 \longrightarrow V_1|
\pi_1(g_1)\circ \varphi =\varphi \circ \pi_2 ( (g_2^{t})^{-1}), \quad  (g_1, g_2) \in \text{Stab}_H(\xi)\}.
\end{equation}
Let us determine the $H$-orbits in  $M_2(F)\times X_F$. They are of the following three kinds:
\begin{itemize}
\item[(i)] Orbit $\{ \xi_a\}$, where $\xi_a= \big(\begin{pmatrix}
  1& 0 \\
  0&1
\end{pmatrix}, \phi^a) \quad $ for any $a\in F^{\times}$;
\item[(ii)] Orbit $\{\eta \}$, where $\eta=(\begin{pmatrix}
  1& 0\\
  0&0
\end{pmatrix}, \phi\big)$;
\item[(iii)] Orbit $\{ \delta \}$, where $ \delta= \big(\begin{pmatrix}
  0& 0 \\
  0&0
\end{pmatrix}, \phi\big)$.
\end{itemize}
By straightforward calculation, the corresponding stabilizer of the given representative element in each orbit has the following form:
\begin{itemize}
\item[(i)] $ \Stab_H(\xi_a)= \{(g, (g^{-1})^t)| g\in G\}$;
\item[(ii)]$\Stab_H(\eta)= \{ (sn_1, s^{-1}n_2)| s\in T, n_1, n_2 \in N \}$;
\item[(iii)] $\Stab_H(\delta)=\{ (g_1, g_2)| $  $g_1, g_2 \in G$ and $\det (g_1g_2)=1 \}$.
\end{itemize}
To obtain the dimension of the vector space $W[\pi_1, \pi_2]$, we state the lemma:
\begin{lemma}\label{dimension}
(1) $W[\pi_1, \pi_2](\xi_a) \neq 0$ if and only if $\pi_1 \simeq \pi_2$, in which case $\dim_{\C}W[\pi_1, \pi_2](\xi_a)=1$;\\
(2) $W[\pi_1, \pi_2](\eta)=0$ except the following cases:
   \begin{itemize}
   \item[(a)] $\dim_{\C} W[\pi_{\chi_1, \chi_2}, \pi_{\chi_1, \chi_2}](\eta) =2$,
   \item[(b)] $\dim_{\C} W[\psi\cdot 1_G, \psi\!\cdot\!1_G](\eta)=1$,
   \item[(c)] $\dim_{\C} W[\psi\!\cdot\!\St_G, \psi\!\cdot\!\St_G](\eta)=1$,
   \item[(d)] $\dim_{\C} W[\psi\!\cdot\!1_G, \psi\!\cdot\!\St_G](\eta)=1$,
   \item[(e)] $\dim_{\C} W[\psi\!\cdot\!\St_G, \psi\!\cdot\!1_G](\eta)=1$,
\end{itemize}
for the characters $\chi_1 \neq \chi_2$, $ \psi$ of $F^{\times}$;\\
(3) $W[\pi_1, \pi_2](\delta)=0$ except $\pi_1=\pi_2= \psi\!\cdot\!1_G$, in that case $\dim_{\C}W[\psi\cdot 1_G, \psi\!\cdot\!1_G](\delta)=1$
for any character $ \psi$ of $F^{\times}$.
\end{lemma}
\begin{proof}
1) By the formula (\ref{equationsimple}), the vector space  $W[\pi_1, \pi_2](\xi_a)$ consists of  the functions $\varphi: V_2 \longrightarrow V_1$
such that
$ \pi_1(g_1)\circ \varphi = \varphi\circ  \pi_2((g_2^{-1})^t)$ for $(g_1, g_2) \in\Stab_H(\xi_a)$.
Hence  $W[\pi_1, \pi_2](\xi_a)$  is isomorphic to $\Hom_G(V_2, V_1)$.\\
2) Note that
$W[\pi_1, \pi_2](\eta) \simeq \Hom_T( V_2 ^N, V_1^N).$
Therefore $W[\pi_1, \pi_2](\eta)=0$ unless $\pi_1, \pi_2$ both are induced representations. Consider the induced representation
 $(\pi_{\chi_1, \chi_2}, V)=\Ind_B^G\, (\chi_1 \otimes \chi_2)$ for $ \chi_1, \chi_2 \in \Irr(F^{\times})$.  The vector space $V^N$ is generated by the following two functions $f_{\chi_1, \chi_2}$ and $g_{\chi_1, \chi_2}$, where
   \begin{itemize}
   \item[1.] the support of $f_{\chi_1, \chi_2}$ belongs to $B$, and $f_{\chi_1, \chi_2}(tn)=\chi_1 \otimes \chi_2(t)$ for $t\in T, n\in N$;
   \item[2.] the support of $g_{\chi_1, \chi_2}$ belongs to $B \omega' N$, and $g_{\chi_1, \chi_2}(tn_1 \omega' n_2)=\chi_1 \otimes \chi_2(t)$ for $t\in T,  n_1, n_2 \in N$.
\end{itemize}
  The action of $T$ on $V^N=\{ f_{\chi_1, \chi_2}, g_{\chi_1, \chi_2}\}$ is simply given by the formulas:
$$t \cdot f_{\chi_1, \chi_2}= \chi_1 \otimes \chi_2(t) f_{\chi_1, \chi_2} \quad \textrm{ and } \quad t\cdot g_{\chi_1, \chi_2}= \chi_2 \otimes \chi_1 (t) g_{\chi_1, \chi_2}, \quad  t\in T.$$
Thus, $\dim_{\C}\Hom_T(\pi_{\chi_1, \chi_2}^N, \pi_{\chi_1, \chi_2}^N)=2$ for $\chi_1 \neq \chi_2 \in \Irr(F^{\times})$,
 and  (a) follows. On the other hand $\dim_{\C}\Hom_T(\pi_{\psi, \psi}^N, \pi_{\psi, \psi}^N)=4$ for $\psi\in \Irr(F^{\times})$.
Clearly $\big(\St_G\big)^{N}=\C\big( q  f_{1_G, 1_G}- g_{1_G, 1_G}\big)$ and $\big(1_G\big)^N= \C \big( f_{1_G, 1_G} + g_{1_G, 1_G}\big)$. So  (b) and (c) hold and $\dim_{\C} W[\psi\!\cdot\!1_G, \psi\!\cdot\!\St_G](\eta) =\dim_{\C} W[\psi\!\cdot\!\St_G, \psi\!\cdot\!1_G](\eta)=1.$\\
3) By the formula (\ref{equationsimple}),
$W[\pi_1, \pi_2](\delta)$ consists of the functions
 $\varphi: V_2 \longrightarrow V_1$
such that
 $ \pi_1(g_1) \circ \varphi = \varphi \circ \pi_2((g_2^{-1})^t)$ for $  (g_1, g_2)\in \Stab_H(\delta)$.
 Since  $\Stab_H(\delta)=\{ (g_1, g_2)| $  $g_1, g_2 \in G$ and $\det (g_1g_2)=1 \}$, we know $\varphi=0$ except that
$\pi_1=\psi_1 \circ \det,  \quad \pi_2=\psi_2 \circ \det, $
 in which case,  $W[\pi_1, \pi_2](\delta)\simeq \Hom_{F^{\times}} (\psi_2, \psi_1)$,  and we get the result.
\end{proof}

\begin{corollary}\label{corollary1}
For any irreducible representation $\pi_1 \otimes \pi_2 \in \Irr_0(H)$, the dimension of the representation $(\pi_0, G, W[\pi_1, \pi_2])$ has the following form:
\begin{itemize}
\item[(i)] $\dim_{\C} W[\pi_{\chi_1, \chi_2}, \pi_{\chi_1, \chi_2}]=q+1$,
\item[(ii)]$\dim_{\C} W[\psi \cdot \St_G, \psi\cdot \St_G]=q$,
\item[(iii)] $\dim_{\C} W[\psi \cdot 1_G, \psi \cdot 1_G]=q+1$,
\item[(iv)] $\dim_{\C} W[ \pi_{\theta }, \pi_{\theta}]= q-1$,
\item[(v)] $\dim_{\C} W[\psi \cdot \St_G, \psi \cdot 1_G]= 1$,
\item[(vi)] $ \dim_{\C} W[\psi \cdot 1_G, \psi \cdot \St_G]=1$,
\end{itemize}
for the characters $\chi_1\neq \chi_2, \psi$ of $F^{\times}$, the regular character $\theta$ of $E^{\times}$. And the above lists are
all the representations $\pi_1\otimes \pi_2 \in \Irr_0(H)$, such that $W[\pi_1,\pi_2] \neq 0$.
\end{corollary}
\begin{proof}
Note that $\dim_{\C}W[ \pi_1, \pi_2]= \sum_{a\in F^{\times}} \dim_{\C} W[\pi_1, \pi_2](\xi_a) + \dim_{\C}W[\pi_1, \pi_2](\eta)+ \dim_{\C}W[\pi_1, \pi_2](\delta)$, so the corollary  results immediately from above Lemma \ref{dimension}.
\end{proof}
\subsection{} We have already calculated  the dimension of the vector space $W[\pi_1, \pi_2]$, and it suffices to prove  the main theorem in this section.
\begin{theorem}\label{mainth1}
For $(\pi, G\times G\times G, W)$, we have:
$$ \pi\simeq   \bigoplus_{\sigma\in \Irr(G)} \Bigg( \sigma\otimes \sigma\otimes \sigma  \Bigg) \oplus \bigoplus_{\psi\in \Irr(F^{\times})}\Bigg(
(\psi\!\cdot\!\St_G \!\otimes\!\psi\!\cdot\!1_G\!\otimes\!\psi\!\cdot\!1_G)
\,\oplus\, (\psi\!\cdot\!1_G \!\otimes\!\psi\!\cdot\!\St_G\!\otimes\!\psi\!\cdot\!1_G)
\,\oplus\, (\psi\!\cdot\!1_G \!\otimes\!\psi\!\cdot\!1_G\!\otimes\!\psi\!\cdot\!\St_G)\Bigg).$$
\end{theorem}
\begin{proof}[Proof]
Since $\pi= \Res_{G\times G\times G}^{(G\times G\times G)\rtimes S_3}\pi'$, by  Clifford theory, the representation $\pi$ is the direct sum of the following three kinds of  representations:
 \begin{itemize}
\item[(1)]  $\tau_0 \otimes \tau_0 \otimes \tau_0$ for $\tau_0\in \Irr(G)$;
\item[(2)]  $\tau_1 \otimes \tau_1 \otimes \tau_2 + \tau_1 \otimes \tau_2\otimes \tau_1 + \tau_2\otimes \tau_1 \otimes \tau_1$ for $\tau_1 \neq \tau_2 \in \Irr(G)$;
\item[(3)]$\tau_0' \otimes \tau_1 '\otimes \tau_2'+\tau_0'\otimes \tau_2'\otimes \tau_1'+\tau_1' \otimes \tau_0'\otimes \tau_2'+\tau_1' \otimes \tau_2'\otimes \tau_0'+\tau_2' \otimes \tau_0'\otimes \tau_1'+\tau_2' \otimes \tau_1'\otimes \tau_0'$ for three different representations $\tau_0',\tau_1',\tau_2'$ in $\Irr(G)$.
\end{itemize}
Comparing  with the results in Corollary \ref{corollary1}  gives the theorem.
\end{proof}
\section{The decomposition of the Weil representation of $\GL_2(F)\times \GL_2(E)$ }\label{demp2}
\subsection{} We first give some notations and formulate the representation concerned in this section.

In this section, we use the following notations: $G=\GL_2(F)$,  $ B=\{\begin{pmatrix}
  a& b\\
  0&d
\end{pmatrix} \in G \}
, \ N= \{
 \begin{pmatrix}
  1& b \\
  0&1
\end{pmatrix} \in G \}, \ T=\{\begin{pmatrix}
  a& 0 \\
  0&d
\end{pmatrix} \in G\},  \  Z=\{\begin{pmatrix}
  a& 0 \\
  0&a  \end{pmatrix} \in G \}$;
  $H=\GL_2(E)$,  $ B'=\{\begin{pmatrix}
  a& b\\
  0&d
\end{pmatrix} \in H \}
, \ N'= \{
 \begin{pmatrix}
  1& b \\
  0&1
\end{pmatrix} \in H \}, \ T'=\{\begin{pmatrix}
  a& 0 \\
  0&d
\end{pmatrix} \in H\},  \  Z'=\{\begin{pmatrix}
  a& 0 \\
  0&a  \end{pmatrix} \in H \}$; $h(r)=\begin{pmatrix}
  r& 0 \\
  0&r^{-1}
\end{pmatrix}$, $u(b)=\begin{pmatrix}
  1& b \\
  0&1
\end{pmatrix}$, $h'(t)=\begin{pmatrix}
  1& 0 \\
  0&t
\end{pmatrix}$, $\omega'=\begin{pmatrix}
  0& 1 \\
  1&0
\end{pmatrix}$, $\omega=\begin{pmatrix}
  0& 1 \\
  -1&0
\end{pmatrix}$  in $G$ or $H$; $\Gal(E/F)=\langle \sigma \rangle.$ \\
If $h \in H$ or $M_2(E)$, we will denote its conjugate by $h^{\sigma}$ or $\overline{h}$, its transpose by $h^t$, and let $h^{\star}:= \overline{h}^{t}$.\\

 Let $V$ be a vector space over $F$ of dimension 2, endowed with a symplectic form $\langle,\, \rangle$. Let $\{ e_1, e_2\}$ be a symplectic basis of $V$. Consider the $E$-vector space $V_E=E\otimes_F V$, endowed with the symplectic form $\langle, \rangle_{V_E}$ induced from $V$. Define a $\Gal(E/F)$-action on $V_E$  by
$$\Gal(E/F)\times E\otimes_F V \longrightarrow E\otimes_F V; (\sigma, \sum_i t_i \otimes e_i) \longmapsto \sum_i t_i^{\sigma}\otimes e_i.$$
Now let $\mathbb{W}= V_E\otimes_E V_E$ endowed with the symmetric form $\langle, \rangle_{\mathbb{W}}=\langle, \rangle_{V_E}\otimes \langle, \rangle_{V_E} $. On $\mathbb{W}$, we consider the twisted Galois action defined by
 $$\Gal(E/F)  \times \mathbb{W} \longrightarrow \mathbb{W}; (\sigma, w=\sum_i u_i\otimes v_i  ) \longmapsto {}^{\sigma}w=\sum_i  v_i^{\sigma} \otimes u_i^{\sigma}.$$
 We will let $\mathbb{W}_0$ denote the set $\{ w\in \mathbb{W}|{}^{\sigma}w=w\}$. It can be checked that the restriction of  $\langle, \rangle_{\mathbb{W}}$ to $\mathbb{W}_0$  defines an  $F$-symmetric form, denoted by $(, )_{\mathbb{W}_0}$.  Let $q$ be its associative quadratic form  given by
 $$(w_0, w'_0)_{\mathbb{W}_0}= q(w_0+w'_0)-q(w_0)-q(w_0'), \quad \quad  w_0, w_0' \in \mathbb{W}_0.$$
 By calculation, each $w_0\in \mathbb{W}_0$ may be expressed in the form
 $$w_0= x e_1\otimes e_1  + \alpha  e_1 \otimes e_2 + \overline{\alpha} e_2 \otimes e_1  + y e_2 \otimes e_2 \textrm{ for } x, y \in F, \alpha \in E.$$
   Every element $w_0$ corresponds to a matrix $\begin{pmatrix}
  x& \alpha  \\
  \overline{\alpha} &y \end{pmatrix}$. So we can  identify $\mathbb{W}_0$ with $M=\{ \begin{pmatrix}
  x& \alpha  \\
  \overline{\alpha} &y \end{pmatrix}| x, y \in F, \alpha \in E\}$. The symmetric form $q$ is transferred as $q(m)=\det(m)$ for $m\in M$.\\

Let $\GO(\mathbb{W})$ denote the  group of symmetric similitudes  of  $(\mathbb{W}, \langle, \rangle_{\mathbb{W}})$. By the definition of $\mathbb{W}$, there exists a morphism of groups: $ \GL(V_E) \times \GL(V_E)  \longrightarrow \GO(\mathbb{W})$. We define a twisted Galois action of $\Gal(E/F)$ on $ \GL(V_E) \times \GL(V_E)$  by
$$\Gal(E/F) \times \bigg(\GL(V_E) \times \GL(V_E)\bigg) \longrightarrow \GL(V_E) \times \GL(V_E) ; h=(g_1, g_2) \longmapsto {}^{\sigma} h:=( g_2^{\sigma}, g_1^{\sigma}).$$
Write $\overline{\GL(V_E)}=\{ h \in \GL(V_E) \times \GL(V_E) | {}^{\sigma}h=h\}.$
 Then there exists an isomorphism of groups $\GL(V_E) \longrightarrow \overline{\GL(V_E)}; g \longmapsto (g, g^{\sigma}).$
  If given $h\in \GL(V_E) \times \GL(V_E),  w\in \mathbb{W}= V_E\otimes_E V_E,$ one can verify that
   ${}^{\sigma}h\cdot {}^{\sigma}w={}^{\sigma}(h\cdot w)$.
  So it induces a morphism from $\GL(V_E)\simeq \overline{\GL(V_E)}$ to $\GO(M)$. By the fixed basis $\{ e_1,e_2\}$,  we obtain a morphism $i: H=\GL_2(E) \longrightarrow \GO(M)$.
\begin{lemma}
 The morphism $i: H=\GL_2(E) \longrightarrow \GO(M)$ is defined by  $H \times M \longrightarrow M; (h, m) \longmapsto h m \overline{h}^t,$ where $\overline{h}^t$ is the transpose conjugate of $h$.
\end{lemma}
 \begin{proof}
 Let  $h=\begin{pmatrix}
  a& b \\
  c&d
\end{pmatrix} \in H$. By definition,
$h\cdot (e_1, e_2):= (e_1,e_2) \begin{pmatrix}
  a& b \\
  c&d
\end{pmatrix}= (ae_1 + ce_2; be_1+ de_2).$
So
$(h, 1)\cdot(\alpha e_1\otimes e_1 + \beta e_1\otimes e_2 +\gamma e_2\otimes e_1 + \delta e_2 \otimes e_2)$
$= (a\alpha + b \gamma )e_1\otimes e_1 + (a\beta + b\delta) e_1\otimes e_2 + (c\alpha + d\gamma ) e_2 \otimes e_1 + (c\beta + d\delta) e_2 \otimes e_2.$
 Forgetting the basis we obtain
 $ (h, 1)\cdot\begin{pmatrix}
  \alpha& \beta \\
  \gamma&\delta
\end{pmatrix} = \begin{pmatrix}
  a& b \\
  c&d
\end{pmatrix}\begin{pmatrix}
  \alpha& \beta \\
  \gamma&\delta
\end{pmatrix}.$
 Similarly,
 $(1 , \overline{h}) \cdot(\alpha e_1 \otimes e_1 + \beta e_1 \otimes e_2 + \gamma e_2 \otimes e_1 + \delta e_2 \otimes e_2)$
  $=(\alpha \overline{a} + \beta \overline{b} ) e_1 \otimes e_1 + (\beta \overline{d} + \alpha \overline{c}) e_1 \otimes e_2 + (\gamma \overline{a} + \delta \overline{b} ) e_2 \otimes e_1 + (\delta \overline{d} + \gamma \overline{c}) e_2 \otimes e_2,$
  i.e.
  $(1 ,\overline{h})\cdot \begin{pmatrix}
  \alpha& \beta \\
  \gamma &\delta
\end{pmatrix} = \begin{pmatrix}
  \alpha \overline{a}+ \beta\overline{b} & \beta \overline{d} + \alpha \overline{c} \\
  \gamma \overline{a} + \delta \overline{b}  &\delta \overline{d} + \gamma \overline{c}
\end{pmatrix}=  \begin{pmatrix}
  \alpha& \beta \\
  \gamma &\delta
\end{pmatrix}  \begin{pmatrix}
  \overline{a}& \overline{c} \\
  \overline{b} &\overline{d}
\end{pmatrix}  = \begin{pmatrix}
  \alpha& \beta \\
  \gamma &\delta
\end{pmatrix} \overline{h}^t.$
\end{proof}

Now  we consider the symplectic vector space $V\otimes M$ over $F$ of dimension $8$. By the above discussion, there is a map from $G\times H$ to  $\GSp(V\otimes_F M) \simeq \GSp_8(F)$. Similarly as in Section \ref{demp1}, we consider the restriction of the  Weil representation $(\rho, \GSp_8(F), W)$ to the group $G\times H$,  denoted  by $(\pi, G\times H, W)$.\\

\begin{proposition}
The Weil representation $(\pi, G\times H, W)$ can be realized in the space $W=\C[M\times X_F]$, and the action of $G\times H$ on $W$ is given by the following formulas:
\begin{equation}\label{formula1}
 \big(\pi([h(a), 1])F\big)(m, \psi)= F(am, \psi ),
\end{equation}
\begin{equation}\label{formula2}
 \big(\pi([u(b), 1])F\big)(m, \psi)= \psi(b\det(m))F(m,\psi),
\end{equation}
\begin{equation}\label{formula3}
 \big(\pi([\omega, 1])F\big)(m,\psi)=-q^{-2} \sum_{n\in M} F(n, \psi)\psi\big(B( m, n)\big),
\end{equation}
\begin{equation}\label{formula4}
 \big(\pi([h'(t), 1])F\big)(m,\psi)= F(m, \psi^{t^{-1}}),
\end{equation}
\begin{equation}\label{formula5}
 \big(\pi([1,h])F\big)(m,\psi)= F(h^{-1}mh^{\star -1}\, \nnn_{E/F}(\det(h)), \psi^{\nnn_{E/F}(\det(h)^{-1})}),
\end{equation}
where $h(a), u(b), h'(t)\in G$; $h\in H, m\in M, \psi \in X_F$; $B(m,n):=q(m+n)-q(m)-q(n)$ for $m,n\in M$.
\end{proposition}
\begin{proof}[Proof]
(\ref{formula1}), (\ref{formula2}) and (\ref{formula4}) follow directly from (\ref{weilGSp1})---(\ref{weilGSp4}) in Section \ref{Prel}.\\
 For (\ref{formula5}): $$ \Big(\pi([1,h])F\Big)(m, \psi)=\rho\Bigg( \begin{pmatrix}
  h& 0 \\
  0&h
\end{pmatrix}\Bigg) F(m, \psi)
= \rho \Bigg( \begin{pmatrix}
  1& 0 \\
  0&\nnn_{E/F}(\det(h))
\end{pmatrix}\Bigg) \rho\Bigg( \begin{pmatrix}
  h& 0 \\
  0&h
\end{pmatrix} \cdot  \begin{pmatrix}
 1& 0 \\
  0& \nnn_{E/F}(\det(h))^{-1}
\end{pmatrix} \Bigg)
 F(m, \psi)$$
$$= F \Big( h^{-1}\cdot \nnn_{E/F}(\det(h))m, \psi^{\nnn_{E/F}(\det(h))^{-1}}\Big)= F\Big( h^{-1} mh^{\star -1}\nnn_{E/F}(\det(h)), \psi^{\nnn_{E/F}(\det(h)^{-1})}\Big).$$
For (\ref{formula3}): Let $E=F(\xi)$ with $\tr(\xi)=0$. Take $f_1= \begin{pmatrix}
 1& 0 \\
  0& \frac{1}{2} \end{pmatrix}, f_2=\begin{pmatrix}
  1& 0 \\
  0& -\frac{1}{2} \end{pmatrix}$, $f_3=\begin{pmatrix}
  0& 1 \\
  1& 0 \end{pmatrix}$, $f_4=\begin{pmatrix}
  0& \xi \\
  \overline{\xi}& 0 \end{pmatrix}$.
Suppose $\nnn_{E/F}(\xi)= \overline{\xi} \xi= -\xi^2=a$.  Then the set $\{ e_1\otimes f_1, e_1\otimes f_2, e_1 \otimes f_3, e_1 \otimes f_4; e_2\otimes f_1, -e_2\otimes f_2, -\frac{1}{2}e_2\otimes f_3, -\frac{1}{2} a^{-1} e_2\otimes f_4\}$ is a symplectic basis of $V\otimes_F M$. By such basis, the image of $\begin{pmatrix}
  0& 1 \\
  -1& 0 \end{pmatrix} \otimes 1$ in $\GSp(V\otimes_F M)$,  under the map $\GL(V) \times \GO(M) \longrightarrow \GSp(V\otimes_FM)$,  is $\overline{\omega} \begin{pmatrix}
  A& 0 \\
  0& A^{-1} \end{pmatrix}$, where  $\overline{\omega}=\begin{pmatrix}
  0 & I\\
  -I & 0 \end{pmatrix}$ and $A=\begin{pmatrix}
  1& 0 &  0& 0\\
  0 & -1 & 0 & 0\\
  0 & 0 & -2 & 0\\
  0&  0 & 0 &-2a \end{pmatrix}$.    Applying the formulas (\ref{weilGSp1}), (\ref{weilGSp3}) in Section \ref{Prel}, we get
  $$\rho( \overline{\omega} \begin{pmatrix}
  A& 0 \\
  0& A^{-1} \end{pmatrix}) F (e_2 \otimes m, \psi)
  = q^{-2} \sum_{n\in M} \Bigg( \rho\begin{pmatrix}
  A& 0 \\
  0& A^{-1} \end{pmatrix}F\Bigg) (e_2 \otimes n, \psi) \psi(\langle e_2\otimes n, \overline{\omega}^{-1}(e_2 \otimes m)\rangle)$$
  $$= q^{-2} \chi_q^+(2^2\xi^2) \sum_{n\in M} F \Big( \begin{pmatrix}
 A^{-1}& 0 \\
  0&A  \end{pmatrix} \cdot e_2\otimes n, \psi\Big) \psi (\langle e_2 \otimes n, \overline{\omega}^{-1}(e_2 \otimes m)\rangle)
  = -q^{-2} \sum_{n\in M} F(e_2 \otimes n, \psi) \psi( \langle \begin{pmatrix}
 A& 0 \\
  0&A^{-1}  \end{pmatrix} e_2 \otimes n, \overline{\omega}^{-1}(e_2\otimes m)\rangle)$$
  $$=-q^{-2} \sum_{n\in M} F(e_2\otimes n, \psi) \psi(\langle e_2\otimes n, \big( \overline{\omega} \begin{pmatrix}
 A& 0 \\
  0&A^{-1}  \end{pmatrix} \big)^{-1} (e_2\otimes m)\rangle)
  = -q^{-2} \sum_{n\in M} F(e_2\otimes n, \psi) \psi(\langle e_2\otimes n, \begin{pmatrix}
 0& 1 \\
  -1&0  \end{pmatrix}^{-1}(e_2) \otimes m\rangle )$$
  $$=-q^{-2} \sum_{n\in M} F(e_2 \otimes n, \psi) \psi(\langle e_2 \otimes n, -e_1\otimes m\rangle)
  =-q^{-2} \sum_{n\in M} F( e_2\otimes n, \psi) \psi(B(m,n)).$$
  Forgetting the basis, we get the formula (\ref{formula3}).
  \end{proof}
  \begin{remark}
  The above representation $\pi$ of $G \times H$ is compatible with $\rho_{Q}$ constructed by Andrade in \cite[p.35, Theorem 5]{Andr} (There the $\rho_Q$ is defined more directly).
  \end{remark}
\subsection{} To decompose the representation $\pi$ of $G \times H$, let us  calculate the dimension of the vector space $W[\pi_1]$( see below  for its definition).

Let $U=\{ x\in E^{\times}| N_{E/F}(x)=1\}$.  We regard $U$ as a subgroup of $H$. Let $\Irr_0(H)$ be the set of the isomorphism  classes of the irreducible representations $ \pi_1$ of $H$, such that $\pi_1$ is  trivial over $U$. For each representation $(\pi_1, V_1) \in \Irr_0(H)$, we associate a  representation $(\pi_0, W[\pi_1])$ of $G$, where the vector space $W[\pi_1]$ consists of   functions: $ f: M \times X_F \longrightarrow V_1$ such that
\begin{equation}\label{star2}
 f\big(hmh^{\star} \nnn_{E/F}(\det(h)^{-1}), \psi^{\nnn_{E/F}(\det(h))}\big)= \pi_1(h) \circ f(m, \psi)
\end{equation}
for $h\in H, (m, \psi) \in M\times X_F$ and the action of $G$ on $W[\pi_1]$ is given by the formulas (\ref{formula1})---(\ref{formula4}).
\begin{proposition}
For the representation $(\pi, G\times H, W)$, we have the following decomposition:\\
\begin{displaymath}
\pi\simeq \bigoplus_{(\pi_1, V_1)\in \Irr_0(H)} W[\pi_1] \otimes \check{V}_1.
\end{displaymath}
\end{proposition}
\begin{proof}[Proof]
The representation $W$ has the decomposition $\pi = \oplus_{(\pi_1, V_1)} W_{\check{\pi}_1} \otimes \check{V}_1$ and then  $W_{\check{\pi}_1} \simeq (W\otimes V_1)^H \simeq W[\pi_1]$.
The action of $G$ on $W[\pi_1]$ arises from the definition of $\pi$ and  above isomorphisms.
\end{proof}
For $(\pi_1, V_1)\in \Irr_0(H),$  we define $W[\pi_1](\xi )= \{f(\xi)| f \in W[\pi_1]\}$ and an $H$-action on the set $M \times X_F$ as follows:
\begin{equation}
h\cdot (m, \psi):= \big(hmh^{\star}\nnn_{E/F}(\det(h))^{-1}, \psi^{\nnn_{E/F}(\det(h))}\big)
\qquad\qquad  h \in H, \psi\in X_F, m\in M.
\end{equation}
It is observed that $W[\pi_1](\xi)
=V_1^{\stab_H(\xi)}$  for any  $\xi \in M\times X_F$.
 \begin{proposition}\label{theobitssection2}
Consider the action of $H$ on $M\times X_F$.\\
(1) The distinct orbit of this action can be described as follows:
  \begin{itemize}
   \item[(i)]  $\Orbit \{ \xi_a\}$, where $\xi_a= \big(\begin{pmatrix}
  1& 0 \\
  0&1
\end{pmatrix}, \phi^a) \quad $ for any $a\in F^{\times}$;
   \item[(ii)]  $\Orbit \{\eta \}$, where $\eta=(\begin{pmatrix}
  1& 0\\
  0&0
\end{pmatrix}, \phi\big)$;
\item[(iii)]  $\Orbit \{ \delta \}$, where $ \delta= \big(\begin{pmatrix}
  0& 0 \\
  0&0
\end{pmatrix}, \phi\big)$.
   \end{itemize}
(2) The corresponding stabilizer of the canonical element in each orbit is presented as following:
\begin{itemize}
   \item[(i)]  $ \Stab_H(\xi_a)= U_2(E)$, where $U_2(E)=\{h\in H| hh^{\star}=1\}$; \\
   \item[(ii)]  $\Stab_H(\eta)=H_1$, where $H_1=\{ h=\begin{pmatrix}
  u& b\\
  0&v
\end{pmatrix}| u,v \in U, b\in E \}$;
\item[(iii)]   $\Stab_H(\delta)=H_2$, where $H_2=\{ h\in H| \det (h)\in U \}$.
   \end{itemize}
\end{proposition}
\begin{proof}
We transfer the $H$-action $\cdot$ to another $H$-action $\odot $, where $\odot$ is defined by $h\odot (m,\psi):=\big(hmh^{\star}, \psi^{\nnn_{E/F}(\det(h)^{-1})}\big)$.  Since  $\alpha: H \longrightarrow H; h\longmapsto h/{\det(h)}$ is a group isomorphism and $\alpha(h)\odot(m,\psi)=h\cdot (m,\psi)$, it reduces to consider the action $\odot$.\\
1) Every element $m\in M$ corresponds to a hermitian form on a $2$-dimension $E$-vector space $V$. By the property of  hermitian form over finite fields and the surjection of the morphism $\nnn_{E/F}: E^{\times} \longrightarrow F^{\times}$, one can find $h\in H$ such that $hmh^{\star}=\diag(a,b)$ where $(a,b)=(1,1)$, $(1,0)$ or $(0,0)$. Let $H_{a,b}= \Stab_{H}(\diag (a,b))$. By calculation, we know:  (a) $H_{1,1}= \{ h \in H| h h^{\star}=1\};$ (b) $H_{1, 0}=\{ \begin{pmatrix}
  u& b\\
  0&v
\end{pmatrix} \in H | u\in U\}; $ (c)  $H_{0,0}= H.$    Consequently we define an action of $H_{a, b}$ on $X_F$ by
$ (h, \psi) \longrightarrow \psi^{\nnn_{E/F}(\det(h))} $ for $h \in H_{a,b}$ and $\psi \in X_F$.
As above, the results follow  by determining the orbits.\\
2) It is straightforward.
\end{proof}
\begin{lemma}\label{thegroupunitary}
\begin{itemize}
   \item[(i)]  $U_2(E)=\{\begin{pmatrix}
 u& 0\\
   0&1\end{pmatrix}\begin{pmatrix}
 a& b\\
   -b^q&a^q\end{pmatrix}| u\in U, a,b \in E, \nnn_{E/F}(a) + \nnn_{E/F}(b)=1\}$;
   \item[(ii)] $|U_2(E)|=(q-1)q(q+1)^2$;
\item[(iii)]   If we choose an element $e_{-1}\in E^{\times}$, such that $e_{-1}^{q+1}=-1$, then the element
$\begin{pmatrix}
 0& 1\\
   -1&e_{-1}^q\end{pmatrix} \notin B' U_2(E)$;
\item[(iv)] $H=B'U_2(E) \cup B' \begin{pmatrix}
 0& 1\\
   -1&e_{-1}^q\end{pmatrix} U_2(E)$.
   \end{itemize}
\end{lemma}
\begin{proof}
(i)(ii) follow from \cite[p. 242-243]{Andr}.\\
iii) $$B'U_2(E)=B'\SU_2(E)=\{\begin{pmatrix}
 x& y\\
  0&z\end{pmatrix} \cdot \begin{pmatrix}
 a& b\\
  -b^q&a^q\end{pmatrix} | x,z\in E^{\times}; a, b, y \in E; \nnn_{E/F}(a)+\nnn_{E/F}(b)=1\}$$
  $$=\{ \begin{pmatrix}
 xa-yb^q& xb+ya^q\\
  -zb^q&za^q\end{pmatrix} | x,z \in E^{\times};  a, b ,y\in E; \nnn_{E/F}(a) + \nnn_{E/F}(b)=1\}.$$
So, $\nnn_{E/F}(-b^qz)+\nnn_{E/F}(a^qz)=\nnn_{E/F}(z)\neq 0$. On the other hand, $\nnn_{E/F}(-1)+ \nnn_{E/F}(-e_{-1}^q)=0$; this implies the result.\\
iv) It is enough to check that they have the same cardinality, i.e.
$$| H|= |B'U_2(E)| +| B'  \begin{pmatrix}
 0& 1\\
  -1&e_{-1}^q\end{pmatrix} U_2(E)|.$$
  By calculation,
  $$B'\cap U_2(E)=\{ \begin{pmatrix}
 u& 0\\
   0&1\end{pmatrix} \begin{pmatrix}
 a& 0\\
   0&a^q\end{pmatrix} | u, a \in U\}.$$
   So
   $$|B'U_2(E)|=|B'||\frac{U_2(E)}{B'\cap U_2(E)}|=| B'| \cdot (q-1)q.$$
    Now let $$g=\begin{pmatrix}
 ua& ub\\
   -b^q&a^q\end{pmatrix}\in U_2(E),  g_0= \begin{pmatrix}
 0& 1\\
   -1&e_{-1}^q\end{pmatrix}.$$
    Then $$ g_0gg_0^{-1}= \begin{pmatrix}
 \big( a-be_{-1}\big)^q& b^q\\
   e_{-1}^q[ \big( a-e_{-1}b\big)^q-u\big( a-e_{-1}b\big)]& ua+\big( be_{-1}\big)^q\end{pmatrix},$$
    which is an element of $B'$ if and only if  $(a-be_{-1})^{q-1}=u$. Since $\det(g_0 g g_0^{-1})=u$, we have
 $$ B' \cap g_0 U_2(E) g_0^{-1}=\{  \begin{pmatrix}
 \big( a-be_{-1}\big)^q& b^q\\
   0&\big( a-be_{-1}\big)^{-1}\end{pmatrix}| \nnn_{E/F}(a) + \nnn_{E/F}(b)=1\}.$$
 So $$| B' \cap g_0 U_2(E) g_0^{-1}|=| \SU_2(E)|=(q-1)q(q+1).$$
  From this, we obtain
 $$|g_0 U_2(E) g_0^{-1}/ B' \cap g_0 U_2(E) g_0^{-1}|=q+1,$$
 and
 $$|B' g_0 U_2(E)|=|B'| |g_0 U_2(E) g_0^{-1}/ B' \cap g_0 U_2(E) g_0^{-1}|= |B'| \cdot (q+1).$$
 Finally
 $$ |B' U_2(E)|+ | B' g_0 U_2(E)|=|B'| \cdot (q-1)q+ |B'| \cdot (q+1)=|B'| \cdot (q^2+1)=|H|.$$
 \end{proof}

Let us determine the dimension of the vector space $W[\pi_1]$ for each $\pi_1\in \Irr_0(H)$.
 \begin{lemma}\label{dimesion2}
 Let $(\pi_1, V_1)$ be an irreducible representation of $H$ in  $\Irr_0(H)$. Then:\\
 (1) In the case  $\pi_1= \Psi\cdot 1_H$ with $\Psi\in \Irr(E^{\times})$, we have
   \begin{itemize}
   \item[(a)] $ W[\pi_1]=0$, if $\Psi \neq \Psi^q$.
   \item[(b)] $\dim_{\C} W[\pi_1](\xi_a)=\dim_{\C} W[\pi_1](\eta)=\dim_{\C} W[\pi_1](\delta)=1$,  if $\Psi = \Psi^q$.
   \end{itemize}
(2)  In the case $\pi_1= \Psi\cdot \St_H$ with $\Psi\in \Irr(E^{\times})$, we have
   \begin{itemize}
   \item[(a)]  $ W[\pi_1]=0$, if $\Psi \neq \Psi^q$.
   \item[(b)] $\dim_{\C} W[\pi_1](\xi_a)=\dim_{\C} W[\pi_1](\eta)=1$ and $\dim_{\C} W[\Psi\cdot \St_H](\delta)=0$, if $\Psi = \Psi^q$.
   \end{itemize}
(3) In the case $\pi_1= \Pi_{\Lambda,\Sigma}$ with   $\Lambda\neq \Sigma \in \Irr(E^{\times})$, we have
   \begin{itemize}
   \item[(a)] $\dim_{\C} W[\Pi_{\Lambda,\Sigma}](\xi_a)=1$,  $\dim_{\C} W[\Pi_{\Lambda,\Sigma}](\eta)=2$, $\dim_{\C} W[\Pi_{\Lambda,\Sigma}](\delta)=0$,  if $\Lambda= \Lambda^q$ and  $ \Sigma=\Sigma^q$.
   \item[(b)] $\dim_{\C} W[\Pi_{\Lambda,\Sigma}](\xi_a)=1$, $\dim_{\C} W[\Pi_{\Lambda,\Sigma}](\eta)=\dim_{\C} W[\Pi_{\Lambda,\Sigma}](\delta)=0$, if $\Lambda=\Sigma^q$ and  $\Sigma= \Lambda^q$.

   \end{itemize}
   For the other kind of  $\Lambda\neq \Sigma \in \Irr(E^{\times})$, $ W[\Pi_{\Lambda,\Sigma}]=0$.\\
(4) In the case $\pi_1= \Pi_{\Theta}$, where $\Theta \in \Irr(E_1^{\times})-\Irr(E^{\times})$ for some quadratic extension $E_1$ of $E$, we have
\begin{itemize}
   \item[(a)] $W[\Pi_{\Theta}]=0$.
\end{itemize}
\end{lemma}
 \begin{proof}
 1) By Proposition \ref{theobitssection2} (2), we know that the image of the map
 $\det: \Stab_H(\xi)\longrightarrow E^{\times}$ is $U$ for $\xi=\xi_a, \eta, \delta$. So
$V_1^{\Stab_H(\xi)}=\left\{\begin{array}{clrr}
     V_1 &  \textrm{ if } \Psi \textrm{ is trivial over } U, \\       0  &  \textrm{ otherwise}.
     \end{array}\right.$\\
2) Let $(\pi_2, V_2)=\Ind_{B'}^H(\Psi\cdot 1_{B'})$.
  By definition, $W[\pi_2](\xi)=W[\Psi\cdot1_H](\xi) \oplus W[\Psi \cdot \St_H](\xi)$. The dimension of  $W[\Psi \cdot 1_H](\xi)$ is known,   so it remains to calculate the dimension $W[\pi_2](\xi)$ for $\xi=\xi_a, \eta, \delta$.\\
(a) $\xi=\xi_a$.  In this case,   $\Stab_H(\xi)=U_2(E)$ and
   $H=B'U_2(E)\cup B'\begin{pmatrix}
 0& 1\\
  -1&e_{-1}^q\end{pmatrix} U_2(E)$ for $ e_{-1}\in E^{\times}$ satisfying $\nnn_{E/F}(e_{-1})=-1$. So $V_2^{U_2(E)}$ is generated by the following functions $\alpha, \beta$:
   \begin{itemize}
   \item[1.] The support of $\alpha$ is $B'U_2(E)$, $\alpha(bu)=\Psi\cdot 1_{B'}(b)$ for $b\in B', u\in U_2(E)$, and $\Psi\cdot 1_{B'}$ is trivial over $B'\cap U_2(E)$.
   \item[2.]  The support of $\beta$ is $B'\begin{pmatrix}
0& 1\\
  -1&e_{-1}^q\end{pmatrix} U_2(E)$,  $\beta\Big(b\begin{pmatrix}
 0& 1\\
  -1&e_{-1}^q\end{pmatrix}u\Big)=\Psi\cdot 1_{B'}(b)$ and $\Psi\cdot 1_{B'}$ is trivial over $B'\cap \begin{pmatrix}
 0& 1\\
  -1&e_{-1}^q\end{pmatrix}U_2(E) {\begin{pmatrix}
 0& 1\\
  -1&e_{-1}^q\end{pmatrix}}^{-1}$.
\end{itemize}
 By Lemma \ref{thegroupunitary},
 $B'\cap U_2(E)=\{ \begin{pmatrix}
 u& 0\\
   0&1\end{pmatrix} \begin{pmatrix}
 a& 0\\
   0&a^q\end{pmatrix} | u, a \in U\},$ then
  $ \left\{\begin{array}{lr}
     \alpha \neq 0 &  \textrm{ if } \Psi = \Psi^q, \\     \alpha =0  &  \textrm{ otherwise}.
     \end{array}\right.$ On the other hand, by Lemma \ref{thegroupunitary},
  $\begin{pmatrix}
 0& 1\\
  -1&e_{-1}^q\end{pmatrix} g  \begin{pmatrix}
 e_{-1}^q& -1\\
   1&0\end{pmatrix} =\begin{pmatrix}
 -(be_{-1})^q+ a^q& b^q\\
   u(-b-ae_{-1}^q)+(ae_{-1})^q-(be_{-1}^2)^q&ua+ (be_{-1})^q\end{pmatrix} \in B'$,
    which implies that
    $u(ae_{-1}^q+b)=(ae_{-1})^q-(be_{-1}^2)^q $, $\nnn_{E/F}(a)+ \nnn_{E/F}(b)=1$ and $ u\in U$.
     In particular, in case  $a=0$,   $ u=(e_{-1}b)^{q-1}$ and  $\nnn_{E/F}(e_{-1}b)=-1$; in case $b=0$,    $u=a^{q-1}$ and $\nnn_{E/F}(a)=1$.
    By calculation,      we see
  $U= \{ u| u=a^{q-1} \text{ with } a \in E^{\times} \textrm{ and } \nnn_{E/F}(a)=\pm 1\}$. Hence
     $\{ \det(g)=u| \begin{pmatrix}
 0& 1\\
  -1&e_{-1}^q\end{pmatrix} g \begin{pmatrix}
 e_{-1}^q& -1\\
   1&0\end{pmatrix} \in B',  g=\begin{pmatrix}
 u& 0\\
   0&1\end{pmatrix} \begin{pmatrix}
 a& b\\
   -b^q&a^q\end{pmatrix} \in U_2(E)\}=U$.
    Finally, we see
$ \left\{\begin{array}{cr}
     \beta \neq 0 &  \textrm{ if } \Psi = \Psi^q, \\     \beta =0  &  \textrm{ otherwise}.
     \end{array}\right.$\\
(b) $\xi=\eta$. In this case,
$\Stab_H(\xi)=H_1= N'\rtimes T''$ and $V_2^{H_1}=(V_2^{N'})^{T''}$ for $ T''=\{ \begin{pmatrix}
 u& 0\\
   0&v\end{pmatrix}| u,v\in U\}$. By Lemma \ref{dimension}, as shown above,  the vector space $V_2^{N'}$ is generated by   the functions $f_{\Psi, \Psi}, g_{\Psi, \Psi}$.  We know that
   $t\cdot f_{\Psi, \Psi}= \Psi \otimes \Psi(t) f_{\Psi, \Psi}$,  and $ t\cdot g_{\Psi, \Psi}= \Psi \otimes \Psi(t) g_{\Psi, \Psi} $ for $ t\in T$.  So
   $V_2^{H_1}=\left\{\begin{array}{cr}
     \{ f_{\Psi, \Psi}, g_{\Psi, \Psi} \} &  \textrm{ if } \Psi = \Psi^q, \\     0  &  \textrm{ otherwise}.
     \end{array}\right.$\\ (c)  $\xi=\delta$. In this case, we have
$\Stab_H(\delta)=H_2$, $H=B'H_2$ and $B'\cap H_2=\{\begin{pmatrix}
a& b\\
   0&d\end{pmatrix}| ad\in U\}$. So $\dim_{\C}V_2^{H_2}=\left\{\begin{array}{cr}
     1 &  \textrm{ if } \Psi|_U=id_U, \\     0  &  \textrm{ otherwise}.
     \end{array}\right.$\\
3)  (a)  In  case $\xi=\xi_a$,  $V_1^{U_2(E)}$ is generated by the following two functions  $\alpha, \beta$ in $V_1$:
\begin{itemize}
   \item[1.]  The support of $\alpha$ belongs to $ B'U_2(E)$,  $\alpha(bu):=\Lambda  \Sigma (b)$ for $b\in B', u\in U_2(E)$ and $\Lambda \otimes \Sigma$ is trivial over $B'\cap U_2(E)$;
   \item[2.]  The support of $\beta$ belongs to $B'\begin{pmatrix}
0& 1\\
  -1&e_{-1}^q\end{pmatrix} U_2(E)$, $\beta\Big(b\begin{pmatrix}
 0& 1\\
  -1&e_{-1}^q\end{pmatrix}u\Big):=\Lambda  \otimes \Sigma(b)$ and $\Lambda  \otimes \Sigma$ is trivial over $B'\cap \begin{pmatrix}
 0& 1\\
  -1&e_{-1}^q\end{pmatrix}U_2(E) \begin{pmatrix}
 e_{-1}^q& -1\\
   1&0\end{pmatrix}$.
    \end{itemize}
By Lemma \ref{thegroupunitary},
$B'\cap U_2(E) =\{ \begin{pmatrix}
ua& 0\\
   0&a^q\end{pmatrix}| a, u \in U\}.$
   Therefore
   $ \left\{\begin{array}{cr}
     \alpha  \neq 0 &  \textrm{ if } \Lambda=\Lambda^q, \Sigma=\Sigma^q, \\  \alpha=   0  &  \textrm{ otherwise}.
     \end{array}\right.$
On the other hand, by Lemma \ref{thegroupunitary},
    $B'\cap \begin{pmatrix}
 0& 1\\
  -1&e_{-1}^q\end{pmatrix}U_2(E) \begin{pmatrix}
 e_{-1}^q& -1\\
   1&0\end{pmatrix}=\{\begin{pmatrix}
(a-b e_{-1})^q& b^q\\
   0&( a-be_{-1})^{-1}\end{pmatrix}| \nnn_{E/F}(a)+ \nnn_{E/F}(b)=1\}$.
   Let
   $t=a-be_{-1}, s=a+be_{-1};$ then $\nnn_{E/F}(a)+\nnn_{E/F}(b)=1$ is equivalent to $ts^q+st^q=2$. And
    $ \beta  \neq 0$   if  and only if $ \beta\Bigg( \begin{pmatrix}
 t^q& b^q\\
   0&t^{-1}\end{pmatrix} \Bigg) =\Lambda^q\Sigma^{-1}(t)=1$    for  $ t\in E^{\times} $  satisfying $ ts^q+st^q=2$.
    Considering $t=s^{-q}$, we know
    $ \left\{\begin{array}{cr}
     \beta  \neq 0 &  \textrm{ if } \Sigma =\Lambda^q,\Sigma\neq\Lambda,  \\   \beta=  0  &  \textrm{ otherwise}.
     \end{array}\right.$\\
(b) In case $\xi=\eta$,   $\Stab_H(\xi)=N'\rtimes T''$ and $V_1^{N'}$ is generated by the  functions $f_{\Lambda, \Sigma}, g_{\Lambda, \Sigma}$ defined
  in Lemma \ref{dimension}. Considering the $T''$-action on $V_1^{N'}$, we know
    $V_1^{H_1}= \left\{\begin{array}{lr}
     \{ f_{\Lambda, \Sigma}, g_{\Lambda, \Sigma} \} &  \textrm{ if } \Lambda=\Lambda^q, \Sigma=\Sigma^q,  \\     0  &  \textrm{ otherwise}.
     \end{array}\right.$\\
(c) In  case  $\xi=\delta$, $\Stab_H(\xi)=H_2$,  and  $H=B'H_2,  B'\cap H_2=\{\begin{pmatrix}
 a& b\\
 0&d\end{pmatrix}| ad \in U\}$.
   $V_1^{H_2}$ is generated by the  function $f$, where
  $f(bh)=\big(\Lambda \otimes \Sigma\big)(b) \textrm{ for } b\in B', h\in H_2$ such that
  $\Lambda \otimes \Sigma$ is trivial over $B'\cap H_2$; this implies $V_1^{H_2}=0$.\\
 4) In  case $\xi=\xi_a$,
 $U_2(E) \supseteq \SU_2(E)=\SL_2(E)\cap U_2(E),$
  and there exists $h\in H$ such that $h\SU_2(E)h^{-1}=\SL_2(F)$ (see \cite[ p. 242, Proposition 4]{Andr}). Hence
  $V_1^{U_2(E)} \subseteq V_1^{\SU_2(E)} \simeq V_1^{\SL_2(F)},$
which vanishes  by \cite[ p. 82, Proposition 1]{Andr}.\\
 If  $\xi=\eta, \delta$,  then $H_i \supseteq N'$ and $V_1^{N'}=0$. So $V_1^{H_i}=0$ for $i=1,2$.
 \end{proof}
\begin{corollary}\label{theirreduciblerepresentationH}
Let $\pi_1$ be an irreducible representation of $H$ in $\Irr_0(H)$. Then:
\begin{itemize}
   \item[(i)] $\dim_{\C} W[\pi_1]=q+1$, \quad if $\pi_1=\Psi \cdot 1_H$ for $  \Psi\in \Irr(E^{\times})$ with $\Psi=\Psi^q$.
   \item[(ii)] $\dim_{\C} W[\pi_1]=q$,  \quad if $\pi_1=\Psi \cdot \St_H$ for $ \Psi\in \Irr(E^{\times})$ with $\Psi=\Psi^q$.
   \item[(iii)] $\dim_{\C} W[\pi_1] = q+1$, \quad if $\pi_1=\Pi_{\Lambda,\Sigma}$ for $\Lambda\neq \Sigma \in \Irr(E^{\times}), \Lambda= \Lambda^q , \Sigma=\Sigma^q $.
   \item[(iv)] $\dim_{\C} W[ \pi_1] = q-1$,  \quad if $\pi_1=\Pi_{\Lambda,\Sigma}$ for $\Lambda\neq \Sigma \in \Irr(E^{\times}), \Lambda= \Sigma^q, \Sigma=\Lambda^q $.
   \end{itemize}
 And the above lists are all the representations $\pi_1 \in \Irr_0(H)$, such that $W[\pi_1] \neq 0$.
\end{corollary}
\begin{proof}
As is known that $\dim_{\C}W[\pi_1]= \sum_{a\in F^{\times}} \dim_{\C} W[\pi_1](\xi_a) + \dim_{\C} W[\pi_1](\eta)+ \dim_{\C}W[\pi_1](\delta)$, so the results follow from above Lemma \ref{dimesion2}.
\end{proof}
\subsection{  The representation $(\pi_0, W[\pi_1])$ I}\label{determinetherepresentationI}
In this subsection, let
$\pi_1=\Psi\cdot 1_H$,   where $\quad \Psi=\Psi^q \in \Irr(E^{\times})$ and $\Psi=\psi\circ \nnn_{E/F}$ for some $\psi\in \Irr(F^{\times})$. The vector space $W[\pi_1]$ is generated by the  functions $F_a, R, S: M\times X_F \longrightarrow V_1$ for any $a\in F^{\times}$. Namely they  all satisfy the equality (\ref{star2}) and
   \begin{itemize}
   \item[1.] $\supp F_a$= Orbit$\{\xi_a\}$, $F_a(\xi_a)=v_0 \in V_1^{U_2(E)}$ for any $a\in F^{\times}$,
   \item[2.]  $\supp  R$= Orbit$\{\eta\}$, $R(\eta)=v_1 \in V_1^{H_1}$,
   \item[3.] $\supp  S$= Orbit$\{\delta\}$, $S(\delta)=v_2 \in V_1^{H_2}$.
   \end{itemize}
\begin{lemma}\label{theactionofpi}
   For $t_1, t_2, r, r_1 \neq r_2 \in F^{\times}$, we have
\[ \left\{\begin{array}{lcr}
    \pi_0(h(r))F_{ar^2}=\psi(r^{-2})F_a & & (\textrm{I}) \\
     \pi_0(h'(t))F_{at^{-1}}=F_a &  & (\textrm{II}) \\
     \pi_0(u(b))F_a=\phi^a(b)F_a &  & (\textrm{III})\\
    \pi_0(h(r))R=R &  & (\textrm{IV})
    \end{array}\right. \quad \textrm{  and } \quad
    \left\{\begin{array}{lcr}
    \pi_0(h'(t))R=\psi(t^{-1})R  &  & (\textrm{V})\\
    \pi_0(u(b))R=R &  & (\textrm{VI})  \\
    \pi_0\big(h(r)\big)S=S &  & (\textrm{VII})\\
    \pi_0(h'(t))S=\psi(t^{-1})S &  & (\textrm{VIII}) \\
    \pi_0(u(b))S=S  & & (\textrm{IX})
     \end{array}\right.\]
\end{lemma}
 \begin{proof}
 Firstly  we know that $\supp (\pi_0( h(r)) F_{ar^2})  \subseteq \{ (g, \psi) | g\in M, \det(g) \neq 0, \psi \in X_F\}$.
   Fix $i \in E^{\times}$ such that $\nnn_{E/F}(i)=r$. Then
  \[\pi_0(h(r))  F_{ar^2} (\xi_b)
= F_{ar^2}\Big(\begin{pmatrix}
  r& 0\\
   0&r\end{pmatrix}, \phi^b\Big)
=F_{ar^2}\Big(\nnn_{E/F}(r){\begin{pmatrix}
  i^{-1}& 0\\
   0& i^{-1}\end{pmatrix}}\begin{pmatrix}
  1& 0\\
   0&1 \end{pmatrix}{\begin{pmatrix}
  i^{-q}& 0\\
   0& i^{-q}\end{pmatrix}}, \phi^b\Big)\]
$$=\pi_1\Big(\begin{pmatrix}
  i^{-1}& 0\\
   0& i^{-1}\end{pmatrix}\Big)\circ F_{ar^2}(\xi_{br^2})
  = \left\{\begin{array}{cr}
     \pi_1\Big(\begin{pmatrix}
  i^{-1}& 0\\
   0& i^{-1}\end{pmatrix}\Big) v_0  =\pi_1\Big(\begin{pmatrix}
  i^{-1}& 0\\
   0& i^{-1}\end{pmatrix}\Big) v_0
=\Psi(i^{-2})v_0
=\psi(r^{-2})F_a(\xi_a)   &  \textrm{ if } b=a, \\     0  &  \textrm{ otherwise}.
     \end{array}\right. $$
Hence (I) follows.  Similarly,   $\pi_0(h'(t))F_{at^{-1}}(\xi_a)=F_{at^{-1}}(\xi_{at^{-1}})=v_0=F_a(\xi_a)$.  So we obtain (II), and (III) is clear.\\
 For (IV),  by (\ref{formula1})--- (\ref{formula4}), we have $\pi_0(h(r) )R (\eta)
=R\Big(\begin{pmatrix}
  r& 0\\
   0&0\end{pmatrix}, \phi\big)
=R\big({\begin{pmatrix}
  i& 0\\
   0& i^{-1}\end{pmatrix}}\begin{pmatrix}
  1& 0\\
   0&0 \end{pmatrix}{\begin{pmatrix}
  i^q& 0\\
   0& i^{-q}\end{pmatrix}}, \phi^a\Big)$
$=\pi_1\Big(\begin{pmatrix}
  i& 0\\
   0& i^{-1}\end{pmatrix}\Big)\circ R(\eta)
=\pi_1\Big(\begin{pmatrix}
  i& 0\\
   0& i^{-1}\end{pmatrix}\Big) v_1
=v_1
=R(\eta),$ which implies our (IV).
Now we let $x\in E^{\times}$ satisfying $\nnn_{E/F}(x)=t$. Then $$\pi_0(h'(t))R(\eta)
=R\Big(\begin{pmatrix}
  1& 0\\
   0&0\end{pmatrix}, \phi^{t^{-1}}\Big)
=R\Big(\nnn_{E/F}(x){\begin{pmatrix}
 x^{-1}& 0\\
   0& 1\end{pmatrix}}\begin{pmatrix}
  1& 0\\
   0&0\end{pmatrix}{\begin{pmatrix}
  x^{-q}& 0\\
   0& 1\end{pmatrix}}, \phi^{\nnn_{E/F}(x^{-1})}\Big)
=\Psi(x^{-1}) R(\eta)= \psi(t^{-1})R(\eta);$$
thus we obtain (V). The following (VI), (VII) and (IX) are easy to verify. However $$\big(\pi_0(h'(t))S\big)(\delta)
=S\Big(\begin{pmatrix}
  0& 0\\
   0&0\end{pmatrix}, \phi^{t^{-1}}\Big)=S\Big(\nnn_{E/F}(x){\begin{pmatrix}
 x^{-1}& 0\\
   0& 1\end{pmatrix}}\begin{pmatrix}
  0& 0\\
   0&0\end{pmatrix}{\begin{pmatrix}
  x^{-q}& 0\\
   0& 1\end{pmatrix}}, \phi^{\nnn_{E/F}(x^{-1})}\Big)=\Psi(x^{-1}) S(\delta)= \psi(t^{-1})S(\delta);$$ in this way we verify  (VIII).
 \end{proof}
\begin{corollary}\label{thetraceofpi0}
For $t_1, t_2, r, r_1 \neq r_2 \in F^{\times}$, we have
 \begin{itemize}
   \item[(1)] $\tr \pi_0  \begin{pmatrix}
  r& 0\\
   0&r \end{pmatrix}=(q+1) \psi(r^{-2})$.
   \item[(2)]  $\tr \pi_0 \begin{pmatrix}
  r_1& 0\\
   0&r_2 \end{pmatrix}=2 \psi(r_1^{-1}r_2^{-1})$.
   \item[(3)] $\tr \pi_0 \begin{pmatrix}
  r& 1\\
   0&r \end{pmatrix}=\psi(r^{-2})$.
   \end{itemize}
 \end{corollary}

\begin{proposition}
 $\check{\pi}_0\simeq \psi\cdot \Ind_B^G1_G$.
 \end{proposition}
 \begin{proof}
By Corollary \ref{thetraceofpi0}, we know $\Res_B^G\pi_0\simeq (2\sigma_{\psi^{-1},\psi^{-1}})\oplus (\psi^{-2}\otimes \sigma) \simeq \Res_B^G(\psi^{-1} \cdot \Ind_B^G1_B)$ for irreducible representations $\sigma_{\psi^{-1}, \psi^{-1}}$, $\psi^{-1} \otimes \sigma$ of $B$ defined in Theorem \ref{representationsofB}
 and the isotypic components $2\sigma_{\psi^{-1},\psi^{-1}}$ are spanned by the functions $R, S$. By Proposition \ref{Rest1}, we have (i) $\pi_0\simeq \psi^{-1}\cdot \Ind_B^G1_G$,  or (ii) $\pi_0\simeq ( 2 \psi^{-1}\cdot 1_G )\oplus \pi_{\theta}$
for certain regular character $\theta$ of $E^{\times}$. But $\big(\pi_0(\omega)S\big)(\delta)= q^{-2}\sum_{n\in M} S(n,\phi)\phi(B(0,n))= q^{-2}S(\delta)$
 so that $
\pi_0(\omega)S\neq S$;
 this means that $\pi_0$  has  at most only one isotypic component $\psi^{-1} \cdot 1_G$. Therefore the above case (ii) is impossible.
\end{proof}

\subsection{The representation $(\pi_0, W[\pi_1])$ II}
In this subsection, let $\pi_1=\Psi\cdot \St_H$, where  $\quad \Psi=\Psi^q \in \Irr(E^{\times})$, $\Psi=\psi\circ \nnn_{E/F}$ for some $\psi\in \Irr(F^{\times})$. The vector space $W[\pi_1]$ is generated by the functions:  $F_a, R: M\times X_F \longrightarrow V_1$ for any $a\in F^{\times}$. They  all satisfy the equality
(\ref{star2}), and
  \begin{itemize}
   \item[1.]   $\supp(F_a)=$ Orbit$\{\xi_a\}$, $F_a(\xi_a)=v_0 \in V_1^{U_2(E)}$ for any $a\in F^{\times}$,
   \item[2.]   $\supp(R)=$ Orbit$\{\eta\}$, $R(\eta)=v_1=q^2f_{\Psi, \Psi}-g_{\Psi,\Psi} \in V_1^{H_1}$ by  Lemma \ref{dimension}.
   \end{itemize}
Similarly as in Section  \ref{determinetherepresentationI},  we  obtain:
\begin{lemma}
\[ \left\{\begin{array}{lr}
    \pi_0(h(r))F_{ar^2}=\psi(r^{-2})F_a & (\textrm{XI}) , \\     \pi_0(h'(t))F_{at^{-1}}=F_a &  (\textrm{XII}),
    \\ \pi_0(u(b))F_a=\phi^a(b)F_a &  (\textrm{XIII}),
    \end{array}\right. \qquad \textrm{  and  } \qquad
    \left\{\begin{array}{lr}
    \pi_0(h(r))R=R & (\textrm{XIV}),\\
    \pi_0(h'(t))R=\psi(t^{-1})R &  (\textrm{XV}),\\
    \pi_0(u(b))R=R &  (\textrm{XVI}).
     \end{array}\right.\]
\end{lemma}
\begin{lemma}
Let $M^{(1)}=\{ n \in M| rank \,n=1\}$. Then:
\begin{itemize}
\item[(a)] $ M^{(1)}=\{   \begin{pmatrix}
  s& 0\\
   0& 0\end{pmatrix}| s\in F^{\times}\} \cup \{ s\begin{pmatrix}
  \nnn_{E/F}(b)& b\\
   b^q& 1\end{pmatrix}| s\in F^{\times}, b\in E\}.$
\item[(b)] $s\begin{pmatrix}
  \nnn_{E/F}(b)& b\\
   b^q& 1\end{pmatrix}=u(b) \begin{pmatrix}
  0& 0\\
   0& s\end{pmatrix}u(b)^{\star}$ for $s\in F^{\times}$.
\end{itemize}
\end{lemma}
\begin{proof}
See \cite[ p.246---247]{Andr}.
\end{proof}
Now, let us consider
 $$\big(\pi_0(\omega)R\big)(\eta)=-q^{-2} \sum_{n\in M} R(n, \phi)\phi\bigg(B\Big(\begin{pmatrix}
  1& 0\\
   0& 0\end{pmatrix}, n \Big)\bigg)
\stackrel{\supp R=M^{(1)}}{=}-q^{-2} \sum_{n\in M^{(1)}} R(n, \phi)\phi\bigg(B\Big(\begin{pmatrix}
  1& 0\\
   0& 0\end{pmatrix}, n \Big)\bigg)$$
$$=-q^{-2}\sum_{s\in F^{\times}}\Bigg[ R\bigg(\begin{pmatrix}
  s& 0\\
   0& 0\end{pmatrix}, \phi\bigg) \phi\Bigg(B\bigg(\begin{pmatrix}
  1& 0\\
   0& 0\end{pmatrix},\begin{pmatrix}
  s& 0\\
   0& 0\end{pmatrix}\bigg)\Bigg)+ \sum_{b\in E} R\bigg(u(b)\begin{pmatrix}
  0& 0\\
   0& s\end{pmatrix}u(b)^{\star}, \phi\bigg)\phi\Bigg(B\bigg(\begin{pmatrix}
  1& 0\\
   0& 0\end{pmatrix},u(b)\begin{pmatrix}
  0& 0\\
   0& s\end{pmatrix}u(b)^{\star}\bigg)\Bigg)\Bigg]$$
$$=-q^{-2} \sum_{s\in F^{\times}} [  \pi_0\Big( h(s)\Big)R\big(\eta\big) +  \sum_{b\in E} \pi_1 \Bigg(u(b)\begin{pmatrix}
  0& -1\\
   1& 0\end{pmatrix}\Bigg) R\bigg( \begin{pmatrix}
 s& 0\\
   0& 0\end{pmatrix}, \phi\bigg)\phi(s)]
 =-q^{-2} \sum_{s\in F^{\times}} [ v_1 +  \sum_{b\in E} \pi_1 \Bigg(u(b)\begin{pmatrix}
  0& -1\\
   1& 0\end{pmatrix}\Bigg) R\bigg( \begin{pmatrix}
 s& 0\\
   0& 0\end{pmatrix}, \phi\bigg)\phi(s)] $$
 $$=-q^{-2} \sum_{s\in F^{\times}}\pi_1(u(b)) [ v_1 + \phi(s) \sum_{b \in E} \pi_1\Bigg(\begin{pmatrix}
  0&-1\\
   1& 0\end{pmatrix} \Bigg)v_1]=-q^{-2}\pi_1(u(b))[(q-1)v_1- q^2\pi_1\Bigg(\begin{pmatrix}
  0& -1\\
   1& 0\end{pmatrix} \Bigg)v_1] \neq R(\eta).$$
It follows that
\begin{equation}\label{formula22}\tag{XVII}
\pi_0(\omega)R\neq R
 \end{equation}
\begin{proposition}
$\check{\pi}_0\simeq \psi\cdot \St_G$.
\end{proposition}
\begin{proof}
By the formulas (XI)---(XVI), we obtain $\Res_B^G\pi_0=\Res_B^G(\psi^{-1}\cdot\St_G)$. Consequently  by the formula (XVII), $\pi_0$ has no $\psi^{-1}\cdot 1_G$ isotypic component.  Comparing  this with Proposition \ref{Rest1} gives  the result.
\end{proof}

\subsection{ The representation $(\pi_0, W[\pi_1])$ III}

In this subsection, let $\pi_1=\Pi_{\Lambda, \Sigma}$ for $ \Lambda\neq \Sigma$ and $\Lambda=\lambda\circ\nnn_{E/F}, \Sigma=\sigma\circ\nnn_{E/F}\in \Irr(E^{\times})$. The vector space $V_1^{H_1}$($=V_1^{N'}$) is generated by the two functions $f_{\Lambda, \Sigma}, g_{\Lambda, \Sigma}$ defined in  Lemma \ref{dimension}.
Let $\Delta: M\times X_F \longrightarrow  V_1$  satisfying $(\ref{star2})$ and $ \supp(\Delta)$= Orbit$\{\eta\}$, $\Delta(\eta)=f_{\Lambda, \Sigma}$. Then
$$(i) \qquad \Big(\pi_0\big(h(r)\big)\Delta\Big)(\eta)
=\Delta\Big(\begin{pmatrix}
  r& 0\\
   0& 0\end{pmatrix}, \phi\Big)
=\Delta\Big({\begin{pmatrix}
 x& 0\\
   0&x^{-1}\end{pmatrix}}\begin{pmatrix}
  1& 0\\
   0&0\end{pmatrix}{\begin{pmatrix}
  x^{q}& 0\\
   0& x^{-q}\end{pmatrix}}, \phi\Big)
= \Pi_{\Lambda, \Sigma}\Big(\begin{pmatrix}
  x& 0\\
   0&x^{-1}\end{pmatrix}\Big) \Delta(\eta)$$
$$=\Pi_{\Lambda, \Sigma}\Big(\begin{pmatrix}
  x& 0\\
   0&x^{-1}\end{pmatrix}\Big)[ f_{\Lambda,\Sigma}]
=\Lambda(x) \Sigma(x^{-1}) f_{\Lambda,\Sigma}
=\lambda(r)\sigma(r^{-1}) f_{\Lambda,\Sigma} \textrm{  for } \nnn_{E/F}(x)=r.$$
$$(ii) \qquad\qquad \Big(\pi_0\big(h'(t)\big)\Delta\Big)(\eta)=
\Delta\Big(\begin{pmatrix}
  1& 0\\
   0& 0\end{pmatrix}, \phi^{t^{-1}}\Big)
=\Delta\Big(\nnn_{E/F}(x){\begin{pmatrix}
 x^{-1}& 0\\
   0&1\end{pmatrix}}\begin{pmatrix}
  1& 0\\
   0&0\end{pmatrix}{\begin{pmatrix}
  x^{-q}& 0\\
   0& 1\end{pmatrix}}, \phi^{\nnn_{E/F}(x^{-1})}\Big)$$
$$= \Pi_{\Lambda, \Sigma}\Big(\begin{pmatrix}
  x^{-1}& 0\\
   0&1\end{pmatrix}\Big)[\Delta(\eta)]
= \Pi_{\Lambda, \Sigma}\Big(\begin{pmatrix}
  x^{-1}& 0\\
   0&1\end{pmatrix}\Big)(f_{\Lambda,\Sigma})
=\Lambda(x^{-1}) f_{\Lambda,\Sigma}
=\lambda(t^{-1})f_{\Lambda,\Sigma}
=\lambda(t^{-1}) \Delta(\eta) \textrm{ for }\nnn_{E/F}(x)=t.$$

By the above (i) and (ii), we obtain that  $\pi_0(h(r))\Delta = \lambda(r)\sigma(r^{-1})\Delta$, $\pi_0(h'(t))\Delta = \lambda(t^{-1})\Delta$. In particular,  $\pi_0\Big(\begin{pmatrix}
 t_1& 0\\
   0&t_2\end{pmatrix}\Big) \Delta = \lambda(t_2^{-1})\sigma(t_1^{-1}) \Delta$. It follows that $\Hom_G\big(\pi_0, \Ind_B^G(\lambda^{-1}\otimes \sigma^{-1})\big) \simeq \Hom_G\big(\pi_0, \Ind_B^G(\sigma^{-1}\otimes \lambda^{-1})\big) \neq 0$.  Since $\dim_{\C} \pi_0=q+1$, surely $\pi_0 \simeq \pi_{\lambda^{-1},\sigma^{-1}}$.

\subsection{The representation $(\pi_0, W[\pi_1])$ IV}
In this subsection, let $\pi_1=\Pi_{\Lambda, \Sigma}$ where $ \Lambda\neq \Sigma\in \Irr(E^{\times})$ and $\Lambda=\Sigma^q, \Sigma=\Lambda^q$.
We start with  recalling some explicit models for certain representations(cf. \cite{Andr}).

\subsubsection{I. Model for $\Pi_{\Lambda,\Sigma}$ }
  By \cite[p.21, Definition 2 ]{Andr}, $\Pi_{\Lambda,\Sigma}$ can be realized in the  vector space $V_1$ spanned  by all the  functions $v:E^2\times E^{\times} \longrightarrow \C$ such that
    \[ \left\{\begin{array}{lr}
    v(a(e_1,e_2); a^{-1}b^{-1}e_3)=\Lambda(a)\Sigma(b) v(e_1,e_2;e_3)& \cdots (\star),  \\    \big(\Pi_{\Lambda,\Sigma}(h)v\big)(e_1,e_2;e_3)=v((e_1,e_2)h; e_3\det(h)^{-1}) & \cdots (\star\star),
     \end{array}\right.\]
for  $ e_1,e_2\in E; a,b, e_3\in E^{\times}; h\in H$.
\subsubsection{ II. Model for $\pi_{\Lambda}$}
 Let $\pi_{\Lambda}$  be  a cuspidal representation of $G$ corresponding to a regular character $\Lambda$ of $E^{\times}$. Invoking  \cite[ p.53, Proposition 4]{Andr}, we know  that $\pi_{\Lambda}$ can be realized in the vector space $\C[X_F]$ as follows:
 \begin{equation}\label{formulaII1}\tag{1}
\pi_{\Lambda}\Big(\begin{pmatrix}
  r& 0\\
   0&r\end{pmatrix}\Big)f=\Lambda(r) f
 \end{equation}
  \begin{equation}\label{formulaII2}\tag{2}
\Big(\pi_{\Lambda}\Big(\begin{pmatrix}
  1& 0\\
   0&t\end{pmatrix}\Big)f\Big)(\psi)=f(\psi^{t^{-1}})
 \end{equation}
   \begin{equation}\label{formulaII3}\tag{3}
\Big(\pi_{\Lambda}\Big(\begin{pmatrix}
  1& s\\
   0&1\end{pmatrix}\Big)f\Big)(\psi)=\psi(s) f(\psi)
 \end{equation}
    \begin{equation}\label{formulaII4}\tag{4}
\Big(\pi_{\Lambda}(\omega)f\Big)(\psi)=-q^{-1} \sum_{y\in E^{\times}} \psi(\Tr_{E/F}(y)) \Lambda(y) f(\psi^{N_{E/F}(y)})
 \end{equation}
where $\psi\in X_F, t, r\in F^{\times}, s\in F$.\\

\subsubsection{ The representation $(\pi_0,  W[\pi_1])$}
 The vector space $W[\pi_1]$ is generated by those functions $F_a: M\times X_F \longrightarrow V_1$ satisfying (\ref{star2}) and $\supp(F_a)$= Orbit$\{\xi_a\}$, $F_a(\xi_a)= v_1 \in V_1^{U_2(E)}$ for any $a\in F^{\times}$. Using the  above model, we choose an element $v_1$ as follows: $v_1: E^2 \times E^{\times} \longrightarrow \C$  satisfies above ($\star$) and
  \begin{itemize}
   \item[(1)] $\supp (v_1)=\cup_{u\in U}$ Orbit $\{(1 , ue_{-1}; 1)\}$,
   \item[(2)] $v_1(1, ue_{-1}; 1)=\Lambda(u)$,
   \end{itemize}
where Orbit$\{(1, ue_{-1}; 1)\}=\{ (a, ue_{-1}a; a^{-1}b^{-1}) \in E^2 \times E^{\times}| a,b \in E^{\times}\}$, and  $e_{-1}$ is a fixed element in $E^{\times}$ such that $\nnn_{E/F}(e_{-1})=-1$.
\begin{lemma}
The above constructed $v_1$ belongs to $V_1^{U_2(E)}$.
\end{lemma}
\begin{proof}
1) For $g=\begin{pmatrix}
  u& 0\\
   0&1\end{pmatrix} \in \U_2(E)$, we have
   $\supp(g\cdot v_1)=\supp (v_1),$
    and
    $g\cdot v_1(1, u_0e_{-1}; 1)=v_1(u, u_0e_1; u^{-1})=\Lambda(u) v_1(1, u_0u^{-1}e_{-1}; 1)=\Lambda(u)\Lambda(u_0u^{-1})=\Lambda(u_0)=v_1(1, u_0e_{-1}; 1);$
    thus $g\cdot v_1=v_1$.\\
2) For $g=\begin{pmatrix}
  a& b\\
   -b^q&a^q\end{pmatrix} \in U_2(E)$, we have
  $ \supp (g\cdot v_1)=\supp(v_1), $
  and
   $g\cdot v_1(1, ue_{-1}; 1)=v_1(a-b^que_{-1}, b+ue_{-1}a^q; 1)=v_1(a-b^que_{-1}, ue_{-1}\big( a-b^que_{-1}\big)^q;1)$
  $ =\Lambda(a-b^que_{-1})v_1(1, ue_{-1}(a-b^que_{-1})^{q-1}; a-b^que_{-1})
   =\Lambda(a-b^que_{-1})\Lambda(u(a-b^que_{-1})^{q-1})\Sigma^{-1}(a-b^que_{-1})$
   $=\Lambda(u)=v_1(1, ue_{-1};1);$
   therefore $g\cdot v_1=v_1$ in this case.
\end{proof}

We define an intertwining operator between $\pi_{\Lambda^{-1}}$ and $\pi_0$ by\\
 $$j: \pi_{\Lambda^{-1}} \longrightarrow W[\pi_1]; f \longmapsto j(f)=\sum_{a\in F^{\times} } f(\phi^a) F_a, \textrm{ i.e. } j(f)(\xi_a)= f(\phi^a) v_1.$$
Claim: $j\big(\pi_{\Lambda^{-1}}(g)f\big) = \pi_0(g) j(f) \textrm{ for } g\in G.$\\
Proof: (1) Let $g=\begin{pmatrix}
  x&y\\
   0&z\end{pmatrix}\in B$.\\
   $$\pi_0\bigg(\begin{pmatrix}
  x&y\\
   0&z\end{pmatrix}\bigg) j(f)(\xi_a)= \sum_{t\in F^{\times}} f(\phi^t) \pi_0\bigg(\begin{pmatrix}
  x&y\\
   0&z\end{pmatrix}\bigg) F_t(\xi_a)$$
  $$=\sum_{t\in F^{\times}} f(\phi^t) \pi_0\bigg( \begin{pmatrix}
  x&0\\
   0&x^{-1}\end{pmatrix} \begin{pmatrix}
  1&0\\
   0&xz\end{pmatrix} \begin{pmatrix}
  1&x^{-1}y\\
   0&1\end{pmatrix}\bigg) F_t(\xi_a)$$
   $$= \sum_{t\in F^{\times}} f(\phi^t) F_t\bigg( \begin{pmatrix}
  x&0\\
   0&x\end{pmatrix}, \phi^{ax^{-1}z^{-1}}\bigg) \phi^a(yz^{-1})$$
$$\stackrel{\nnn_{E/F}(r)=x}{=} \sum_{t\in F^{\times}} f(\phi^t) F_t\bigg( \begin{pmatrix}
  r^{-1}&0\\
   0&r^{-1}\end{pmatrix} \begin{pmatrix}
  1&0\\
   0&1\end{pmatrix} \begin{pmatrix}
  r^{-q}&0\\
   0&r^{-q}\end{pmatrix} \nnn_{E/F}(r^{2}), \phi^{\nnn_{E/F}(r^{-2}) axz^{-1}}\bigg) \phi^a(yz^{-1})$$
$$=\sum_{t\in F^{\times}} f(\phi^t) \pi_1\bigg( \begin{pmatrix}
  r^{-1}&0\\
   0&r^{-1}\end{pmatrix}\bigg) F_t\bigg( \begin{pmatrix}
  1&0\\
   0&1\end{pmatrix}, \phi^{axz^{-1}}\bigg) \phi^a(yz^{-1})$$
$$=f(\phi^{axz^{-1}}) [\pi_1\bigg( \begin{pmatrix}
  r^{-1}&0\\
   0&r^{-1}\end{pmatrix}\bigg) v_1] \phi^a(yz^{-1})$$
$$= f(\phi^{axz^{-1}})\Lambda(r^{-1}) \Sigma(r^{-1})v_1 \phi^a (yz^{-1})=f(\phi^{axz^{-1}}) \Lambda(x^{-1}) \phi^a(yz^{-1})v_1.$$
So $$\pi_0\Big(\begin{pmatrix}
  x&y\\
   0&z\end{pmatrix}\Big)j(f)=\sum_{a\in F^{\times}} f(\phi^{ax^{-1}z}) \Lambda(x^{-1}) \phi^a(yz^{-1})F_a.$$
On the other hand,
$$j(\pi_{\Lambda^{-1}} \Big(\begin{pmatrix}
  x&y\\
   0&z\end{pmatrix}\Big)f)= \sum_{a\in F^{\times}} [\pi_{\Lambda^{-1}} \Big( \begin{pmatrix}
  x&y\\
   0&z\end{pmatrix} \Big)f](\phi^a) F_a$$
$$=\sum_{a\in F^{\times}}[ \pi_{\Lambda^{-1}}\Bigg( \begin{pmatrix}
  x&0\\
   0&x\end{pmatrix}  \begin{pmatrix}
  1&0\\
   0&x^{-1}z\end{pmatrix}  \begin{pmatrix}
  1&x^{-1}y\\
   0&1\end{pmatrix}\Bigg) f](\phi^a) F_a$$
$$= \sum_{a\in F^{\times}} \Lambda(x^{-1}) f(\phi^{axz^{-1}}) \phi^{axz^{-1}}(x^{-1}y) F_a$$
$$=\sum_{a\in F^{\times}} \Lambda(x^{-1}) f(\phi^{axz^{-1}}) \phi^a(yz^{-1}) F_a.$$
(2) Let $g=\omega$.
 $$\Big(\pi_0(\omega) j(f)\Big) (\xi_a)= -q^{-2}  \sum_{n\in M} \phi^a (B(\id_H, n)) j(f) (n, \phi^a)$$
  $$\stackrel{ \textrm{ consider }  \supp\big( j(f)\big)}{=} -q^{-2}  \sum_{n\in M, \det n\neq 0} \phi^a (B(\id_H, n)) j(f) (n, \phi^a)$$
 $$= -q^{-2} |U_2(E)|^{-1} \sum_{h\in H}\phi^{a}\Big(B\big(\id_H,\nnn_{E/F}(\det(h^{-1})) h h^{\star}\big)\Big)j(f)(\nnn_{E/F}(\det(h^{-1}))hh^{\star}, \phi^a)$$
 $$= -q^{-2} |U_2(E)|^{-1} \sum_{h\in H}\phi^{a}\Big(B\big(\id_H,\nnn_{E/F}(\det(h^{-1})) h h^{\star}\big)\Big)\pi_1(h)v_1[j(f)\big(\xi_{a \nnn_{E/F}(\det(h^{-1}))}\big)]$$
 $$= -q^{-2} |U_2(E)|^{-1} \sum_{h\in H}\phi^{a}\Big(B\big(\id_H,\nnn_{E/F}(\det(h^{-1})) h h^{\star}\big)\Big)\pi_1(h) f(\phi^{a\nnn_{E/F}(\det(h^{-1}))}) $$
$$\stackrel{ \textrm{ replace } h \textrm{ by }h^{-1}\det(h)}{=}- q^{-2} |U_2(E)|^{-1} \sum_{h\in H}\phi^{a}\Big(B\big(\id_H, h^{-1} (h^{\star})^{-1}\big)\Big)f(\phi^{a\nnn_{E/F}(\det(h))})\pi_1(h^{-1}\det(h))v_1.$$
 Let $$\kappaup_a= -q^{-2} |U_2(E)|^{-1} \sum_{h\in H}\phi^{a}\Big(B\big(\id_H, h^{-1} (h^{\star})^{-1}\big)\Big)f(\phi^{a\nnn_{E/F}(\det(h))})\Big(\pi_1(h^{-1}\det(h))v_1\Big)(1, e_{-1}; 1)$$
 and
 $$\kappaup_a^s=- q^{-2} |U_2(E)|^{-1} \sum_{h\in H,\nnn_{E/F}(\det(h))=s }\phi^{a}\Big(B\big(\id_H, h^{-1} (h^{\star})^{-1}\big)\Big)f(\phi^{a\nnn_{E/F}(\det(h))})\Big(\pi_1(h^{-1}\det(h))v_1\Big)(1, e_{-1}; 1)
\textrm{  for any } s\in F^{\times}.$$
Then $\kappaup_a=\sum_{s\in F^{\times}}\kappaup_a^s$, and
 $$\kappaup_a^1 = -q^{-2} |U_2(E)|^{-1} \sum_{h\in \mathcal {M}}\phi^{a}\Big(\nnn_{E/F}(\alpha)+\nnn_{E/F}(\beta )+\nnn_{E/F}(\gamma )+\nnn_{E/F}(\delta ) \Big)f(\phi^a) v_1\big(\delta - \gamma e_{-1}, -\beta + \alpha e_{-1}; \det(h^{-1})\big)$$
 \begin{equation}\label{ka1}
 = -q^{-2} |U_2(E)|^{-1} \sum_{h\in \mathcal {M}}\phi^{a}\Big(\nnn_{E/F}(\alpha)+\nnn_{E/F}(\beta )+\nnn_{E/F}(\gamma )+\nnn_{E/F}(\delta ) \Big)f(\phi^a) \Lambda^{q-1}(\delta- \gamma e_{-1})\Lambda^{-q}(\alpha\delta-\beta\gamma)  v_1\big(1, \frac{ -\beta + \alpha e_{-1}}{\delta - \gamma e_{-1}}; 1),
 \end{equation}
 where $\mathcal {M}=\{h=\begin{pmatrix}
  \alpha & \beta \\
   \gamma &\delta \end{pmatrix}\in H| \nnn_{E/F}(\alpha\delta-\beta\gamma)=1; -\beta+\alpha e_{-1}=u_1e_{-1}(\delta-\gamma e_{-1}),\delta-\gamma e_{-1}\neq 0 \textrm{ for  some } u_1\in U\}$;\\
By the equations in $\{ \alpha\delta-\beta\gamma=u_2;  -\beta+\alpha e_{-1}=u_1e_{-1}(\delta-\gamma e_{-1})$ and $ \delta-\gamma e_{-1}=z$ for $u_1, u_2 \in U, z\in E^{\times}\}$, we change the variables $\alpha, \beta, \gamma, \delta$ by $u_1,u_2,z,\gamma$. Note that this is reasonable.

By $-\beta + \alpha e_{-1}=u_1e_{-1}\big( \delta-\gamma e_{-1}\big)$, we get $-\beta e_{-1}^{-1} + u_1 \gamma e_{-1}=u_1 \delta -\alpha.$  Then
$\big(-\beta e_{-1}^{-1} + u_1 \gamma e_{-1}\big) \big( -\beta e_{-1}^{-1} + u_1 \gamma e_{-1}\big)^q= \big( u_1 \delta -\alpha\big) \big( u_1 \delta -\alpha\big)^q.$
By calculation, we have
$\nnn_{E/F}(\alpha) +\nnn_{E/F}(\beta) + \nnn_{E/F}(\gamma) + \nnn_{E/F}(\delta)$
$=\tr_{E/F}\Big( u_1\big( \alpha^q \delta -\gamma \beta^q e_{-1}^{1-q}\big)\Big).$
Set $A=\alpha^q \delta -\gamma \beta^q e_{-1}^{1-q}$. Now we  consider
$$u_1 u_2^{-1} z=u_1 u_2^q z
=u_1 \Big( \alpha^q \delta^{q+1}-\beta^q \gamma^q \delta -\alpha^q \delta^q \gamma e_{-1} + \beta^q \gamma^{q+1} e_{-1}\Big),$$
and also
$$u_1 A z^q=u_1[ \Big( \alpha^q \delta-\gamma\beta^q e_{-1}^{1-q}\Big)\Big( \delta^q -\gamma^q e_{-1}^q\Big)]$$
$$=u_1[\alpha^q \delta^{q+1}- \alpha^q \delta\gamma^q e_{-1}^q-\delta^q \beta^q \gamma e_{-1}^{1-q} + \beta^q \gamma^{q+1}e_{-1}]
= u_1[ u_2^{-1} z + \beta^q \gamma^q \delta +\alpha^q \delta^q \gamma e_{-1} - \alpha^q \delta\gamma^q e_{-1}^q-\delta^q \beta^q \gamma e_{-1}^{1-q}]$$
$$=u_1 u_2^{-1} z + u_1e_{-1}^{-q} (-\beta + \alpha e_{-1})^q \big( \delta^q \gamma e_{-1}-\delta \gamma^q e_{-1}^q\big)=u_1 u_2^{-1} z+ z^q\big( \delta^q \gamma e_{-1}-\delta \gamma^q e_{-1}^q\big).$$
So
$$u_1A= u_1 u_{2}^{-1} z^{1-q} +\Big( \delta^q \gamma e_{-1} -\delta \gamma^q e_{-1}^q\Big).$$
In this way, we obtain
$$\tr_{E/F}(u_1A)=\tr_{E/F}(u_1u_2^{-1}z^{1-q}).$$
Hence
$$(\ref{ka1})=- q^{-2} |U_2(E)|^{-1} \sum_{u_1,u_2\in U, z\in E^{\times},\gamma\in E}\phi^{a}\big(\tr_{E/F}(u_1u_2^{-1}z^{1-q})\big)\Lambda^{-1}(u_1u_2^{-1}z^{1-q})f(\phi^a) $$
$$\stackrel{u_1=x_1^{1-q},u_2= x_2^{q-1}}=- q^{-2} |U_2(E)|^{-1}\frac{1}{(q-1)^2} \sum_{x_1, x_2, z\in E^{\times},\gamma\in E}\phi^{a}\bigg(\tr_{E/F}\Big(\big(x_1x_2z\big)^{1-q}\Big)\bigg)\Lambda^{-1}\bigg(\big(x_1x_2z\big)^{1-q}\bigg)f(\phi^a) $$
$$=-\frac{1}{(q-1)q}\sum_{ z\in E^{\times}}\phi^{a}\bigg(\tr_{E/F}\big(z^{1-q}\big)\bigg)\Lambda^{-1}\bigg(z^{1-q}\bigg)f(\phi^a)$$
$$=-q^{-1}\sum_{y\in E^{\times}, \nnn_{E/F}(y)=1}\phi^a(\tr_{E/F}(y)) \Lambda^{-1}(y) f(\phi^{a}).  $$
Similarly we obtain $$\kappaup_a^s=-q^{-1}\sum_{y\in E^{\times}, \nnn_{E/F}(y)=s}\phi^a(\tr_{E/F}(y)) \Lambda^{-1}(y) f(\phi^{a s}).  $$
Finally
 $$\kappaup_a=-q^{-1} \sum_{y\in E^{\times}} \phi^a(\tr_{E/F}(y)) \Lambda^{-1}(y) f(\phi^{aN_{E/F}(y)}).$$
Since $V_1^{U_2(E)}$ is one-dimensional, we have
 $\big(\pi_0(\omega) j(f)\big)(\xi_a) = j(\pi_{\Lambda^{-1}}(\omega) f) (\xi_a),$ which means $\pi_0(\omega)j(f)=j(\pi_{\Lambda^{-1}}(\omega)f)$. \\
By the above (1), (2), we prove $\pi_0\simeq \pi_{\Lambda^{-1}}$.
\subsection{}
By the above discussion I---IV about the representation $(\pi_0, W[\pi_1])$,  finally we achieve the main theorem in this section:
\begin{theorem}\label{mainth2}
For the representation $(\pi, G\times H, W)$, we have the  following  decomposition:
$$ \pi\simeq \bigoplus_{\sigma\in \Irr(G)} \sigma \otimes \Bc_{E/F}(\sigma)\oplus \bigoplus_{\psi\in \Irr(F^{\times}), \Psi\in \Irr(E^{\times}),\Psi=\psi\circ\nnn_{E/F}} \psi\!\St_G\!\otimes\Psi\!\cdot\!1_H.$$
\end{theorem}

\section{The decomposition of the Weil representation of $\GL_2(K)$}\label{demp3}
\subsection{}\label{section3notation}

In this section,  we use the following notations: $G=\GL_2(K)$,  $ B=\{\begin{pmatrix}
  a& b\\
  0&d
\end{pmatrix} \in G \}
, \ N= \{
 \begin{pmatrix}
  1& b \\
  0&1
\end{pmatrix} \in G \}, \ T=\{\begin{pmatrix}
  a& 0 \\
  0&d
\end{pmatrix} \in G\},  \  Z=\{\begin{pmatrix}
  a& 0 \\
  0&a  \end{pmatrix} \in G \}$; $\Gal(K/F)=\langle \sigma \rangle$;
 \  $u(b)=
 \begin{pmatrix}
  1& b \\
  0&1
\end{pmatrix}$ for $b\in K$,\  $ h(a,d)=
 \begin{pmatrix}
  a& 0 \\
  0&d
\end{pmatrix}$ for $a,d\in K^{\times}$, \ $\omega=
 \begin{pmatrix}
  0& 1 \\
 -1&0
\end{pmatrix}\in G$.
\subsection{}\label{Weildescent1}
We recall the technique of  Weil's Galois descent  to construct a morphism from $G$ to $\GSp_8(F)$.

Let $V_0$ be a vector space over $F$ of dimension $2$, endowed with a symplectic form $\langle, \rangle_{V_0}$. Let $\{e_1, e_2\}$ be a symplectic base of $V_0$. Namely $V=V_0\otimes_F K$ is a symplectic $K$-vector space, endowed with the symplectic form $\langle, \rangle_V$ induced from $V_0$ by  scalar extension. Let us define a $\Gal(K/F)$-action on $V$  by
$$\Gal(K/F)\times K\otimes_F V_0 \longrightarrow K\otimes_F V_0; (\sigma, \sum_i k_i \otimes e_i) \longmapsto \sum_i k_i^{\sigma}\otimes e_i.$$
Let $W= V\otimes_K V\otimes_K V$, and we assign $W$ a symplectic form $\langle, \rangle_W=\langle, \rangle_{V}\otimes \langle, \rangle_V \otimes \langle, \rangle_V$. On $W$, we will consider the twisted Galois action defined by
 $$\Gal(K/F)  \times W \longrightarrow W; (\sigma, w=\sum_{i=1}^n u_i\otimes v_i \otimes w_i ) \longmapsto {}^{\sigma}w=\sum_{i=1}^n w_i^{\sigma}\otimes u_i^{\sigma} \otimes v_i^{\sigma}.$$
 We will let $W_0$ denote the set $\{ w\in W|{}^{\sigma}w=w\}$.  By calculation, each $w_0\in W_0$ may be expressed in the form
 $$w_0= x e_1\otimes e_1 \otimes e_1 + \alpha e_1\otimes e_1 \otimes e_2 + \alpha^{\sigma} e_2 \otimes e_1 \otimes e_1 + \alpha^{\sigma^2} e_1\otimes e_2 \otimes e_1$$ $$ + \beta^{\sigma^2} e_2\otimes e_1 \otimes e_2 + \beta^{\sigma} e_1\otimes e_2\otimes e_2 + \beta e_2 \otimes e_2 \otimes e_1 + y e_2 \otimes e_2\otimes e_2 \textrm{ for } x, y \in F, \alpha, \beta \in K.$$
 Every element $w_0$ of this form is well-defined by its corresponding coefficients.  For simplicity,  we write $w_0=\begin{pmatrix}
  x& \alpha  \\
  \beta &y \end{pmatrix}$ instead of the whole term.  One can check that  the restriction of $\langle, \rangle_W$ to $W_0$ defines  an  $F$-symplectic form, denoted by $\langle, \rangle_{W_0}$. More precisely,
  $$\langle w_0, w_0'\rangle_{W_0}= xy'-x'y -\tr_{K/F}(\alpha \beta') + \tr_{K/F} (\alpha' \beta ) \textrm{ for } w_0= \begin{pmatrix}
  x& \alpha  \\
  \beta &y \end{pmatrix}, w_0'=\begin{pmatrix}
  x'& \alpha' \\
  \beta' &y' \end{pmatrix}.$$
Let $\GSp(W)$ denote the group of  symplectic similitudes  of  $(W, \langle, \rangle_W)$. By  definition, there actually exists a morphism of groups
$$\bigg( \GL(V) \times \GL(V) \times \GL(V)\bigg) \rtimes S_3 \longrightarrow \GSp(W).$$
Here the group $S_3$ acts on $W$ by permutating its three variables. Now we define a twisted Galois action of $\Gal(K/F)$ on $ \GL(V) \times \GL(V) \times \GL(V)$  by
$$\Gal(K/F) \times \bigg(\GL(V) \times \GL(V) \times \GL(V)\bigg) \longrightarrow \GL(V) \times \GL(V) \times \GL(V); h=(g_1, g_2, g_3) \longmapsto {}^{\sigma} h:=(g_3^{\sigma}, g_1^{\sigma}, g_2^{\sigma}).$$
Write
$\overline{\GL(V)}=\{ h \in \GL(V) \times \GL(V) \times \GL(V)| {}^{\sigma}h=h\}.$
 Then there exists an isomorphism of groups
 $\GL(V) \longrightarrow \overline{\GL(V)}; g \longmapsto (g, g^{\sigma}, g^{\sigma^2}).$ If given
 $h\in \GL(V) \times \GL(V) \times \GL(V),  \quad w\in W=V\otimes_K V\otimes_K V,$
  one can verify that ${}^{\sigma}h\cdot {}^{\sigma}w={}^{\sigma}(h\cdot w).$
   So it induces a morphism from $\GL(V)\simeq \overline{\GL(V)}$ to $\GSp(W_0)$. By the fixed basis $\{ e_1,e_2\}$, we obtain a morphism: $ G \stackrel{i}{\longrightarrow} \GSp(W_0)$.

\subsection{}\label{Weildescent2}
We interpret the above construction of the  morphism $G \stackrel{i}{\longrightarrow} \GSp(W_0)$ in terms of the language of algebraic groups.

Let $\textbf{V}$ be the $K$-algebraic vector space associated to $V$. That is to say:
 $$\textbf{V}: \textbf{Alg}_K \longrightarrow \textbf{Vect}_K; R \longmapsto V\otimes_K R,$$ a functor from the category of  unital commutative associative $K$-algebras to the category of $K$-vector spaces. Namely $V\otimes_KR$ inherits the $R$-symplectic structure from $V$. We  define a $\Gal(K/F)$-action on $\textbf{V}$ in the following way:\\
$$\Gal(K/F) \times V \otimes_K R \longrightarrow V\otimes_K R; (\sigma, \sum_{i=1}^n v_i\otimes r_i) \longmapsto \sum_{i=1}^n v_i^{\sigma} \otimes r_i^{\sigma}.$$
Now let $\textbf{W}$ be the $K$-algebraic vector space associated to $W$, and $\textbf{W}_{0}$ the $F$-algebraic vector space associated to $W_0$.  We define a twisted $\Gal(K/F)$-action on $\textbf{W}$  in the following way:\\
$$ \Gal(K/F) \times V\otimes_KV \otimes_K V\otimes_K R \longrightarrow V\otimes_KV\otimes_KV\otimes_KR;$$
$$(\sigma, \sum_{i=1}^n u_i\otimes v_i\otimes w_i\otimes r_i) \longmapsto  \sum_{i=1}^n w_i^{\sigma}\otimes u_i^{\sigma} \otimes v_i^{\sigma}\otimes r_i^{\sigma}.$$
So $\textbf{W}_{0}$ is the  $\Gal(K/F)$-invariant algebraic scheme of $\textbf{W}$ in the following sense:
\begin{itemize}
\item[(1)] $\textbf{W} \simeq \textbf{W}_0 \times_F K.$
\item[(2)] $\textbf{W}(R)^{\Gal(K/F)} \simeq   \textbf{W}_0(R^{\Gal(K/F)})$  for any   $R \in \textbf{Alg}_K$.
\end{itemize}
On the other hand, we also define a twisted Galois action of $\Gal(K/F)$ on $\textbf{GL}_{2/K} \times \textbf{GL}_{2/K}\times \textbf{GL}_{2/K}$ as
$$\Gal(K/F) \times \Big( \GL_2(R) \times \GL_2(R) \times \GL_2(R)\Big) \longrightarrow \GL_2(R) \times \GL_2(R) \times \GL_2(R);$$
$$\big( \sigma, (g_1,g_2,g_3)\big) \longmapsto \big( g_3^{\sigma}, g_2^{\sigma}, g_1^{\sigma}\big).$$
We denote by $\textbf{H}^{\Gal(K/F)}$, the $\Gal(K/F)$-invariant algebraic  scheme of $\textbf{H}=\textbf{GL}_{2/K}\times \textbf{GL}_{2/K} \times \textbf{GL}_{2/K}$. Indeed, by  definition,
$$\textbf{H}^{\Gal(K/F)} \simeq \Res_{K/F}(\textbf{GL}_{2/K}).$$
There exists an action of $\textbf{H}^{\Gal(K/F)}$ on $\textbf{W}_0$, and it preserves the symplectic form up to the similitude factors. Thus we  obtain a morphism of algebraic group schemes:
$$\textbf{i}: \Res_{K/F}(\textbf{GL}_{2/K}) \longrightarrow \textbf{GSp}_{W_0}.$$
\subsection{}
 Let $X_0= \{ w_0=\begin{pmatrix}
  x& \alpha  \\
  0 &0 \end{pmatrix} | w_0\in W_0\}$, $Y_0=\{ w_0= \begin{pmatrix}
  0& 0  \\
  \beta &y \end{pmatrix}| w_0 \in W_0\}$. Then $X_0, Y_0$ are two vector spaces over $F$ and $W_0= X_0 \oplus Y_0$ is a complete polarization of $W_0$.   Via the morphism $i: G \longrightarrow \GSp(W_0)$, it gives rise to a $G$-action on $W_0$ by  the following formulas:
  \begin{itemize}
\item[] For $g=\begin{pmatrix}
  a& b \\
  c &d \end{pmatrix} \in G$, $ w_0= \begin{pmatrix}
  x& \alpha \\
  \beta &y \end{pmatrix}$,  write  $g\cdot w_0 =\begin{pmatrix}
  x'& \alpha' \\
  \beta' &y' \end{pmatrix}$. Then
\item[] $x'= \nnn_{K/F}(a) x + \tr_{K/F}(aa^{\sigma} b^{\sigma^2}\alpha) + \tr_{K/F}(bb^{\sigma}a^{\sigma^2}\beta)+ \nnn_{K/F}(b)y$;
\item[] $\alpha'=aa^{\sigma}c^{\sigma^2}x+(aa^{\sigma} d^{\sigma^2}\alpha+ba^{\sigma} c^{\sigma^2}\alpha^{\sigma} + ab^{\sigma} c^{\sigma^2}\alpha^{\sigma^2}) + (bb^{\sigma} c^{\sigma^2}\beta+ ab^{\sigma} d^{\sigma^2}\beta^{\sigma} + ba^{\sigma} d^{\sigma^2}\beta^{\sigma^2}) +bb^{\sigma} d^{\sigma^2}y$;
\item[] $\beta'=dd^{\sigma} b^{\sigma^2}y + (dd^{\sigma} a^{\sigma^2}\beta+cd^{\sigma} b^{\sigma^2}\beta^{\sigma}+dc^{\sigma} b^{\sigma^2}\beta^{\sigma^2})+(cc^{\sigma} b^{\sigma^2}\alpha+ dc^{\sigma} a^{\sigma^2}\alpha^{\sigma} + cd^{\sigma} a^{\sigma^2}\alpha^{\sigma^2}) + cc^{\sigma}a^{\sigma^2}x$;
\item[] $y'=\nnn_{K/F}(d)y+ \tr_{K/F}(dd^{\sigma} c^{\sigma^2}\beta)+ \tr_{K/F}(cc^{\sigma} d^{\sigma^2}\alpha)+ \nnn_{K/F}(c)x$.
\end{itemize}
We write each element $h\in \GSp(W_0)$ in the form of $h=\begin{pmatrix}
  a& b\\
  c &d \end{pmatrix}$ with $a\in \End_F(X_0), b\in \Hom_F(Y_0, X_0), c\in \Hom_F(X_0, Y_0), d\in \End_F(Y_0)$.
\begin{corollary}\label{demp3cor1}
Through the map $i: G \longrightarrow \GSp(W_0)$, the actions of $u(b), h(a,d), \omega$ on $W_0$ are described  as follows:
\begin{itemize}
\item[(1)]$i(u(b))= \begin{pmatrix}
  m& n \\
  0&m^{\vee}
\end{pmatrix}$, where  $m \begin{pmatrix}
 x& \alpha \\
  0&0
\end{pmatrix}$=$\begin{pmatrix}
  x+\tr_{K/F}(b^{\sigma^2}\alpha)& \alpha  \\
  0&0
\end{pmatrix},$\\ $n \begin{pmatrix}
 0& 0\\
  \beta &y
\end{pmatrix}= \begin{pmatrix}
  \tr_{K/F}(bb^{\sigma}\beta)+ \nnn_{K/F}(b)y & b^{\sigma}\beta^{\sigma}+ b\beta^{\sigma^2} + bb^{\sigma}y \\
  0&0
\end{pmatrix},$ $m^{\vee} \begin{pmatrix}
 0& 0 \\
  \beta &y
\end{pmatrix}= \begin{pmatrix}
  0& 0 \\
  \beta+ b^{\sigma^2}y & y
\end{pmatrix};$
\item[(2)] $i(h(a,d))=  \begin{pmatrix}
  m& 0 \\
  0& n \end{pmatrix}$, where $m\begin{pmatrix}
  x&\alpha\\
  0 &0
\end{pmatrix}= \begin{pmatrix}
  \nnn_{K/F}(a)x &aa^{\sigma}d^{\sigma^2} \alpha \\
  0& 0
\end{pmatrix}$, $n\begin{pmatrix}
  0& 0 \\
  \beta& y
\end{pmatrix}= \begin{pmatrix}
  0& 0 \\
  dd^{\sigma}a^{\sigma^2}\beta  & \nnn_{K/F}(d)y
\end{pmatrix}$;
\item[(3)] $i(\omega)= \begin{pmatrix}
  0&u\\
  v & 0
\end{pmatrix}$, where $u\begin{pmatrix}
  0& 0 \\
  \beta & y
\end{pmatrix}= \begin{pmatrix}
  y& -\beta\\
  0 &0
\end{pmatrix}$, $v\begin{pmatrix}
  x& \alpha\\
  0 & 0
\end{pmatrix}= \begin{pmatrix}
  0& 0 \\
  \alpha & -x
\end{pmatrix}.$
\end{itemize}
\end{corollary}
Let $(\rho, V)$ be the Weil representation of the  symplectic  similitude group $\GSp(W_0)$. Via the map $i$, it gives rise to a representation $(\pi,  V)$ of $G$ which  can be realized in the vector space $V= \C[Y_0\times X_F]$ of complex functions on $Y_0\times X_F$.
\begin{proposition}\label{demp3pro2}
For the representation $(\pi, G, \C[Y_0\times X_F])$, the action is determined by the following formulas:
\begin{itemize}
\item[(1)] $[\pi(u(b))F]\bigg(\begin{pmatrix}
  0& 0 \\
 \beta & y
\end{pmatrix},\psi\bigg)=\psi\Big(\tr_{K/F} (bb^{\sigma}\beta y)- \nnn_{K/F}(b)y^2- \tr_{K/F}\big(b\beta\beta^{\sigma^2}\big)\Big) F\bigg( \begin{pmatrix}
  0& 0 \\
 \beta - b^{\sigma^2}y & y
\end{pmatrix}, \psi\bigg)$;\\
\item[(2)]  $[\pi(h(a,d))F]\bigg(\begin{pmatrix}
  0& 0 \\
 \beta & y
\end{pmatrix},\psi\bigg)= \chi_q^{+} (\nnn_{K/F}(ad)) F \bigg(\begin{pmatrix}
  0& 0 \\
 \frac{\nnn_{K/F}(ad)}{dd^{\sigma}a^{\sigma^2}}\beta & \nnn_{K/F}(a)y
\end{pmatrix},\psi^{\nnn_{K/F}(ad)^{-1}}\bigg)$;\\
\item[(3)] $[\pi(\omega) F]\bigg(\begin{pmatrix}
  0& 0 \\
 \beta & y
\end{pmatrix},\psi\bigg)= q^{-2} \sum_{
 \beta' \in K,  y'\in F} F\bigg(\begin{pmatrix}
  0& 0 \\
 \beta' & y'
\end{pmatrix},\psi\bigg)\psi(yy'+ \tr_{K/F}(\beta\beta'))$.
\end{itemize}
\end{proposition}
\begin{proof}
See Appendix 1.
\end{proof}

\subsection{}\label{demp3ps}
The whole goal of this section is to determine the different isotypic components of $\pi$. We first consider the principal series representations.

Let $\alpha, \beta \in \Irr(K^{\times})$. To determine the principal series components of $\pi$, it involves to calculate the dimension of the vector space $\Hom_G(V, \Ind_B^G(\alpha\otimes  \beta))$.   Applying  Frobenius reciprocity, we see
$$\Hom_G(V, \Ind_B^G(\alpha\otimes  \beta))\simeq \Hom_T(V_N, \alpha\otimes \beta) \simeq \Hom_T(V^N, \alpha\otimes \beta).$$ Therefore we shall first describe the vector space $V^N$, and then consider the $T$-action on it. Once  we regard the action of $N$ on the vector space $V$, as described in Proposition \ref{demp3pro2} (1), we should consider the following action:
$$N \times (Y_0 \times X_F) \longrightarrow Y_0 \times X_F; \Bigg( \begin{pmatrix}
  1& b \\
 0 & 1
\end{pmatrix},  \bigg( \begin{pmatrix}
  0& 0 \\
 \beta & y
\end{pmatrix}, \psi\bigg)\Bigg) \longmapsto \bigg( \begin{pmatrix}
  0& 0 \\
 \beta-b^{\sigma^2}y & y
\end{pmatrix}, \psi\bigg).$$
The orbits of this action are following:
\begin{itemize}
\item[(i)] Orbit$\{ \xi_{(\beta,0;\psi)}\}$, where $\xi_{(\beta,0;\psi)}= \bigg(\begin{pmatrix}
  0& 0 \\
 \beta & 0
\end{pmatrix}, \psi\bigg)$ for any $\beta\in K, \psi\in X_F$;
\item[(ii)] Orbit$\{\eta_{(0,y;\psi)}\}$, where $\eta_{(0,y;\psi)}=\bigg( \begin{pmatrix}
  0& 0 \\
 0 & y
\end{pmatrix},\psi\bigg)$ for  any  $y\in F^{\times}$,   $\psi\in X_F$.
\end{itemize}
The stabilizer of the chosen  element in  each orbit is  described as follows:
 $$\Stab_N( \xi_{(\beta,0;\psi)})=N \quad \textrm{  and } \quad  \Stab_N(\eta_{(0,y;\psi)})=1_N.$$
A   function  $F$   belongs to $V^N$ if and only if it satisfies the equality:
\begin{equation}\label{demp3eq1}
\psi\big( \tr_{K/F}(bb^{\sigma}\beta y)- \nnn_{K/F}(b)y^2- \tr_{K/F}(b\beta\beta^{\sigma^2})\big) F \bigg(\begin{pmatrix}
  0& 0 \\
 \beta-b^{\sigma^2}y & y
 \end{pmatrix}, \psi\bigg)= F\bigg( \begin{pmatrix}
  0& 0 \\
 \beta & y
 \end{pmatrix}, \psi\bigg)
\end{equation}
for any $b \in K$.
\begin{proposition}\label{demp3pro3}
(1) The vector space $V^N$ is generated by the following functions:
\begin{itemize}
\item[(i)] $F_{(0,0;\psi)}, \quad$ where $\supp(F_{(0,0;\psi)} )=$ Orbit $\{\xi_{(0,0;\psi)} \}$, $F_{(0,0,\psi)}(\xi_{(0,0;\psi)})=1$ and it satisfies the equation (\ref{demp3eq1}) for any $\psi \in X_F$;
\item[(ii)] $G_{(0,y;\psi)}, \quad$ where $\supp(G_{(0,y;\psi)})=$ Orbit $\{\eta_{(0,y;\psi)}\}$, $G_{(0,y;\psi)}(\eta_{(0,y;\psi)})=1$ and it satisfies the equation (\ref{demp3eq1}) for any $y\in F^{\times}$, any $\psi \in X_F$.
\end{itemize}
(2) Let $t=h(a,d) \in T$. Then the action of $t$ on the vector space $V^N$ is  given as follows:
\begin{itemize}
\item[(i)] $\pi(t) F_{(0,0;\psi)} = \chi_q^{+}(\nnn_{K/F}(ad)) F_{(0,0;\psi^{\nnn_{K/F}(ad)})}$;
\item[(ii)] $\pi(t) G_{(0,y;\psi)} = \chi_q^{+}(\nnn_{K/F}(ad)) G_{(0,\frac{1}{\nnn_{K/F}(a)}y;\psi^{\nnn_{K/F}(ad)})}$.
\end{itemize}
\end{proposition}
\begin{proof}
1) Every element $F$ in $V^N$, that satisfies the equation (\ref{demp3eq1}), is completely determined by its values at the points in $\{\xi_{(\beta,0;\psi)},\eta_{(0,y;\psi)}\}$. Let $\delta$ be one point among them. Then $F(\delta)$ can be nonzero if and only if the coefficient on the left-hand side of the equation (\ref{demp3eq1}) is trivial over the stabilizer of $\delta$, After checking each such point, we obtain the result.\\
2) It is straightforward.
\end{proof}
Let $\Phi$ be an element  in $\Hom_T(V^N, \alpha \otimes\beta)$. Then it is determined by the following two  equations:
\begin{itemize}
\item[(1)] $\chi_q^{+}(\nnn_{K/F}(ad)) \Phi( F_{(0,0; \psi^{\nnn_{K/F}(ad)})})= \alpha(a) \beta(b) \Phi(F_{(0,0; \psi)}), \qquad    a, d \in K^{\times}$.
\item[(2)] $\chi_q^{+}(\nnn_{K/F}(ad)) \Phi( G_{(0,\frac{1}{\nnn_{E/F}(a)}y; \psi^{\nnn_{K/F}(ad)})})= \alpha(a) \beta(b) \Phi(G_{(0,y; \psi)}), \qquad  a, d \in K^{\times}. $
\end{itemize}
Now let us  define a $T$-action on the vector space $V^N$:
\[t\cdot  F_{(0,0; \psi)} :=  F_{(0,0; \psi^{\nnn_{K/F}(ad)})} \quad \textrm{ and } \quad t\cdot G_{(0,y; \psi)}= G_{(0,\frac{1}{\nnn_{E/F}(a)}y; \psi^{\nnn_{E/F}(ad)})}, \qquad t=\begin{pmatrix}
  a& 0 \\
 0 & d
 \end{pmatrix}.\]
For such action, there are two kinds of orbits:
\[ (i)\  \Orbit \{F_{(0,0; \phi)}\} \quad \textrm{ and }  \quad(ii)\  \Orbit \{G_{(0,1; \phi)}\}, \qquad\textrm{  for the fixed } \phi\in X_F.\]
The stabilizer of the representative element in  each orbit has the following form:
\begin{itemize}
\item[(i)] $\stab_T(F_{(0,0; \phi)})=\{h(a,d)\in T| \nnn_{K/F}(ad)=1\}$;
\item[(ii)]   $\stab_T(G_{(0,1; \phi)})=\{h(a,d)\in T| \nnn_{K/F}(a)=\nnn_{K/F}(d)=1\}$.
\end{itemize}
Now we present one   statement about the principal series components of the representation $\pi$:
\begin{proposition}\label{demp3pro4}
Let $\alpha, \beta\in \Irr(K^{\times})$.
\begin{itemize}
\item[(1)]  If $\alpha=\chi_1\circ \nnn_{K/F}, \beta=\chi_2 \circ \nnn_{K/F}$ for some characters $\chi_1\neq \chi_2 \in \Irr(F^{\times})$, then $\dim_{\C}\Hom_G( V, \Ind_B^G(\alpha\otimes  \beta))=1$.
\item[(2)]   If $\alpha=\beta=\chi\circ \nnn_{K/F}$ for  a character $\chi\in \Irr(F^{\times})$, then $\dim_{\C} \Hom_G\big(V, \Ind_B^G (\alpha \cdot 1_B)\big)=2$.
\end{itemize}
 For the other kind of $\alpha,\beta \in \Irr(K^{\times})$, $\Hom_{G}(V, \Ind_B^G(\alpha\otimes  \beta))=0$.
\end{proposition}
\begin{proof}
 By Frobenius reciprocity, we see $\Hom_{G}(V, \Ind_B^G(\alpha\otimes  \beta)) \simeq \Hom_T(V^N, \alpha\otimes \beta)$. Let $\Phi \in \Hom_T(V^N, \alpha\otimes \beta)$.  The function $\Phi$ is completely determined by its values at the points $F_{(0,0; \phi)}$ and $ G_{(0,1; \phi)}$. The value $\Phi(F_{(0,0; \phi)})$ can be any complex number  if and only if $\alpha\otimes \beta(t)=1$ for $t=h(a,d) \in \stab_T(F_{(0,0; \phi)})$,  which is equivalent to $\alpha=\beta=\chi\circ \nnn_{K/F}$ for some character $\chi\in \Irr(F^{\times})$. Similarly the value $\Phi(G_{(0,1; \phi)})$ can be any complex number if and only if $\alpha=\chi_1\circ \nnn_{K/F}, \beta=\chi_2 \circ \nnn_{K/F}$ for two characters $\chi_1, \chi_2 \in \Irr(F^{\times})$; thus we obtain  the results.
\end{proof}
\subsection{}\label{table}
Now it  reduces to check  whether the representation $\chi\circ \nnn_{E/F} \cdot 1_G$ of $G$ is a sub-representation of $\pi$.

Let $(\alpha\cdot\pi, V_{\alpha})$ be the  representation of $\pi$ twisted by the character $\alpha=\chi\circ \nnn_{K/F} \in \Irr(G)$. Since $\Hom_G(\pi, \alpha^{-1}\!\cdot\! 1_G) \simeq (V_{\alpha})^G$,  it suffices to determine the dimension of $(V_{\alpha})^G$ for the representation $(\alpha\cdot \pi, G, V_{\alpha})$. Notice that $(V_{\alpha})^N \simeq V^N$ which is generated by  two functions $F_{(0,0; \psi)}, G_{(0,y; \psi)}$;   the action of $T$ on $(V_{\alpha})^N$ is given by the following formulas:
 \begin{itemize}
\item[(1)]  $[\alpha\cdot\pi]\big(h(a,d)\big) F_{(0,0; \psi)} = \chi \cdot \chi_q^{+}(\nnn_{K/F}(ad))F_{(0,0; \psi^{\nnn_{E/F}(ad)})}$;
\item[(2)]  $[\alpha\cdot\pi]\big(h(a,d)\big)G_{(0,y; \psi)}= \chi \cdot \chi_q^{+}(\nnn_{K/F}(ad))G_{(0,\frac{1}{\nnn_{K/F}(a)}y; \psi^{\nnn_{K/F}(ad)})}$.
 \end{itemize}
\begin{proposition}
The vector space $(V_{\alpha})^B$ is generated by two non-zero functions $A= \sum_{t\in T} \alpha\!\cdot\!\pi(t) F_{(0,0; \phi)}$ and $B=\sum_{t\in T} \alpha\!\cdot\!\pi(t) G_{(0,1; \phi)}$ for the fixed $\phi\in X_F$.
\end{proposition}
\begin{proof}
It is straightforward.
\end{proof}

Our final task for this subsection  is to consider the action of $\omega$ on the vector space $(V_{\alpha})^B$.  Observe that $[\alpha\cdot \pi](\omega)A, [\alpha\cdot\pi](\omega)B$ both belong to $(V_{\alpha})^T$.  Consider the $T$-action on the set $Y_0 \times X_F$: ( We treat the vector space $(V_{\alpha})^T$ similarly as $(V_{\alpha})^N$.)
$$T \times \big( Y_0 \times X_F\big) \longrightarrow Y_0 \times X_F; \Bigg( \begin{pmatrix}
  a& 0 \\
 0 & d
 \end{pmatrix},  \bigg( \begin{pmatrix}
  0& 0 \\
 \beta & y
 \end{pmatrix}, \psi \bigg) \Bigg) \longmapsto \Bigg(  \begin{pmatrix}
  0& 0 \\
 \frac{\nnn_{K/F}(ad)}{dd^{\sigma} a^{\sigma^2}} \beta & \nnn_{K/F}(a) y
 \end{pmatrix}, \psi^{\nnn_{K/F}(ad)^{-1}} \Bigg).$$
The orbits of this action are following:
\begin{itemize}
\item[ ] Orbit $\{x_{00}\}$, $x_{00}= \bigg(\begin{pmatrix}
  0& 0 \\
 0&0
 \end{pmatrix}, \phi\bigg)$ and Orbit $\{x_{10}\}$,  $x_{10}= \bigg( \begin{pmatrix}
  0& 0 \\
 1 & 0
 \end{pmatrix}, \phi\bigg)$,
\item[ ]  Orbit $\{x_{01}\}$, $x_{01}= \bigg( \begin{pmatrix}
  0& 0 \\
 0&1
 \end{pmatrix}, \phi\bigg)$ and  Orbit $\{y_k\}$, $y_k= \bigg( \begin{pmatrix}
  0& 0 \\
  1 & 1
 \end{pmatrix}, \phi^k\bigg)$
\end{itemize}
for the fixed character $\phi\in X_F$  and any $k\in F^{\times}$. By the  calculations in Appendix 2, we obtain the following table for the values of the functions $A, B, [\alpha\cdot\pi](\omega) A, [\alpha\cdot\pi](\omega) B$ at the points (1) $x_{00}$; (2) $x_{01}$; (3) $x_{10}$; (4) $y_k$.\\

\begin{tabular}{|p{2cm}|p{3cm}|p{3cm}|p{3cm}|p{3cm}|}
\hline
  \ &  $x_{00}$ &  $x_{01}$ & $x_{10}$ & $y_k$\\
\hline
A& $(q-1)(q^2+q+1)^2$ &0 &0 &0\\
\hline
B& 0 & $(q^2+q+1)^2$ &0 & $\chi\chi_q^{+}(k)\phi(-k)(q^2+q+1)^2$ \\
\hline
$\alpha\cdot\pi(\omega)A$ & $q^{-2}(q-1)(q^2+q+1)^2$ & $q^{-2}(q-1)(q^2+q+1)^2$ & $q^{-2}(q-1)(q^2+q+1)^2$ & $\chi\chi_q^{+}(k)q^{-2}(q-1)(q^2+q+1)^2 $\\
\hline
$\alpha\cdot\pi(\omega)B$& -$q^{-1}(q-1)(q+1)(q^2+q+1)^2$ & $q^{-1}(q+1)(q^2+q+1)^2$ & $q^{-1}(q^2+q+1)^2$  & \  \\
\hline
\end{tabular}

\begin{corollary}
The element $qA-(q-1)B\in V_{\alpha}^B$ is $[\alpha \cdot\pi](\omega)$-invariant.
\end{corollary}
\begin{proof}
Let us consider $C=\sum_{g\in G} [\alpha\cdot \pi](g) F_{(0,0;\phi)}$. Then
$$C(x_{00})=\sum_{n\in N, b\in B}[\alpha\cdot \pi](n\omega b) F_{(0,0; \phi)} (x_{00})+ \sum_{b\in B} [\alpha \cdot \pi](b) F_{(0,0; \phi)} (x_{00})$$
$$=q^3[ \sum_{n\in N} [\alpha \cdot \pi](n) [\alpha\cdot \pi](\omega) A+ A] (x_{00})=q^3[ q^3 [\alpha \cdot \pi](\omega) A(x_{00}) + A(x_{00})]= q^3(q+1)(q-1)(q^2+q+1)^2\neq 0.$$
 As $[\alpha \cdot\pi](\omega) A\neq A$, this means that $\dim (V_{\alpha})^G =1$. So there exists two constants $a, b \in \C^{\times}$ such that $aA+b B $ is $[\alpha \cdot \pi](\omega)$-invariant. By the above diagram, we can let $a=q, b=-(q-1)$.
\end{proof}
\begin{corollary}\label{demp3cor2}
For any character $\chi \in \Irr(F^{\times})$ and  $\alpha^{-1}=\chi^{-1}\circ \nnn_{K/F}$, we have: \begin{itemize}
\item[(1)] $\dim_{\C} \Hom_{G}( V, \alpha^{-1}\cdot 1_G)=1$;
\item[(2)] $\dim_{\C} \Hom_{G}( V, \alpha^{-1}\cdot \St_G)=1$.
\end{itemize}
\end{corollary}
\begin{proof}[Proof]
1) $\Hom_{\C}(V, \alpha^{-1}\cdot 1_G) \simeq (V_{\alpha})^G,$
 which is of  dimension smaller than $2$.  As $\alpha \cdot\pi(\omega) A \neq A$ and  $\alpha \cdot \pi(\omega)(qA-(q-1)B)=qA-(q-1)B$, we know  that $\dim_{\C}(V_{\alpha})^G=1$.\\
2) It follows from the above (1) and Proposition \ref{demp3pro4}.
\end{proof}
\begin{proposition}\label{noncusp}
The non-cuspidal part of the Weil representation $\pi$ is presented as follows:
\begin{displaymath}
\pi_{non-cusp} \simeq \bigoplus_{\sigma \in \Irr_{non-cusp}(\GL_2(F))} \Bc_{K/F}(\sigma),
\end{displaymath}
where $\pi_{non-cusp}$ is the non-cuspidal part of the representation $\pi$ and $\Bc_{K/F}$ is the map of base change from $\Irr(\GL_2(F))$ to $\Irr(\GL_2(K))$.
\end{proposition}
\begin{proof}
It follows from Theorem \ref{basechangeGL2}(2), Proposition \ref{demp3pro4} and Corollary \ref{demp3cor2}.
\end{proof}
\begin{corollary}\label{thedimofcus}
The total dimension of the cuspidal part of $\pi$  equals $\frac{(q-1)q(q^3-1)}{2}$.
\end{corollary}
\begin{proof}
By Proposition \ref{noncusp}, the dimension of the  non-cuspidal part of $\pi$ equals
$$(q^3+1) \cdot \frac{(q-1)(q-2)}{2} + 1 \cdot (q-1) + q^3  \cdot (q-1) = \frac{(q-1)q(q^3+1)}{2};$$
the dimension of $\pi$ is $(q-1)q^4$, and $(q-1)q^4- \frac{(q-1)q(q^3+1)}{2}=\frac{(q-1)q(q^3-1)}{2}$.
\end{proof}
\subsection{} We continue the above discussion and  determine the cuspidal part of $\pi$.

Now let $K_1$ (resp. $F_1$) be a quadratic field extension of $K$ (resp. $F$). Assume $K_1 \supset F_1$.  Let $\rho$(resp. $\rho_1$) denote the Weil representation of $\textbf{GSp}_{W_0}(F)$(resp. $\textbf{GSp}_{W_0}(F_1)$). Denote by $\pi=\rho|_{\GL_2(K)}$ and $\pi_1=\rho_1|_{\GL_2(K_1)}$.
By Proposition \ref{liftingGSp} in Section \ref{shintanilift}, there exists a unique representation $\widetilde{\rho_1}$ of the group
$\textbf{GSp}_{W_0}(F_1)\rtimes \Gal(F_1/F) $ such that $0-\res(\widetilde{\rho_1})=\rho_1$,   and $1-\res(\widetilde{\rho_1})=\rho$.
By the result in Section \ref{Weildescent1}, there exists a morphism from $\Res_{K/F}(\textbf{GL}_2)$ to $\textbf{GSp}_{W_0}$, which induces a map $\widetilde{p_1}: \textbf{GL}_2(K_1) \rtimes \Gal(K_1/K)  \simeq \Res_{K/F}(\textbf{GL}_2)(F_1)\rtimes  \Gal(F_1/F)  \longrightarrow  \textbf{GSp}_{W_0}(F_1)\rtimes \Gal(F_1/F)$.
 Via the map $\widetilde{p_1}$, we let $\widetilde{\pi_1}=\widetilde{\rho_1}|_{ \textbf{GL}_2(K)\rtimes \Gal(K_1/K) }$. By Lemma \ref{Indresnorm}, one sees $0-\res(\widetilde{\pi_1})=\pi_1 \textrm{  and } 1-\res(\widetilde{\pi_1})=\pi.$ For a cuspidal representation $\Pi_{\Lambda}$ of $\GL_2(K)$, by Theorem \ref{basechangeGL2} we know $\Bc_{K_1/K}(\Pi_{\Lambda})=\Pi_{\Lambda, \Lambda^{q^3}}$. Let $\widetilde{\Pi_{\Lambda, \Lambda^{q^3}}}$ denote the unique representation of the group $\textbf{GL}_2(K_1)\rtimes \Gal(K_1/K) $ such that
$0-\res (\widetilde{\Pi_{\Lambda, \Lambda^{q^3}}})= \Pi_{\Lambda, \Lambda^{q^3}} \textrm{  and  }1-\res (\widetilde{\Pi_{\Lambda, \Lambda^{q^3}}})=\Pi_{\Lambda}$.
By Proposition \ref{demp3pro4},  $\langle\pi_1,\Pi_{\Lambda, \Lambda^{q^3}} \rangle_{\GL_2(K_1)}=1$ for  $\Lambda=\lambda\circ \nnn_{K_1/F_1}$, where $\lambda$ is a regular character of $F_1^{\times}$. By Lemma \ref{basenorm}(i), we have
$$\langle\tr \widetilde{ \pi_1},\tr\widetilde{\Pi_{\Lambda, \Lambda^{q^3}}}\rangle_{  \textbf{GL}_2(K_1)\rtimes \Gal(K_1/K)}$$
$$= \frac{1}{|\GL_2(K_1)\rtimes \Gal(K/F)|} \Bigg( \sum_{g\in \GL_2(K_1)}  \tr\widetilde{\pi_1}\Big( (1, g)\Big) \overline{\tr \widetilde{\Pi_{\Lambda, \Lambda^{q^3}}}\Big( (1,g)\Big)} + \sum_{g\in \GL_2(K_1)}   \tr\widetilde{\pi_1}\Big( (\sigma, g)\Big) \overline{\tr\widetilde{\Pi_{\Lambda, \Lambda^{q^3}}}\Big( (\sigma,g)\Big) }\Bigg)$$
$$=\frac{|\GL_2(K_1)|}{ | \GL_2(K_1)\rtimes \Gal(K_1/K) |} \langle\tr \pi_1, \tr \Pi_{\Lambda, \Lambda^{q^3}}\rangle_{\GL_2(K_1)} + \frac{|\GL_2(K_1)|}{ |\GL_2(K_1) \rtimes \Gal(K_1/K) |} \langle\tr \pi, \tr\Pi_{\Lambda} \rangle_{\GL_2(K)}$$
$$=\frac{1}{2}\bigg(\langle \tr \pi_1, \tr \Pi_{\Lambda, \Lambda^{q^3}} \rangle_{\GL_2(K_1)}+\langle\tr \pi, \tr \Pi_{\Lambda} \rangle_{\GL_2(K)}\bigg)=\frac{1}{2}\bigg(1+\langle \tr \pi, \tr\Pi_{\Lambda} \rangle_{\GL_2(K)}\bigg)$$
for $\Lambda=\lambda\circ \nnn_{K_1/F_1}$. It follows that for such $\Lambda$, $\langle\pi,\Pi_{\Lambda} \rangle_{\GL_2(K)}\geq 1$.  By Corollary \ref{thedimofcus}, we see $\langle\pi,\Pi_{\Lambda} \rangle_{\GL_2(K)}= 1$ and it will also turn out that there are no other kind of cuspidal sub-representations of $\pi$. Finally we achieve the main theorem in this section:
\subsection{}
\begin{theorem}\label{mainth4}
The representation  $(\pi,V )$ has the following decomposition:
\begin{displaymath}
\pi\simeq \bigoplus_{\sigma \in \Irr(\GL_2(F))} \Bc_{K/F}(\sigma),
\end{displaymath}
where $\Irr(\GL_2(F))$ is the set of the classes of the irreducible representations of $\GL_2(F)$, and $Bc_{K/F}$ is the base change   from $\Irr(\GL_2(F))$ to $\Irr(\GL_2(K))$.
\end{theorem}
\begin{proof}
It follows from  Proposition \ref{noncusp}  for  non-cuspidal representations and the above discussion for  cuspidal representations.
\end{proof}
\subsection{ Appendix 1} In the following, we explain  how to get  the formulas  in Proposition \ref{demp3pro2}.

(1):
$$[\pi\big( u(b)F\big)]\bigg( \begin{pmatrix}
  0& 0 \\
 \beta & y
\end{pmatrix}, \psi\bigg)$$
$$= [\rho\bigg( i\Big(u(b)\Big)F\bigg)]\bigg( \begin{pmatrix}
  0& 0 \\
 \beta & y
\end{pmatrix}, \psi\bigg)$$
$$=[\rho\bigg( \begin{pmatrix}
  m& 0 \\
 0 & m^{\vee}
\end{pmatrix}\begin{pmatrix}
  1& m^{-1}n \\
 0 & 1
\end{pmatrix}\bigg)F]\bigg( \begin{pmatrix}
  0& 0 \\
 \beta & y
\end{pmatrix}, \psi\bigg)$$
$$=\chi_q^+(det_{X_0}m)[\rho\bigg(\begin{pmatrix}
  1& m^{-1}n \\
 0 & 1
\end{pmatrix}\bigg)F]\Bigg( m^{\vee-1} \begin{pmatrix}
  0& 0 \\
 \beta & y
\end{pmatrix}, \psi\Bigg)$$
$$=[\rho\bigg(\begin{pmatrix}
  1& m^{-1}n\\
 0 & 1
\end{pmatrix}\bigg)F]\Bigg( \begin{pmatrix}
  0& 0 \\
 \beta-b^{\sigma^2}y & y
\end{pmatrix}, \psi\Bigg)$$
$$=\psi\Bigg(\frac{1}{2} \langle m^{-1}n \begin{pmatrix}
  0& 0 \\
 \beta-b^{\sigma^2}y & y
\end{pmatrix},\quad \begin{pmatrix}
  0& 0 \\
 \beta-b^{\sigma^2}y & y
\end{pmatrix}\rangle \Bigg) F\bigg( \begin{pmatrix}
  0& 0 \\
 \beta-b^{\sigma}y & y
\end{pmatrix}, \psi\bigg)$$
$$=\psi\bigg( \frac{1}{2}\langle m^{-1}\begin{pmatrix}
  \tr_{E/F}(bb^{\sigma}\beta)-2\nnn_{K/F}(b)y& b^{\sigma}\beta^{\sigma}+ b\beta^{\sigma^2}-bb^{\sigma}y \\
 0& 0
\end{pmatrix}, \quad \begin{pmatrix}
  0& 0 \\
 \beta-b^{\sigma^2}y & y
\end{pmatrix} \rangle\bigg) F\Bigg( \begin{pmatrix}
  0& 0 \\
 \beta-b^{\sigma^2}y & y
\end{pmatrix}, \psi\Bigg) $$
$$=\psi\bigg( \frac{1}{2} \langle \begin{pmatrix}
  \nnn_{K/F}(b)y-\tr_{K/F}(bb^{\sigma}\beta)& b^{\sigma}\beta^{\sigma}+b\beta^{\sigma^2}-bb^{\sigma}y \\
 0 & 0
\end{pmatrix}, \quad \begin{pmatrix}
  0& 0 \\
 \beta-b^{\sigma^2}y & y
\end{pmatrix}\rangle \bigg)F\Bigg( \begin{pmatrix}
  0& 0 \\
 \beta-b^{\sigma^2}y & y
\end{pmatrix}, \psi\Bigg)$$
$$=\psi\bigg( \tr_{K/F}(bb^{\sigma}\beta y)-\nnn_{K/F}(b)y^{2}-\tr_{K/F}(b\beta\beta^{\sigma^2})\bigg) F\Bigg( \begin{pmatrix}
  0& 0 \\
 \beta-b^{\sigma^2}y & y
\end{pmatrix}, \psi\Bigg).$$
(2):
$$[\pi(h(a,d))F]\bigg( \begin{pmatrix}
  0& 0 \\
 \beta & y
\end{pmatrix}, \psi\bigg)$$
$$= [\rho\bigg( i\Big(h(a,d)\Big)\bigg)F]\Bigg( \begin{pmatrix}
  0& 0 \\
 \beta & y
\end{pmatrix}, \psi\Bigg)$$
$$=[\rho\bigg( \begin{pmatrix}
  m& 0 \\
 0 & n
\end{pmatrix} \begin{pmatrix}
  1& 0 \\
 0& \nnn_{K/F}(ad)^{-1}
\end{pmatrix}  \begin{pmatrix}
  1& 0 \\
 0& \nnn_{K/F}(ad)
\end{pmatrix}\bigg)F] \Bigg(  \begin{pmatrix}  0& 0 \\ \beta& y \end{pmatrix}, \psi\Bigg)$$
$$=\chi_q^+(det_{X_0}(m))[\rho\bigg(  \begin{pmatrix}
  1& 0 \\
 0& \nnn_{K/F}(ad)
\end{pmatrix}\bigg) F]\Bigg( n^{-1}  \begin{pmatrix}
  0& 0 \\
 \nnn_{K/F}(ad)\beta& \nnn_{K/F}(ad)y
\end{pmatrix}, \psi\Bigg)$$
$$=\chi_q^+(\nnn_{K/F}(ad))[\rho\bigg(  \begin{pmatrix}
  1& 0 \\
 0& \nnn_{K/F}(ad)
\end{pmatrix}\bigg) F]\Bigg(  \begin{pmatrix}
  0& 0 \\
 \frac{\nnn_{K/F}(ad)}{dd^{\sigma}a^{\sigma^2}}\beta& \nnn_{K/F}(a)y
\end{pmatrix}, \psi\Bigg)$$
$$=\chi_q^+(\nnn_{K/F}(ad))F\Bigg(  \begin{pmatrix}
  0& 0 \\
 \frac{\nnn_{K/F}(ad)}{dd^{\sigma}a^{\sigma^2}}\beta& \nnn_{K/F}(a)y
\end{pmatrix}, \psi^{\nnn_{K/F}(ad)^{-1}}\Bigg).$$
(3): Assume $K=F(\xi)$. For a matrix $X$, we denote its transpose by ${}^t X$. Choose a basis $\mathcal{A}=\{ m_0=e_1\otimes e_1 \otimes e_1, m_1=\xi e_1\otimes e_1 \otimes e_2 + \xi^{\sigma} e_2 \otimes e_1 \otimes e_1 + \xi^{\sigma^2} e_1 \otimes e_2 \otimes e_1, m_2=\xi^{\sigma} e_1\otimes e_1 \otimes e_2 + \xi^{\sigma^2} e_2 \otimes e_1 \otimes e_1 + \xi e_1 \otimes e_2 \otimes e_1, m_3=\xi^{\sigma^2} e_1\otimes e_1 \otimes e_2 + \xi e_2 \otimes e_1 \otimes e_1 + \xi^{\sigma} e_1 \otimes e_2 \otimes e_1; n_0=-e_2\otimes e_2 \otimes e_2, n_1=\xi e_2\otimes e_2 \otimes e_1+ \xi^{\sigma} e_1 \otimes e_2 \otimes e_2 + \xi^{\sigma^2} e_2 \otimes e_1 \otimes e_2,
n_2=\xi^{\sigma} e_2\otimes e_2 \otimes e_1+ \xi^{\sigma^2} e_1 \otimes e_2 \otimes e_2 + \xi e_2 \otimes e_1 \otimes e_2,  n_3=\xi^{\sigma^2} e_2\otimes e_2 \otimes e_1+ \xi e_1 \otimes e_2 \otimes e_2 + \xi^{\sigma} e_2 \otimes e_1 \otimes e_2\}$ in $W_0$. Then by Corollary \ref{demp3cor1} (3), we know $i(\omega)(m_i)=n_i$ and $i(\omega)(n_i)=-m_i$ for $0 \leq i \leq 3.$ By calculation,  we obtain
\begin{displaymath}
\left( \begin{array}{cccccccc}
\langle m_0, m_0\rangle   &\ldots    &\langle m_0, m_3\rangle   &\langle m_0, n_0 \rangle  &\ldots    &\langle m_0, n_3 \rangle \\
\vdots                    &\ddots    &  \vdots                  &  \vdots                  &\ddots    & \vdots        \\
\langle m_3, m_0\rangle   &\ldots    &\langle m_3, m_3\rangle   &\langle m_3, n_0 \rangle  &\ldots    &\langle m_3, n_3 \rangle \\
\langle n_0, m_0\rangle   &\ldots    &\langle n_0, m_3\rangle   &\langle n_0, n_0 \rangle  &\ldots    &\langle n_0, n_3 \rangle \\
\vdots                    &\ddots    &  \vdots                  &  \vdots                  &\ddots    & \vdots        \\
\langle n_3, m_0\rangle   &\ldots    &\langle n_3, m_3\rangle   &\langle n_3, n_0 \rangle  &\ldots    &\langle n_3, n_3 \rangle \\
\end{array} \right)=
\left( \begin{array}{cccc}
0&  A\\
-A & 0
\end{array} \right)
\end{displaymath}
where  $A=\left( \begin{array}{cccccccc}
-1      & 0                            & 0                              &0                                   \\
0       & -\tr_{K/F}(\xi^2)            & -\tr_{K/F}(\xi\xi^{\sigma})    & -\tr_{K/F}(\xi\xi^{\sigma})         \\
0       & -\tr_{K/F}(\xi\xi^{\sigma})  & -\tr_{K/F}(\xi^2)              & -\tr_{K/F}(\xi\xi^{\sigma})         \\
0       & -\tr_{K/F}(\xi\xi^{\sigma})  & -\tr_{K/F}(\xi\xi^{\sigma})    & -\tr_{K/F}(\xi^2)
\end{array} \right). $\\
Suppose $A={}^t P_1 P_1$ and $(g_0, \cdots , g_3; h_0, \cdots , h_3)= (m_0, \cdots, m_3; n_0, \cdots, n_3) P$ for some $P= \left( \begin{array}{cccc}
P_1^{-1}& 0 \\
0     & P_1^{-1}
\end{array} \right)$.  Then:
\begin{displaymath}
\left( \begin{array}{cccccccc}
\langle g_0, g_0\rangle   &\ldots    &\langle g_0, g_3\rangle   &\langle g_0, h_0 \rangle  &\ldots    &\langle g_0, h_3 \rangle \\
\vdots                    &\ddots    &  \vdots                  &  \vdots                  &\ddots    & \vdots        \\
\langle g_3,  g_0\rangle   &\ldots    &\langle g_3, g_3\rangle   &\langle g_3, h_0 \rangle  &\ldots    &\langle g_3, h_3 \rangle \\
\langle h_0, g_0\rangle   &\ldots    &\langle h_0, g_3\rangle   &\langle h_0, h_0 \rangle  &\ldots    &\langle h_0, h_3 \rangle \\
\vdots                    &\ddots    &  \vdots                  &  \vdots                  &\ddots    & \vdots        \\
\langle h_3, g_0\rangle   &\ldots    &\langle h_3, g_3\rangle   &\langle h_3, h_0 \rangle  &\ldots    &\langle h_3, h_3 \rangle \\
\end{array} \right)={}^tP\left( \begin{array}{cccccccc}
\langle m_0, m_0\rangle   &\ldots    &\langle m_0, m_3\rangle   &\langle m_0, n_0 \rangle  &\ldots    &\langle m_0, n_3 \rangle \\
\vdots                    &\ddots    &  \vdots                  &  \vdots                  &\ddots    & \vdots        \\
\langle m_3, m_0\rangle   &\ldots    &\langle m_3, m_3\rangle   &\langle m_3, n_0 \rangle  &\ldots    &\langle m_3, n_3 \rangle \\
\langle n_0, m_0\rangle   &\ldots    &\langle n_0, m_3\rangle   &\langle n_0, n_0 \rangle  &\ldots    &\langle n_0, n_3 \rangle \\
\vdots                    &\ddots    &  \vdots                  &  \vdots                  &\ddots    & \vdots        \\
\langle n_3, m_0\rangle   &\ldots    &\langle n_3, m_3\rangle   &\langle n_3, n_0 \rangle  &\ldots    &\langle n_3, n_3 \rangle \\
\end{array} \right)P
\end{displaymath}

\[= \left( \begin{array}{cccc}
{}^tP_1^{-1}& 0\\
0& {}^tP_1^{-1}
\end{array} \right)\left( \begin{array}{cccc}
0& A \\
 -A & 0
\end{array} \right) \left( \begin{array}{cccc}
P_1^{-1}&0\\
0  & P_1^{-1}
\end{array} \right)=\left( \begin{array}{cccc}
0& I \\
-I& 0
\end{array} \right).\]
i.e. the set $\{ g_0, \cdots, g_3; h_0, \cdots, h_3\}$ is a symplectic basis of $W_0$. Moreover $i(\omega) \big( g_0, \cdots, g_3; h_0, \cdots, h_3\big)=
\big( g_0, \cdots, g_3; h_0, \cdots, h_3) P^{-1} \left(\begin{array}{cccc}
0& -I \\
I& 0
\end{array} \right)P$. And $ P^{-1} \left(\begin{array}{cccc}
0& -I \\
I& 0
\end{array} \right)P= \left(\begin{array}{cccc}
0& -I \\
I& 0
\end{array}\right)=\omega_{\GSp(W_0)}^{-1} \in \GSp(W_0)$ with respect to the symplectic basis $\{ g_0, \cdots, g_3; h_0, \cdots, h_3\}$.
Now let $\alpha=a_1 \xi + a_2 \xi^{\sigma} + a_3 \xi^{\sigma^2} $, $\beta=b_1\xi+ b_2\xi^{\sigma} + b_3 \xi^{\sigma^2} $ $\in K$. Put $b=\big( b_0, \cdots, b_3\big)$ and $a=\big( a_0, \cdots, a_3\big)$. Then
$$[\pi(\omega)F]\bigg( \left(\begin{array}{cccc}
0& 0 \\
\beta& -b_0
\end{array}\right), \psi\bigg)= \rho[ i(\omega) F]\Big( (n_0, \cdots, n_3) {}^tb, \psi\Big)$$
$$=q^{-2} \sum_{(n_0, \cdots, n_3){}^ta \in Y_0} F\big( (n_0, \cdots, n_3){}^ta, \psi\big) \psi\Big( \langle (n_0, \cdots, n_3){}^ta, \omega_{\GSp(W_0)}\big( (n_0, \cdots, n_3){}^tb\big)\rangle \Big)$$
$$=q^{-2} \sum_{(n_0, \cdots, n_3){}^ta \in Y_0} F\big( (n_0, \cdots, n_3) {}^ta, \psi\big) \psi\Big( \langle (n_0, \cdots, n_3) {}^ta, i(\omega^{-1})[(n_0, \cdots, n_3){}^tb]\rangle\Big)$$
$$= q^{-2} \sum_{(n_0, \cdots, n_3) {}^ta \in Y_0} F\Big( -a_0 e_2 \otimes e_2 \otimes e_2 + \alpha e_2 \otimes e_2 \otimes e_1  +  \alpha^{\sigma} e_1 \otimes e_2 \otimes e_2 + \alpha^{\sigma^2} e_2 \otimes e_1 \otimes e_2, \psi\Big) $$
$$\psi(\langle -a_0e_2 \otimes e_2 \otimes e_2 + \alpha e_2 \otimes e_2 \otimes e_1 + \alpha^{\sigma} e_1 \otimes e_2 \otimes e_2 + \alpha^{\sigma^2} e_2 \otimes e_1 \otimes e_2,$$
$$b_0 e_1 \otimes e_1 \otimes e_1 + \beta e_1 \otimes e_1 \otimes e_2 + \beta^{\sigma} e_2 \otimes e_1 \otimes e_1 + \beta^{\sigma^2} e_1 \otimes e_2 \otimes e_1\rangle )$$
$$=q^{-2} \sum_{a_0 \in F, \alpha \in K} F(-a_0 e_2 \otimes e_2 \otimes e_2 + \alpha e_2 \otimes e_2 \otimes e_1  +  \alpha^{\sigma} e_1 \otimes e_2 \otimes e_2 + \alpha^{\sigma^2} e_2 \otimes e_1 \otimes e_2, \psi) \psi (a_0 b_0 + \tr_{K/F}( \alpha \beta))$$
$$=q^{-2}  \sum_{a_0 \in F, \alpha \in K}  F\bigg(\begin{pmatrix}
  0& 0 \\
 \alpha& -a_0
\end{pmatrix},\psi\bigg)\psi (a_0 b_0 + \tr_{K/F}( \alpha \beta)).$$
Finally, we obtain
$$[\pi(\omega) F]\bigg(\begin{pmatrix}
  0& 0 \\
 \beta & y
\end{pmatrix},\psi\bigg)= q^{-2} \sum_{
 \beta' \in K,  y'\in F} F\bigg(\begin{pmatrix}
  0& 0 \\
 \beta' & y'
\end{pmatrix},\psi\bigg)\psi(yy'+ \tr_{K/F}(\beta\beta')).$$
\subsection{ Appendix 2} We put the calculations for the table in Section \ref{table} in this appendix.
From the definition, we see:
 $$A \bigg( \begin{pmatrix}
 0& 0\\
  \beta & y\end{pmatrix}, \phi^k\bigg)= \sum_{t\in T} [\alpha\cdot \pi](t)F_{(0,0; \phi)} \bigg( \begin{pmatrix}
 0& 0\\
  \beta & y\end{pmatrix}, \phi^k\bigg)$$
 $$= \chi\chi_q^+(\nnn_{K/F}(ad)) \sum_{a, d\in K^{\times}} F_{(0,0; \phi)}  \bigg( \begin{pmatrix}
 0& 0\\
  \frac{\nnn_{K/F}(ad)}{dd^{\sigma}a^{\sigma^2}}\beta & \nnn_{K/F}(a) y\end{pmatrix}, \phi^{\nnn_{K/F}(ad)^{-1}k}\bigg)$$
 $$=  \sum_{a, d\in K^{\times}, \nnn_{K/F}(ad)=k} \chi\chi_q^+(k) F_{(0,0; \phi)} \bigg( \begin{pmatrix}
 0& 0\\
  \frac{\nnn_{K/F}(ad)}{dd^{\sigma}a^{\sigma^2}}\beta & \nnn_{K/F}(a) y\end{pmatrix}, \phi\bigg);$$
  \ \\
$$B \bigg( \begin{pmatrix}
 0& 0\\
  \beta & y\end{pmatrix}, \phi^k\bigg)=  \sum_{t\in T} [\alpha\cdot \pi](t) G_{(0,1; \phi)} \bigg( \begin{pmatrix}
 0& 0\\
  \beta & y\end{pmatrix}, \phi^k\bigg)$$
$$= \sum_{a, d\in K^{\times}} \chi\chi_q^+(\nnn_{K/F}(ad)) G_{(0,1; \phi)} \bigg( \begin{pmatrix}
 0& 0\\
  \frac{\nnn_{K/F}(ad)}{dd^{\sigma}a^{\sigma^2}}\beta & \nnn_{K/F}(a) y\end{pmatrix}, \phi^{\nnn_{K/F}(ad)^{-1}k}\bigg)$$
 $$=  \sum_{a, d\in K^{\times}, \nnn_{K/F}(ad)=k}  \chi\chi_q^+(k)G_{(0,1; \phi)} \bigg( \begin{pmatrix}
 0& 0\\
  \frac{\nnn_{K/F}(ad)}{dd^{\sigma}a^{\sigma^2}}\beta & \nnn_{K/F}(a) y\end{pmatrix}, \phi\bigg);$$
\ \\
$$ \big( [\alpha \cdot \pi](\omega) A\big) \bigg( \begin{pmatrix}
 0& 0\\
  \beta & y\end{pmatrix}, \phi^k\bigg) $$
  $$= \sum_{t\in T} [\alpha \cdot \pi](t) \Big( [\alpha \cdot \pi](\omega) F_{(0,0; \phi)}\Big) \bigg( \begin{pmatrix}
  0& 0\\
  \beta & y\end{pmatrix}, \phi^k\bigg)$$
  $$= \sum_{a,d \in K^{\times}} \chi \chi_q^+ (\nnn_{K/F}(ad) ) \big( [\alpha \cdot \pi](\omega) F_{(0,0; \phi)}\big) \bigg(\begin{pmatrix}
  0& 0\\
  \frac{\nnn_{K/F}(ad)}{dd^{\sigma}a^{\sigma^2}}  \beta &\nnn_{K/F}(a) y\end{pmatrix}, \phi^{k\nnn_{K/F}(ad)^{-1}}\bigg)$$
  $$= q^{-2} \sum_{a,d \in K^{\times}}\chi\chi_q^+(\nnn_{K/F}(ad)) \sum_{ \begin{pmatrix}
  0& 0\\
   \beta' & y'\end{pmatrix} \in Y_0} F_{(0,0; \phi)} \bigg( \begin{pmatrix}
  0& 0\\
   \beta' & y'\end{pmatrix}, \phi^{k\nnn_{K/F}(ad)^{-1}}\bigg) \phi^{k\nnn_{K/F}(ad)^{-1}} \Big( \nnn_{K/F}(a) y y' + \tr_{K/F}(\frac{\nnn_{K/F}(ad)}{dd^{\sigma}a^{\sigma^2}} \beta \beta')\Big)$$
  $$= q^{-2} \sum_{a,d\in K^{\times}, \nnn_{K/F}(ad)=k} \chi \chi_q^+(k)=q^{-2}\chi \chi_q^+(k) (q^3-1)(q^2+q+1);$$
  \ \\
 Notice: $G_{(0,1; \phi)}(\begin{pmatrix}
  0& 0\\
   \beta & 1\end{pmatrix}, \phi)=\phi(-\nnn_{K/F}(\beta))$ by the formula (\ref{demp3eq1}).

$$(\alpha \cdot \pi(\omega) B) \bigg( \begin{pmatrix}
  0& 0\\
   \beta & y\end{pmatrix}, \phi^k\bigg) =\sum_{t\in T} \alpha\cdot \pi(t) [\alpha\cdot \pi(\omega)] G_{(0,1; \phi)} ( \begin{pmatrix}
  0& 0\\
   \beta & y\end{pmatrix}, \phi^k)$$
$$= \sum_{a,d\in K^{\times}} \chi \chi_q^+(\nnn_{K/F}(ad)) [\alpha \cdot \pi](\omega) G_{(0,1;\phi)} \bigg(\begin{pmatrix}
  0& 0\\
  \frac{\nnn_{K/F}(ad)}{dd^{\sigma}a^{\sigma^2}} \beta & \nnn_{K/F}(a) y\end{pmatrix}, \phi^{k\nnn_{K/F}(ad)^{-1}}\bigg)$$
$$= q^{-2} \sum_{a, d\in K^{\times}} \chi \chi_q^+(\nnn_{K/F}(ad)) \sum_{\begin{pmatrix}
  0& 0\\
   \beta' & y'\end{pmatrix} \in Y_0} G_{(0,1; \phi)} \bigg( \begin{pmatrix}
  0& 0\\
   \beta' & y'\end{pmatrix}, \phi^{k \nnn_{K/F}(ad)^{-1}}\bigg) \phi^{k\nnn_{K/F}(ad)^{-1}} \Big(\nnn_{K/F}(a)yy' + \tr_{K/F}(\frac{\nnn_{K/F}(ad)}{dd^{\sigma}a^{\sigma^2}} \beta \beta')\Big)$$
$$= q^{-2} \sum_{a,d \in K^{\times}, \nnn_{K/F}(ad)=k} \chi \chi_q^+(k) \sum_{\beta'\in K} \phi(-\nnn_{K/F}(\beta')) \phi\bigg(\nnn_{K/F}(a)y + \tr_{K/F}(aa^{\sigma} d^{\sigma^2} \beta \beta')\bigg);$$
\ \\
(1) $$A(x_{00})=\sum_{a, d \in K^{\times}, \nnn_{K/F}(ad)=1} F_{(0,0, \phi)} (\begin{pmatrix}
  0& 0\\
   0&0\end{pmatrix}, \phi)$$
$$= \sum_{a, d \in K^{\times}, \nnn_{K/F}(ad)=1} 1= (q^3-1)(q^2+q+1).$$
$$A(x_{10})=A(x_{01})=A(y_k)=0.$$
(2) $$B(x_{00})=\sum_{a, d\in K^{\times}, \nnn_{K/F}(ad)=1}  G_{(0,1;\phi)} (\begin{pmatrix}
  0& 0\\
   0&0\end{pmatrix}, \phi)=0=B(x_{10}).$$
$$B(x_{01})=\sum_{a,d \in K^{\times}, \nnn_{K/F}(ad)=1} G_{(0,1; \phi)}\bigg( \begin{pmatrix}
  0& 0\\
   0&\nnn_{K/F}(a)\end{pmatrix}, \phi\bigg)$$
$$= \sum_{a,d\in K^{\times}, \nnn_{K/F}(a)=\nnn_{K/F}(d)=1} 1=(q^2+q +1)^2.$$
$$B(y_k)= \sum_{t\in T} [\alpha \cdot \pi](t) G_{(0,1, \phi)}(y_k)$$
$$= \sum_{a,d\in K^{\times}, \nnn_{K/F}(ad)=k } \chi \chi_q^+ (k) G_{(0,1;\phi)} \bigg( \begin{pmatrix}
  0& 0\\
  \frac{\nnn_{K/F}(ad)}{dd^{\sigma} a^{\sigma^2}}& \nnn_{K/F}(a)\end{pmatrix}, \phi\bigg)$$
$$= \sum_{a,d\in K^{\times}, \nnn_{K/F}(a)=1, \nnn_{K/F}(d)=k} \chi\chi_q^+(k) G_{(0,1; \phi)} \bigg(  \begin{pmatrix}
  0& 0\\
   aa^{\sigma} d^{\sigma^2} &1\end{pmatrix}, \phi\bigg)$$
$$\stackrel{ \textrm{ the equality }(\ref{demp3eq1}) }{=} \sum_{a,d\in K^{\times}, \nnn_{K/F}(a)=1,\nnn_{K/F}(d)=k} \chi \chi_q^+(k) \phi\Big(-\nnn_{K/F}(aa^{\sigma} d^{\sigma^2})\Big)$$
$$= \sum_{a,d \in K^{\times}, \nnn_{K/F}(a)=1, \nnn_{K/F}(d)=k} \chi \chi_q^+(k) \phi(-k)$$
$$= \phi(-k) \chi \chi_q^+(k)(q^2 + q+ 1)^2.$$
(3) $$ [\alpha \cdot \pi](\omega) A (x_{00})= [\alpha \cdot \pi](\omega) A(X_{10})= [\alpha\cdot \pi](\omega) A(x_{01})=q^{-2} (q^3-1)(q^2+ q+1).$$
$$ [\alpha \cdot \pi](\omega)A(y_k)=\chi\chi_q^+(k) q^{-2} (q^3-1)(q^2+q+1)=\chi\chi_q^+(k) q^{-2}(q-1)(q^2+q+1)^2.$$
(4) $$\alpha \cdot \pi(\omega) B(x_{00})=q^{-2}\sum_{a,d \in K^{\times}, \nnn_{K/F}(ad)=1, \beta'\in K} \phi(-\nnn_{K/F}(\beta'))=q^{-2}(q^3-1) (q^2+q+1) (-q^2-q)=-q^{-1}(q-1)(q+1)(q^2+q+1)^2.$$
(Since $\sum_{\beta' \neq 0} \phi(-\nnn_{K/F}(\beta')) + q^2+q +1 =0$, we have $\sum_{\beta'\in K} \phi(-\nnn_{K/F}(\beta')) =-q^2-q$).
$$ \alpha \cdot \pi(\omega) B(x_{10}) = q^{-2} \sum_{a,d \in K^{\times}, \nnn_{K/F}(ad)=1} \sum_{\beta'\in K} \phi(-\nnn_{K/F}(\beta')) \phi( \tr_{K/F}( \frac{1}{dd^{\sigma}a^{\sigma^2}} \beta')$$
$$= q^{-2} \sum_{a, d \in K^{\times}, \nnn_{K/F}(ad)=1, \beta' \in K} \phi(-\nnn_{K/F}(dd^{\sigma} a^{\sigma^2} \beta')) \phi( \tr_{K/F}(\beta'))$$
$$=q^{-2} \sum_{\beta' \in K} \sum_{ a, d\in K^{\times}, \nnn_{K/F}(ad)=1} \phi(-\nnn_{K/F}(d)\nnn_{K/F}(\beta')) \phi(\tr_{K/F}(\beta')) \cdots (\star)$$
(i) If $\beta'=0$, $\sum_{a, d\in K^{\times}, \nnn_{K/F}(ad)=1} \phi(-\nnn_{K/F}(d) \nnn_{K/F}(\beta')+ \tr_{K/F}(\beta'))=\sum_{a,d \in K^{\times} , \nnn_{K/F}(ad)=1} 1= (q^3-1) (q^2+q+1)$;\\
(ii) If $\beta'\neq 0$, $$\sum_{a, d\in K^{\times}, \nnn_{K/F}(ad)=1}  \phi\big(-\nnn_{K/F}(d)\nnn_{K/F}(\beta')\big)$$
$$= \sum_{l\in K^{\times}, \nnn_{K/F}(l)=1} \sum_{d\in K^{\times}} \phi(-\nnn_{K/F}(d) \nnn_{K/F}(\beta'))$$
$$= (q^2+q+ 1) (-q^2-q-1) = -(q^2+q+1)^2,$$
so $$(\star)= q^{-2}[ (q^3-1) (q^2+q+1) + \sum_{\beta' \in K^{\times}} - (q^2+q+1)^2 \phi(\tr_{K/F}(\beta'))]$$
$$= q^{-2} [ (q^2+q+1)(q^3-1) + (q^2+q +1)^2]= q^{-2}(q^2+q+1)^2q =q^{-1}(q^2+q+1)^2.$$
\ \\
$$\alpha \cdot \pi(\omega) B(x_{01})=q^{-2} \sum_{a, d\in K^{\times}, \nnn_{K/F}(ad)=1} \sum_{\beta' \in K}\phi(-\nnn_{K/F}(\beta')) \phi(\nnn_{K/F}(a))$$
$$= q^{-2} \sum_{a,d \in K^{\times}, \nnn_{K/F}(ad)=1} \phi(\nnn_{K/F}(a)) (-q^2-q)$$
$$=q^{-2} (q^2+ q+1) (-q^2-q-1)(-q^2-q)=q^{-1}(q+1)(q^2+q+1)^2.$$

\end{document}